\documentclass[a4paper,twoside,10pt, reqno]{article}
\usepackage[a4paper,left=2.5cm,right=2.5cm, top=2.5cm, bottom=2.5cm]{geometry}
\usepackage[latin1]{inputenc}
\usepackage{cuted}

\usepackage{amsthm}

\usepackage[normalem]{ulem}
\usepackage{amssymb}
\usepackage{mathrsfs}
\usepackage{orcidlink}
\usepackage{mathtools}
\usepackage{accents}
\usepackage{multirow}
\usepackage{graphicx}
\usepackage{algorithm2e}
\usepackage{amscd,amsmath,amssymb,mathrsfs,bbm,listings}
\usepackage{cancel}
\usepackage{epstopdf}
\allowdisplaybreaks
\usepackage{amsmath}
\usepackage{footmisc}
\usepackage{amsthm}
\usepackage{amssymb}
\usepackage{stmaryrd}
\SetSymbolFont{stmry}{bold}{U}{stmry}{m}{n}
\usepackage{float}
\usepackage{bigints}
\usepackage{cite}
\usepackage{color}
\usepackage[abs]{overpic}
\usepackage[font=footnotesize,labelfont=bf]{caption}
\usepackage{cases}
\usepackage{tikz}
\usepackage{rotating}
\usepackage{blkarray}
\usetikzlibrary{matrix,calc,arrows,cd}
\usepackage{pdfcomment}
\usepackage{soul,xcolor}
\usepackage{enumerate}
\usepackage{enumitem}
\usepackage{verbatim}
\usepackage{graphicx}
\usepackage{subcaption}
\usepackage{mwe}
\usepackage{tikz}
\usepackage{authblk}

\usepackage{marginnote}

\newtheorem{theorem}{Theorem}[section]
\newtheorem{lemma}[theorem]{Lemma}
\newtheorem{proposition}[theorem]{Proposition}

\newtheorem{remark}[theorem]{Remark}
\newtheorem{definition}[theorem]{Definition}
\newtheorem{assumption}{Assumption}
\newcommand{\eremk}{\hbox{}\hfill\rule{0.8ex}{0.8ex}}

\numberwithin{equation}{section}
\definecolor{myblue}{rgb}{0,0,0.6}   
\hypersetup{
    colorlinks,
    linktoc=section,
    citecolor=red,
    linkcolor=myblue,
    urlcolor=magenta,
    pdfborder={0 0 0}
}

\makeatletter
\newcount\my@repeat@count
\newcommand{\pdt}[1]{%
\partial_{
  \begingroup
  \my@repeat@count=\z@
  \@whilenum\my@repeat@count<#1\do{t\advance\my@repeat@count\@ne}%
  \endgroup}
}
\makeatother

\newcommand{\purple}[1]{\textcolor{purple}{#1}}

\newcommand{\Ca}{C_{\alpha}}
\newcommand{\CS}{C_S}
\newcommand{\Ctr}{C_{\mathrm{tr}}}
\newcommand{\Cp}{C_{\Pi}}

\newcommand{\Cinv}{C_{\mathrm{inv}}}
\newcommand{\hmin}{h_{\mathrm{min}}}

\newcommand{\diam}{\mathrm{diam}}
\newcommand{\Id}{\mathsf{Id}}
\newcommand\hK{h_K}
\newcommand\In{I_n}
\newcommand{\Kn}{K_n}
\newcommand\SO{\Sigma_0}
\newcommand\ST{\Sigma_T}
\newcommand\Sn{\Sigma_n}
\newcommand\Qn{Q_n}
\newcommand\tnmo{t_{n-1}}
\newcommand\tn{t_n}

\newcommand\tm{t_m}

\newcommand\Th{\mathcal{T}_h}
\newcommand\Tt{\mathcal{T}_{\tau}}
\newcommand{\dV}{\text{d}V}
\newcommand\dx{\mathrm{d}\bx}

\newcommand\ds{\matrhm{d}s}
\newcommand\dt{\mathrm{d}t}

\newcommand{\Ulm}{\mathcal{U}^{\ell, m}}
\newcommand{\Bht}{\mathcal{B}_{h, \tau}}
\newcommand{\Bhtt}{\widetilde{\mathcal{B}}_{h, \tau}}
\newcommand\uht{u_{h, \tau}}

\newcommand{\Vhp}{\mathcal{V}_{h}^p}
\newcommand{\Vtq}{\mathcal{V}_{\tau}^q}
\newcommand{\Vht}{\mathcal{V}_{h, \tau}^{p, q}}
\newcommand{\Wht}{\mathcal{W}_{h, \tau}^{p, q - 1}}
\newcommand{\wt}{w_{\tau}}

\newcommand{\zh}{z_h}
\newcommand{\vh}{v_h}
\newcommand{\wht}{w_{h, \tau}}

\newcommand\wh{w_h}

\newcommand\alphaht{\alpha_{h, \tau}}

\newcommand{\la}{\underline{\alpha}}
\newcommand{\ua}{\overline{\alpha}}

\def\Cpione{C_{\Piht}^{(1)}}
\def\Cpitwo{C_{\Piht}^{(2)}}

\newlength{\dhatheight}

\newcommand\IR{\mathbb{R}}
\newcommand\IN{\mathbb{N}}
\newcommand\IZ{\mathbb{Z}}
\newcommand\QT{Q_T}

\newcommand\bx{\boldsymbol{x}}

\newcommand\dpt{\partial_t}
\newcommand\dptt{\partial_{tt}}

\newcommand\calD{\mathcal{D}}

\newcommand{\Tnorm}[1]{|\!|\!|#1|\!|\!|}
\newcommand{\Tnormd}[1]{|\!|\!|#1|\!|\!|_{\delta}}

\newcommand{\Norm}[2]{\|#1\|_{#2}}

\newcommand{\SemiNorm}[2]{\left|#1\right|_{#2}}
  
\newcommand{\jump}[1]{\llbracket #1\rrbracket}

\newcommand{\Pp}[2]{\mathbb{P}^{#1}(#2)}

\newcommand{\Ih}{\mathcal{I}_h}

\newcommand{\Pt}{\mathcal{P}_{\tau}}
\newcommand{\Rh}{\mathcal{R}_h}
\newcommand{\Piht}{\Pi_{h\tau}}
\newcommand{\eu}{e_u}
\newcommand{\epi}{e_{\Pi}}

\def\uhtdelta{\uht^{(\delta)}}
\def\uhtzero{\uht^{(0)}}
\def\uhtone{\uht^{(1)}}
\def\uhttwo{\uht^{(2)}}

\def\ball{\mathcal{B}}

\def \Cballone{C_{\ball}^{(1)}}
\def \Cballtwo{C_{\ball}^{(2)}}

\def \Clinone{C_{\textup{lin}}^{(1)}}
\def \Clintwo{C_{\textup{lin}}^{(2)}}

\def\alphahtone{\alphaht^{(1)}}
\def\alphahttwo{\alphaht^{(2)}}

\def\Cwellp{C_{\textup{wellp}}}
\def\Cwellpone{C^{(1)}_{\textup{wellp}}}
\def\Cwellptwo{C^{(2)}_{\textup{wellp}}}

\def\buht{\bar{u}_{h,\tau}}
\def\balphaht{\bar{\alpha}_{h,\tau}}
\def\bdelta{\bar{\delta}}
\def\calT{\mathcal{T}}

\def\Linf{L^\infty(\Omega)}
\def\Ltwo{L^2(\Omega)}

\def\ds{\, \textup{d}s}
\def\mapping{\mathcal{T}}
\def\ball{\mathbb{B}}
\newcommand{\Honethree}{{H_\diamondsuit^3(\Omega)}}
\def\La{\mathfrak{L}_{\alpha}}

\def\CapproxI{C_{\textup{approx}}^{\Ih{}}}
\def\CapproxPt{C_{\textup{approx}}^{\Pt}}
\def\CstabI{C_{\textup{stab}}^{\Ih{}}}
\def\CstabRh{C_{\textup{stab}}^{\Rh{}}}
\def\CstabPt{C_{\textup{stab}}^{\Pt{}}}
\def\CapproxPt{C_{\textup{approx}}^{\Pt{}}}

\def\Hlplusone{H^{\ell+1}(\Omega)}
\def\CapproxRh{C^{\Rh}_{\textup{approx}}}

\def\Ltwod{L^2(\Omega)^d}

\title{Combined DG--CG finite element method for\\ the Westervelt equation}
\author[1,2]{Sergio G\'omez\,\orcidlink{0000-0001-9156-5135}}
\author[2]{Vanja Nikoli\'c\,\orcidlink{0009-0003-7610-2900}}
\date{}

\affil[1]{\small Department of Mathematics and Applications, University of Milano-Bicocca, Via Cozzi 55, 20125 Milan, Italy (\href{mailto:sergio.gomezmacias@unimib.it}{sergio.gomezmacias@unimib.it})}
\affil[2]{\small  {IMATI-CNR ``E. Magenes", Via Ferrata 5, 27100 Pavia, Italy}}
\affil[3]{\small Department of Mathematics,
		Radboud University,      
		Heyendaalseweg 135,    
		6525 AJ Nijmegen, The Netherlands (\href{mailto:vanja.nikolic@ru.nl}{vanja.nikolic@ru.nl})}  
\begin{document}

\maketitle

\begin{abstract}
\noindent 
We propose and analyze a space--time finite element method for Westervelt's quasilinear model of ultrasound waves in second-order formulation.
The method combines conforming finite element spatial discretizations with a discontinuous-continuous Galerkin time stepping. Its analysis is challenged by the fact that standard Galerkin testing approaches for wave problems do not allow for bounding the discrete energy at all times.
By means of redesigned energy arguments for a linearized problem combined with Banach's fixed-point argument, we show the well-posedness of the scheme, \emph{a priori} error estimates, and robustness with respect to the strong damping parameter~$\delta$.
Moreover, the scheme preserves the asymptotic preserving property of the continuous problem; more precisely, we prove that the discrete solutions corresponding to~$\delta>0$ converge, in the singular vanishing dissipation limit, to the solution of the discrete inviscid problem. We use several numerical experiments in~$(2 + 1)$ dimensions to validate our theoretical results.

\end{abstract}

\paragraph{Keywords.} Westervelt's equation, space--time method, discontinuous-continuous Galerkin time stepping, asymptotic-preserving method

\paragraph{Mathematics Subject Classification.} 65M12, 65M60, 35L72

\section{Introduction \label{sec:introduction}}
Westervelt's equation~\cite{westervelt1963parametric} is a classical model of nonlinear acoustic and ultrasonic wave propagation. It is especially attractive for simulating medical and industrial applications of ultrasound as it does not assume any directionality of the wave field. 
In this work, we develop and analyze a space--time finite element approach for its discretization in the framework of the~DG--CG scheme introduced in~\cite{Walkington:2014} that combines continuous and discontinuous Galerkin time-stepping methodologies.

We consider a space--time 
cylinder~$\QT = \Omega \times (0, T)$, where~$\Omega \subset \IR^d$ ($d \in \{1, 2, 3\}$) is an open, bounded polytopic domain with Lipschitz boundary~$\partial \Omega$, and~$T > 0$ is the final propagation time.
We define the surfaces~$\SO := \Omega \times \{0\}$ and~$\ST := \Omega \times \{T\}$. Given a medium-dependent nonlinearity parameter~$k \in \IR$, a sound diffusivity coefficient~$\delta \geq 0$, the speed of sound~$c > 0$, a source term~$f: \QT \rightarrow \IR$, and initial data~$u_0,\, u_1 : \Omega \rightarrow \IR$,
we consider the following Westervelt IBVP: find the acoustic pressure~$u : \QT \rightarrow \IR$ such that
\begin{equation}
\label{eq:Westervelt-IBVP}
\begin{cases}
\dpt((1 + k u) \dpt u) - c^2 \Delta u - \delta \Delta(\dpt u) = f  & \text{ in } \QT,\\
u = 0 & \text{ on } \partial \Omega \times (0, T),\\
u = u_0  \quad \text{ and } \quad \dpt u = u_1 & \text{ on } \SO.
\end{cases}
\end{equation}
It is known that the damping parameter plays a decisive role in the behavior of solutions to~\eqref{eq:Westervelt-IBVP}. When $\delta>0$, Westervelt's equation has a parabolic character due to the strong damping $-\delta \Delta \dpt u$, which unlocks the global well-posedness of \eqref{eq:Westervelt-IBVP} with an exponential decay of the energy of the system, as shown in~\cite{kaltenbacher2009global}. In the inviscid hyperbolic case where $\delta=0$, on the other hand, one can expect only the local existence of solutions; see~\cite{dorfler2016local}. Given that the damping parameter is often small in practical scenarios, it becomes vital to understand how the solutions behave as $\delta$ approaches zero. The question is complex due to the quasilinear nature of \eqref{eq:Westervelt-IBVP}. 
Such a singular limit analysis for \eqref{eq:Westervelt-IBVP} can be found in~\cite{kaltenbacher2022parabolic}. More recently, robustness with respect to the sound diffusivity has also become a prominent issue when considering numerical schemes for \eqref{eq:Westervelt-IBVP} and other nonlinear acoustic models; see~\cite{meliani2024mixed, gomez2024asymptotic, Nikolic:2023, Dorich_Nikolic:2024}. These works have built upon the insights into the rigorous numerical investigations of the inviscid Westervelt equation in, e.g.,~\cite{dorich2024strong, hochbruck2022error, maier2023error} and the strongly damped equation in, e.g.,~\cite{nikolic2019priori}. The second major goal of this work is to shed further light on this problem by establishing the conditions under which the developed DG--CG method is robust with respect to $\delta$ and furthermore preserves the limiting behavior of the exact solutions in the vanishing dissipation limit.

We are particularly interested in fully discrete schemes based on Galerkin-type time discretizations, as they can be easily interpreted as space--time methods.
Some of the appealing properties of space--time methods are that \vspace*{-1mm}
\begin{itemize}
\item[\emph{i)}] they allow for simultaneous high-order convergence in space and time, \vspace*{-1mm}
\item[\emph{ii)}] they can be analyzed in a discrete variational setting closer to the continuous one, and \vspace*{-1mm}
\item[\emph{iii)}] the error analysis does not rely on Taylor expansions, so they typically require weaker regularity assumptions than most standard time-stepping schemes.\vspace*{-1mm}
\end{itemize}

Galerkin-type time stepping schemes for the wave equation can be classified in two groups: those based on a first-order-in-time system (i.e., introducing an auxiliary variable~$v = \dpt u$) and those defined directly for the second-order formulation.
The first group includes the discontinuous Galerkin method by Johnson~\cite{Johnson:1993} and the continuous Petrov--Galerkin method by French and Peterson~\cite{French_Peterson:1996} (see also its more robust analysis in~\cite{Bales_Lasiecka:1994,Bales_Lasiecka:1995,Gomez:2025}).
Both methods are unconditionally stable without requiring the presence of a damping term; however, at the expense of doubling the number of degrees of freedom.
The DG time discretizations by Hulbert and Hughes in~\cite{hulbert1990space} and by French in~\cite{French:1993} belong to the second group. 
In order to obtain {well-posedness}, the former uses a least-squares stability term, whereas the latter
requires employing explicit exponential weights in the inner products. In~\cite{Antonietti_et_all:2018}, Antonietti \emph{et al.} proposed and analyzed a DG time discretization for the second-order formulation relying on the presence of a first-order damping term. 
Recently, in~\cite{Shao:2022}, the DG method from~\cite{Antonietti_et_all:2018} was applied to a class of second-order nonlinear hyperbolic equations by transforming the original problem in terms of a new variable~$w = e^{-\gamma t} u$ so as to obtain the necessary damping term in the transformed problem.

Our method is based on the combined DG--CG scheme developed for the linear wave equation by Walkington in~\cite{Walkington:2014}, which provides a discrete solution that is continuous in time, but with a possibly discontinuous first-order time derivative.
The method in~\cite{Walkington:2014} is applied directly to the second-order wave formulation and does not require the presence of a damping term to guarantee well-posedness.
Moreover, unconditional stability and optimal convergence can be shown in 
$L^{\infty}(0, T; X)$-type norms. 
Consequently, the scheme is endowed with good stability and approximation properties, and requires fewer degrees of freedom than the methods based on the Hamiltonian formulation~\cite{Johnson:1993,French_Peterson:1996}.
An additional important advantage for linear wave problems is that this method does not impose a CFL-type restriction on the time step.
Theorem~\ref{thm:linear-error-bound} below shows that this advantage extends to damped linear problems with a variable leading coefficient.

\paragraph{Main contributions.} 
Our first main contribution pertains to advancing the theoretical framework of the DG--CG finite element methodology from~\cite{Walkington:2014} to accommodate quasilinear wave models in the form of \eqref{eq:Westervelt-IBVP}. In addition, to the best of our knowledge, this work contains the first rigorous space--time finite element analysis of the Westervelt equation. Specifically, we determine the sufficient conditions for the discretization and exact solution $u$, that ensure that the discrete solution $\uht$ satisfies the following \emph{a priori} bounds:
 \begin{equation} \label{final est}
     \begin{aligned}
         \Norm{\dpt (u-\uht) }{L^{\infty}(0, T; L^2(\Omega))} 
        & {\lesssim}  \tau^m \Norm{u}{m}
		+ h^{\ell + 1} \Norm{u}{\ell}, \\
		 \Norm{\nabla (u-\uht)}{L^{\infty}(0, T; L^2(\Omega)^d)}
        & {\lesssim} \tau^m \|u\|_{m,\tau}+ h^{\ell} \|u\|_{\ell,h},
     \end{aligned}
 \end{equation}
 provided $h^{\ell + 1 - \frac{d}{2}} \lesssim \tau \lesssim h^{\frac{d}{{2(m-1)}}}$, $2 \le m \le q$, ${1} \le \ell \le p$, {and~$\ell + 1 - d/2 \geq d/(2(m-1))$}, where $p$ and $q$ denote the degrees of approximation in space and time, respectively. The norms $\|\cdot\|_{m}$, $\|\cdot\|_{\ell}$, $\|\cdot\|_{m,\ell}$, and $\|\cdot\|_{\ell, h}$ on the right-hand side of \eqref{final est} are norms in Bochner spaces, defined in \eqref{m ell norms} below. We refer to Theorem~\ref{thm: h tau conv nonlinear} for the details. \\
 \indent Secondly, our analysis guarantees that the constants in estimates \eqref{final est} do not depend on the sound diffusivity $\delta$. This fact paves the way for our second main theoretical contribution, which quantifies the $\delta$-asymptotic behavior of discrete solutions of dissipative problems as $\delta \searrow 0$. More precisely, Theorem~\ref{thm: delta conv nonlinear} reveals that, for a properly set up discretization, the discrete solutions preserve the asymptotic behavior of the exact ones in the vanishing dissipation limit as they converge to the discrete solution to the inviscid problem at a linear rate. 

\paragraph{Approach and challenges.} We approach the analysis of the DG--CG method for \eqref{eq:Westervelt-IBVP} by first considering a linearized discrete problem with a variable coefficient. After deriving \emph{a priori} bounds for the linearization, we combine them with a fixed-point argument set up on a ball centered at $u$ with an~$(h, \tau)$-dependent radius, in the spirit of, e.g.,~\cite{nikolic2019priori, makridakis1993finite, ortner2007discontinuous}. This approach yields the well-posedness of the discrete nonlinear problem and error bounds \eqref{final est} simultaneously.

However, the reasoning used in the {robustness} analysis of the semidiscrete-in-space finite element formulation of the Westervelt equation (see, for example,~\cite{Nikolic:2023}) cannot be applied to analyze the fully discrete scheme in this case. The primary issue is that standard testing techniques would provide bounds for the energy norm of the discrete solution to the linearized problem only at discrete times~$\{\tn\}_{n = 1}^N$ (see also Proposition~\ref{prop:weak-continuous-dependence} below), which is not enough to bound the discrete energy at all times. Thus, a big challenge in the  \emph{a priori} analysis here lies in developing the right testing strategy for the linearized problem that leads to the bounds in the form of \eqref{final est} and that is furthermore suitable for later closing the fixed-point argument.

The numerical analysis of the Westervelt equation must guarantee that the term $1+k \uht$ does not degenerate so that the discrete problem is still a valid wave model. This requirement translates to needing a uniform small bound on $\|k \uht\|_{L^\infty(0,T; \Linf)}$.
Another special feature of the Westervelt equation is that the second time derivative is involved in the nonlinearity. To prove the uniqueness of solutions to the discrete problem, we thus also need a suitable uniform bound on the discrete second time derivative, which is a nontrivial task in a time-discrete setting. 
The particular nonlinearities present in the model thus have significant implications for designing the testing strategy in the stability and error analysis, and for the assumptions made on the discretization parameters and polynomial degrees.

\paragraph{Notation.}
We use standard notation for~$L^p$, Sobolev, and Bochner spaces. For instance, given an open, bounded domain~$\calD \subset \IR^d$ $(d \in \{1, 2, 3\})$ with Lipschitz boundary~$\partial \calD$, $p \in [1, \infty]$, and~$m \in \IR^+$, we denote the corresponding Sobolev space by~$W_p^m(\calD)$,  and its seminorm and norm, respectively, by~$|\cdot|_{W_p^m(\calD)}$ and~$\Norm{\cdot}{W_p^m(\calD)}$.
In particular, for~$p = 2$, we use the notation~$H^s(\calD) := W_2^s(\calD)$, endowed with the seminorm~$|\cdot|_{H^s(\calD)}$ and the norm~$\Norm{\cdot}{H^s(\calD)}$. For~$s = 0$, $H^0(\calD) := L^2(\calD)$ is the space of Lebesgue square integrable functions over~$\calD$ with inner product~$(\cdot, \cdot)_{\calD}$, and~$H_0^1(\calD)$ is the closure of~$C_0^{\infty}(\calD)$ in the~$H^1(\calD)$ norm.
Given a time interval~$(a, b)$, 
and a Banach space~$(X, \Norm{\cdot}{X})$,
the corresponding Bochner space is denoted by~$W_p^s(a, b; X)$.

Given~$q \in \IN$, we denote by~$\Pp{q}{\calD}$ the space of polynomials defined on~$\calD$ of degree at most~$q$.

\paragraph{Structure of the remaining of the paper.} In Section~\ref{sec:fully-discrete}, we provide a precise description of the DG--CG method for the Westervelt equation. 
Section~\ref{sec:stab-linearized-problem} contains the stability analysis of a linear problem with a variable coefficient, which is complemented in Section~\ref{sect:a-priori-linearized} by the \emph{a priori} error analysis of
{that} problem in the energy norm. The results for the linearized problem are then combined with Banach's fixed-point theorem to arrive at the well-posedness and error bounds for the Westervelt equation in Section~\ref{Sec: error analysis nonlinear}.
In Section~\ref{sec:vanishing delta analysis}, we investigate the vanishing $\delta$ dynamics of the discrete problem and establish the convergence rates with respect to this parameter.
Finally, our theoretical findings are illustrated in Section~\ref{sec:numerical-results} with several numerical experiments, and we include some concluding remarks in Section~\ref{sec:conclusions}.

\section{Description of the method \label{sec:fully-discrete}}
Let~$\{\Th\}_{h > 0}$ be a family of shape-regular conforming simplicial meshes for the spatial domain~$\Omega$. {Let also~$\{\Tt\}_{\tau}$} be a {family of} (not necessarily uniform) partition{s} of the time interval~$(0, T)$, {where~$\tau$ denotes the maximum time step and each~$\Tt$ is} of the form~$0 =: t_0 < t_1 < \ldots < t_N := T$.

For each element~$K \in \Th$, we denote the diameter of~$K$ by~$\hK$. Moreover, for~$n = 1, \ldots, N$, we define the time interval~$I_n := (\tnmo, \tn) \in \Tt$, the time step~$\tau_n := \tn - \tnmo$, the partial cylinder~$\Qn := \Omega \times \In$, and the surface~$\Sn := \Omega \times \{\tn\}$. 

In addition, for~$n = 1, \ldots, N - 1$ and any piecewise smooth function~$v$, we define the time jump\footnote{The time jumps are sometimes defined as~$\jump{v}_n(\bx) := v(\bx, \tn^+) - v(\bx, \tn^-)$. Instead, in the context of space--time methods, the definition in~\eqref{eq:time-jumps} corresponds to setting the normal vector on the \emph{so-called} ``space-like" facets pointing to the future.}
\begin{equation}
\label{eq:time-jumps}
\jump{v}_n(\bx) := v(\bx, \tn^-) - v(\bx, \tn^+) \qquad \forall \bx \in \Omega,
\end{equation}
where
\begin{equation*}
v(\bx, \tn^-) := \lim_{\varepsilon \to 0^+} v(\bx, \tn - \varepsilon) \quad \text{ and } \quad v(\bx, \tn^+) := \lim_{\varepsilon \to 0^+} v(\bx, \tn + \varepsilon).
\end{equation*}

Let~$p, q \in \IZ^+$ be some given degrees of approximation in space and time, respectively. We define the following piecewise polynomial spaces: 
\begin{subequations}
\begin{alignat*}{2}
\Vhp & := \big\{v \in H_0^1(\Omega) \, : \, v_{|_K} \in \Pp{p}{K} \text{ for all } K \in \Th\big\}, \\
\Vtq & := \big\{v \in
C^0[0, T] \, : \, v_{|_{\In}} \in \Pp{q}{\In} \text{ for } n = 1, \ldots, N\big\},\\
\Vht & := \big\{v \in 
C^0([0, T]; H_0^1(\Omega)) \, :\, v_{|_{\Qn}} \in \Pp{q}{\In; \Vhp} \text{ for all } \In \in \Tt\big\}, \\
\Wht & := \big\{w \in L^2(0, T; H_0^1(\Omega)) \, :\, w_{|_{\Qn}} \in \Pp{q-1}{\In; \Vhp} \text{ for all } \In \in \Tt\big\}.
\end{alignat*}
\end{subequations}
Moreover, we denote by~$\Pp{q}{\Tt}$ the space of piecewise polynomials of degree at most~$q$ on each~$\In \in \Tt$.

In the spirit of the DG--CG time stepping scheme in~\cite{Walkington:2014}, the proposed space--time formulation reads: find~$\uht \in \Vht$ such that
\begin{equation}
\label{eq:space-time-formulation}
\Bht(\uht, (\wht, \zh)) = \mathcal{L}(\wht, \zh) \quad \forall (\wht, \zh) \in \Wht \times \Vhp,
\end{equation}
where
\begin{subequations}
\begin{alignat*}{3}
\nonumber
\Bht(\uht, (\wht, \zh)) & := \sum_{n = 1}^N \big(\dpt((1 + k \uht) \dpt \uht), \wht\big)_{\Qn} \\
\nonumber
& \quad - \sum_{n = 1}^{N - 1} \big((1 + k \uht(\cdot, \tn)) \jump{\dpt \uht}_n, \wht(\cdot, \tn^+) \big)_{\Omega} 
+ \big((1 + k \uht) \dpt \uht, \wht \big)_{\SO} \\
& \quad  
+ c^2 \big(\nabla \uht, \nabla \wht\big)_{\QT} 
+ \delta \big(\nabla \dpt \uht, \nabla \wht \big)_{\QT} + (\nabla \uht(\cdot, 0), \nabla \zh)_{\Omega}, \\
\mathcal{L}(\wht, \zh) & := \big(f, \wht\big)_{\QT} + \big((1 + k u_0) u_1, \wht(\cdot, 0) \big)_{\Omega} + (\nabla u_0, \nabla \zh)_{\Omega}.
\end{alignat*}
\end{subequations}
Since we are interested in the vanishing dynamics, we assume that~$\delta \in [0, \bdelta]$ for some fixed~$\bdelta>0$.

\paragraph{Energy norm.} We introduce the upwind-jump functional 
\begin{equation*}
\SemiNorm{v}{\sf J}^2 := \frac12 \big(\Norm{v}{L^2(\ST)}^2 + \sum_{n = 1}^{N - 1} \Norm{\jump{v}_n}{L^2(\Omega)}^2 + \Norm{v}{L^2(\SO)}^2 \big),
\end{equation*}
and the following energy norm in the space~$\Vht$:
\begin{equation}
\label{eq:energy-norm-delta}
\begin{split}
\Tnormd{\uht}^2 & := \Norm{\dpt \uht}{L^{\infty}(0, T; L^2(\Omega))}^2 + c^2 \Norm{\nabla \uht}{L^{\infty}(0, T; L^2(\Omega)^d)}^2 + \SemiNorm{\dpt \uht}{\sf J}^2 \\
& \quad + c^2\Norm{\nabla \uht}{L^2(\ST)^d}^2 + c^2 \Norm{\nabla \uht}{L^2(\SO)^d}^2 + \delta \Norm{\nabla \dpt \uht}{L^2(\QT)^d}^2.
\end{split}
\end{equation}

\begin{remark}[Linear case]
By setting~$\delta = 0$ and~$k = 0$ in \eqref{eq:space-time-formulation}, the proposed method reduces to the DG--CG scheme in~\cite{Walkington:2014} for the linear acoustic wave equation without damping, whereas, for~$k = 0$ and~$\delta > 0$, it corresponds to a space--time discretization of a strongly damped wave equation.
\eremk
\end{remark}

\begin{remark}[Discrete initial conditions]
\label{rem:discrete-initial-condition}
In the space--time variational formulation~\eqref{eq:space-time-formulation}, the discrete initial condition for~$\uht(\cdot, 0)$ is strongly enforced as the Ritz projection of~$u_0$,
whereas the discrete initial condition for~$\dpt \uht(\cdot, 0)$ is imposed in a weak sense.
\eremk
\end{remark}

As announced, we approach the analysis of \eqref{eq:space-time-formulation} by first considering a linearized problem with a variable coefficient, where we replace $1+k\uht$ in the leading term by $1+k\alphaht$ for a suitably chosen $\alphaht$. The estimates will later be combined with Banach's fixed-point theorem to obtain the results for~\eqref{eq:space-time-formulation}, uniformly in $\delta$. Thus, $\alphaht$ can be understood in this work as a placeholder for a fixed-point iterate.
\section{Stability analysis of a linearized problem\label{sec:stab-linearized-problem}}

In this section, we analyze the stability of the following linearized space--time formulation: given a discrete approximation~$\alphaht \in \Vht$ of~$u$, find~$\uht \in \Vht$ such that
\begin{equation}
\label{eq:linearized-space-time-formulation}
\begin{split}
\Bhtt(\uht, (\wht, \zh))  = \mathcal{L}_{\delta}(\wht, \zh) \qquad \forall \wht \in \Wht, \forall \zh \in \Vhp,
\end{split}
\end{equation}
where
\begin{alignat}{3}
\nonumber
\Bhtt(\uht, (\wht, \zh)) & := \sum_{n = 1}^N \big(\dpt((1 + k \alphaht) \dpt \uht), \wht\big)_{\Qn} \\
\nonumber
& \quad - \sum_{n = 1}^{N  - 1} \big((1 + k \alphaht(\cdot, \tn)) \jump{\dpt \uht}_n, \wht(\cdot, \tn^+) \big)_{\Omega} + \big((1 + k \alphaht) \dpt \uht, \wht \big)_{\SO} \\
\label{eq:Bhtt-def}
& \quad + c^2 \big(\nabla \uht, \nabla \wht\big)_{\QT} + \delta \big(\nabla \dpt \uht, \nabla \wht \big)_{\QT} + (\nabla \uht(\cdot, 0), \nabla \zh)_{\Omega},\\[2mm]
\nonumber
\mathcal{L}_{\delta}(\wht, \zh) & := \big(f, \wht\big)_{\QT} + \big((1 + k u_0) u_1, \wht(\cdot, 0) \big)_{\Omega} \\
\label{eq:ell-delta-def}
& \quad + \delta (\nabla \dpt \mu, \nabla \wht)_{\QT} + \sum_{n = 1}^N \big(\xi, \dpt \wht \big)_{\Qn} + (\nabla u_0, \nabla \zh)_{\Omega},
\end{alignat}
with~$\mu \in H^1(0, T; H_0^1(\Omega)^d)$ being a perturbation function that will represent the interpolation error of the damping term, and~$\xi \in L^2(\QT)$
{being} a function representing the lack of orthogonality of the projection~$\Pt$ used in the~\emph{a priori} error analysis in Section~\ref{subsect:a-priori-estimates-linearized} below due to the presence of the discrete coefficient~$\alphaht$.

Henceforth, we denote by~$\Id$ the identity operator, and by~$\Pi_{q-1}^t$ the~$L^2(0, T)$ orthogonal projection onto~$\Pp{q}{\Tt}$. In Appendix~\ref{appendix:Stab}, we recall some standard results concerning stability and approximation properties of~$\Pi_{q-1}^t$, polynomial inverse estimates, and standard trace inequalities that we heavily rely on in the analysis below.

\subsection{{Summary of results in this section}}
We assume that the linearized problem does not degenerate in the following sense.
\begin{assumption}[Nondegeneracy of the discrete coefficient]
\label{asm:nondegeneracy}
Assume that the discrete coefficient~$\alphaht \in \Vht$ does not degenerate, i.e., there exist positive constants~$\la$ and~$\ua$ independent of~$h$, $\tau$, and~$\delta$ such that  
\begin{equation}
\label{eq:nondegeneracy-assumption}
0 < 1 - |k| \la \le 1 + k \alphaht(\bx, t) \le 1 + |k| \ua \qquad \forall (\bx, t) \in \QT.
\end{equation}
\end{assumption}
Later on, in the fixed-point argument, we show that the nondegeneracy assumption is satisfied for $\uht$ as well.

As mentioned, the main difficulty in reproducing the arguments used for the semidiscrete-in-space formulation of {the} linearized Westervelt equation, e.g.,~\cite{Nikolic:2023} to analyze the linearization of the fully discrete scheme~\eqref{eq:space-time-formulation} is that standard techniques provide bounds for the energy norm of the discrete solution only at the discrete times~$\{\tn\}_{n = 1}^N$ (see also Proposition~\ref{prop:weak-continuous-dependence} below), which
is not enough to bound the energy for all times. 
To address this issue, we 
first 
{extend} the ideas from~\cite{Walkington:2014} to the present setting of the linear problem with a variable coefficient in order to obtain continuous dependence on the data controlling the energy at all times. Thus, the next theorem, which we prove in Section~\ref{sec:strong-continuous-dependence} below, is the main result of this section.

\begin{theorem}[Continuous dependence on the data]
\label{thm:strong-continuous-dependence}
Let Assumption~\ref{asm:nondegeneracy} on the discrete coefficient~$\alphaht$ hold. 
There exists a positive constant~$\Ca$ independent of~$h$, $\tau$, and~$\delta$ such that, if
\begin{equation}
\label{eq:bound-alpha-W-1-infty}
|k|\Norm{\dpt \alphaht}{L^1(0, T; L^{\infty}(\Omega))} \le \Ca,
\end{equation}
then the following stability estimate holds:
\begin{alignat}{3}
\nonumber
\Tnormd{\uht}^2 & \lesssim \Norm{f}{L^1(0, T; L^2(\Omega))}^2 +  \Norm{(1 + k u_0) u_1}{L^2(\Omega)}^2 + c^2 \Norm{\nabla u_0}{L^2(\Omega)^d}^2 \\
\label{eq:strong-continuous-dependence}
& \quad + \overline{\delta} \Norm{\nabla \dpt \mu}{L^2(\QT)^d}^2 + \Norm{\overline{\tau}^{-1} \xi}{L^1(0, T; L^2(\Omega))}^2.
\end{alignat}
\end{theorem}

\begin{remark}[Existence and uniqueness of a discrete solution to the linearized problem]
\label{rem:well-posedness-linearized-problem}
From~\eqref{eq:strong-continuous-dependence}, we can deduce that the unique discrete solution to the linearized space--time formulation~\eqref{eq:linearized-space-time-formulation} with zero data~($f = 0$, $u_0 = 0$, $u_1 = 0$, $\mu = 0$, and~$\xi = 0$) is~$\uht = 0$. 
Since~\eqref{eq:linearized-space-time-formulation} corresponds to a square system of linear equations, this implies the existence and uniqueness of a discrete solution to~\eqref{eq:linearized-space-time-formulation}.
\eremk
\end{remark}

\subsection{Auxiliary weight function}
Following~\cite{Walkington:2014} (see also~\cite{Dong_Mascotto_Wang:2024} for a~$q$-explicit analysis), we introduce a set of auxiliary linear functions that are used in \purple{the proof of} Theorem~\ref{thm:strong-continuous-dependence} to derive stability estimates in the energy norm~\eqref{eq:energy-norm-delta}. Let~$\theta$ be a strictly positive parameter that will be chosen later. For~$n = 1, \ldots, N$, we define the following linear function: 
\begin{equation}
\label{eq:polynomial-weight}
\varphi_n(t) := \theta - \lambda_n(t - \tnmo) \ \text{ with } \lambda_n := \frac{\zeta_q}{\tau_n}\ \text{ and }\ \zeta_q := \frac{1}{4(2 q + 1)} \quad \text{ for all }t \in (\tnmo, \tn),
\end{equation}
The constant~$\zeta_q$ depends only on the degree of approximation~$q$. 
From Lemma~\ref{lemma:Linfty-L2}, we deduce that
\begin{equation*}
\begin{split}
\frac{\zeta_q}{(1 + \Cinv)^2}\Norm{\wt}{L^{\infty}(\In; L^2(\Omega))}^2 & \le \lambda_n \Norm{ \wt}{L^2(\Qn)}^2 \quad \forall w_{\tau} \in \Pp{q}{\In; L^2(\Omega)}, \end{split}
\end{equation*}
which will be exploited in Section~\ref{sec:strong-continuous-dependence} to obtain~$L^\infty$-bounds in time on the approximate solution.

We assume that~$\theta > \zeta_q$. Then, the weight function~$\varphi_n$ satisfies the uniform bounds
\begin{alignat}{2}
\label{eq:uniform-bound-phi_n-1}
0 < \theta - \zeta_q
\le \varphi_n(t) \le \theta & \qquad \text{ for all } t\in [\tnmo, \tn], 
\end{alignat}
as well as
\begin{alignat*}{2}
\varphi_n'(t) = -\lambda_n = - \frac{1}{4 \tau_n(2q + 1)} & \qquad \text{ for all } t \in [\tnmo, \tn].
\end{alignat*}
Furthermore, we have the following approximation result, which relies on the properties of Legendre polynomials.
\begin{lemma} \label{lemma: weight function}
For~$n = 1, \ldots, N$ and~$K \in \Th$, the following bounds hold:
\begin{alignat}{3}
\label{eq:estimate-phi_n-1}
\Norm{(\Id - \Pi_{q - 1}^t)(\varphi_n \wht)}{L^2(\Kn)} & \le \zeta_q \Norm{\wht}{L^2(\Kn)} & \quad \forall \wht \in \Wht, \\
\nonumber
\Norm{(\Id - \Pi_{q - 1}^t)(\varphi_n \wht)(\cdot, \tnmo^+)}{L^2(\Omega)} & \le \frac{1}{4\sqrt{2 q + 1} \sqrt{\tau_n}}\Norm{\wht}{L^2(\Qn)} \\
\label{eq:estimate-phi_n-2}
& = \frac{\sqrt{\lambda_n}}{2} \Norm{\wht}{L^2(\Qn)} & \quad \forall \wht \in \Wht.
\end{alignat}
\end{lemma}
\begin{proof}
Let~$L_q(t)$ denote the Legendre polynomial of degree~$q$ on the interval~$\In$ (the explicit dependence on the interval~$\In$ is omitted). 
Since~$\varphi_n \wht \in \mathcal{W}_{h, \tau}^{p, q}$, the following identity holds for all~$(\bx, t) \in \Qn$:
\begin{alignat}{3}
\nonumber
(\Id - \Pi_{q - 1}^t)(\varphi_n \wht)(\bx, t) 
& = \frac{(2q + 1) L_q(t)}{\tau_n} \int_{\In} L_q(t) \varphi_n(t) \wht(\bx, t) \dt \\
\label{eq:Id-Pi-explicit}
& = -\frac{(2q + 1) \lambda_n L_q(t)}{\tau_n} \int_{\In} L_q(t)  (t - \tnmo) \wht(\bx, t) \dt,
\end{alignat}
where we have used the fact that~$\wht(\bx, \cdot) \in \Pp{q-1}{\In}$ for all~$\bx \in \Omega$.
Therefore, using the H\"older and the Cauchy--Schwarz inequalities, we get
\begin{alignat*}{2}
\Norm{(\Id - \Pi_{q - 1}^t)(\varphi_n \wht)}{L^2(\Kn)} 
& \le \frac{(2q + 1) \lambda_n}{\tau_n} \Big(\Norm{L_q}{L^2(\In)}^2 \int_K \Big(\int_{\In} L_q(t) (t - \tnmo) \wht(\bx, t) \dt \Big)^2 \dx \Big)^{\frac12} \\
& \le \frac{(2q + 1) \lambda_n}{\tau_n} \Norm{t - \tnmo}{L^{\infty}(\In)} \Norm{L_q}{L^2(\In)} \Big( \int_K \Norm{L_q}{L^2(\In)}^2 \Norm{\wht(\bx, \cdot)}{L^2(\In)}^2 \dx\Big)^{\frac12}\\
& \le \frac{(2q + 1) \lambda_n}{\tau_n} \Norm{t - \tnmo}{L^{\infty}(\In)} \Norm{L_q}{L^2(\In)}^2 \Norm{\wht}{L^2(\Kn)} \\
& \le
\lambda_n \tau_n \Norm{\wht}{L^2(\Kn)} \\
& = \zeta_q \Norm{\wht}{L^2(\Kn)}.
\end{alignat*}

Similarly, using the fact that~$L_q(\tnmo) = (-1)^q$, we have
\begin{alignat*}{3}
\Norm{(\Id - \Pi_{q-1}^t)(\varphi_n \wht)(\cdot, \tnmo^+)}{L^2(\Omega)} & \le \frac{(2q + 1) \lambda_n }{\tau_n} \Big(\int_{\Omega} \Big(\int_{\In} L_q(t) (t - \tnmo) \wht(\bx, t) \dt \Big)^2 \dx \Big)^{\frac12} \\
& \le \frac{(2q + 1) \lambda_n}{\tau_n} \Norm{t - \tnmo}{L^{\infty}(\In)} \Norm{L_q}{L^2(\In)} \Norm{\wht}{L^2(\Qn)} \\
& \le \frac{1}{4 \sqrt{2q + 1} \sqrt{\tau_n}} \Norm{\wht}{L^2(\Qn)} \\
& = \frac{\sqrt{\lambda_n}}{2} \Norm{\wht}{L^2(\Qn)},
\end{alignat*}
which completes the proof.
\end{proof}
In what follows, 
the projection operator~$\Pi_{q-1}^t$
is to be understood as applied pointwise in space. The constant~$\CS$ in Lemma~\ref{lemma:stability-L2-proj-Lp} depends on the index~$s$. In the analysis below, whenever Lemma~\ref{lemma:stability-L2-proj-Lp} is used for different values of~$s$, the constant~$\CS$ is to be understood as the maximum of such constants. Similar conventions are made for the constants~$\Cp$ in Lemma~\ref{lemma:polynomial-approx-time} and~$\Cinv$ in~\eqref{eq:inverse-estimate-time}.

\subsection{Weak partial bound on the discrete solution}
In this section, we prove a partial bound on the discrete solution~$\uht$ of the linearized space--time formulation~\eqref{eq:linearized-space-time-formulation} that will be used to build the ``stronger" result in Section~\ref{sec:strong-continuous-dependence} below. In the first step, we show the following lower bound.
\begin{proposition}
Let Assumption~\ref{asm:nondegeneracy} on the discrete coefficient~$\alphaht$ hold. For~$n = 1, \ldots, N$, the following bound holds:
\begin{alignat}{3}
\nonumber
\Bhtt(\uht, (\wht^{\star, n}, c^2 \uht(\cdot, 0))) & \geq \frac{(1 - |k| \la)}{2} \Big(\Norm{\dpt \uht(\cdot, \tn^-)}{L^2(\Omega)}^2 + \sum_{m = 1}^{n - 1} \Norm{\jump{\dpt \uht}_{m}}{L^2(\Omega)}^2 + \Norm{\dpt \uht}{L^2(\SO)}^2 \Big) \\
\nonumber 
& \quad + \frac{c^2}{2} \big(\Norm{\nabla \uht}{L^2(\Sn)^d}^2 + \Norm{\nabla \uht}{L^2(\SO)^d}^2 \big) + \delta \Norm{\nabla \dpt \uht}{L^2(0, \tn; L^2(\Omega)^d)}^2 \\
\label{eq:coercivity-Bhtt-whstar}
& \quad + \frac{k}{2} \int_{0}^{\tn} \int_{\Omega} \dpt \alphaht (\dpt \uht)^2 \dx \dt,
\end{alignat}
where
\begin{equation}
\label{def:wh-star-n}
\wht^{\star, n}{}_{|_{Q_m}} :=
\begin{cases}
\dpt \uht{}_{|_{Q_m}} & \text{if } m \le n, \\
0 & \text{if } m > n.
\end{cases}
\end{equation}
\end{proposition}
\begin{proof}
We use the definition of~$\Bhtt(\cdot, \cdot)$ given in~\eqref{eq:Bhtt-def} and the identity
\begin{equation*}
\dpt((1 + k \alphaht) \dpt \uht) \dpt \uht = \frac12 \dpt\big( (1 + k \alphaht) (\dpt \uht)^2 \big) + \frac{k}{2} \dpt \alphaht (\dpt \uht)^2,
\end{equation*}
to obtain 
\begin{alignat*}{3}
\Bhtt(\uht, (\wht^{\star, n}, &\ c^2 \uht(\cdot, 0)))\\
& = \sum_{m = 1}^n \big(\dpt((1 + k \alphaht) \dpt \uht), \dpt \uht\big)_{Q_m} \\
\nonumber
& \quad - \sum_{m = 1}^{n  - 1} \big((1 + k \alphaht(\cdot, \tm)) \jump{\dpt \uht}_m, \dpt \uht(\cdot, \tm^+) \big)_{\Omega} + \big((1 + k \alphaht) \dpt \uht, \dpt \uht \big)_{\SO} \\
& \quad + c^2 \int_0^{\tn} \big(\nabla \uht, \nabla \dpt \uht\big)_{\Omega} \dt  + \delta \Norm{\nabla \dpt \uht}{L^2(0, \tn; L^2(\Omega)^d)}^2 + c^2\Norm{\nabla \uht}{L^2(\SO)^d}^2\\
\nonumber
& = \frac12 \sum_{m = 1}^n \Big( \int_{Q_m} \dpt \big((1 + k \alphaht) (\dpt \uht)^2 \big) \dV + k \int_{Q_m} \dpt \alphaht (\dpt \uht)^2 \dV \Big)\\
\nonumber
& \quad - \sum_{m = 1}^{n  - 1} \big((1 + k \alphaht(\cdot, \tm)) \jump{\dpt \uht}_m, \dpt \uht(\cdot, \tm^+) \big)_{\Omega} + \big((1 + k \alphaht) \dpt \uht, \dpt \uht \big)_{\SO} \\
& \quad + \frac{c^2}{2} \Big(\Norm{\nabla \uht}{L^2(\Sn)^d}^2 + \Norm{\nabla \uht}{L^2(\SO)^d}^2 \Big)  + \delta \Norm{\nabla \dpt \uht}{L^2(0, \tn; L^2(\Omega)^d)}^2\\
\nonumber
& = \frac12 \Norm{(1 + k \alphaht(\cdot, \tn))^{\frac12} \dpt \uht(\cdot, \tn^-)}{L^2(\Omega)}^2 + \frac12\Norm{(1 + k \alphaht)^{\frac12} \dpt \uht}{L^2(\SO)}^2\\
\nonumber
& \quad + \sum_{m = 1}^{n  - 1} \int_{\Omega} (1 + k \alphaht(\cdot, \tm)) \Big(\frac12 \jump{(\dpt \uht)^2}_m - \jump{\dpt \uht}_m \dpt \uht(\cdot, \tm^+)\Big) \dx \\
& \quad + \frac{c^2}{2} \Big(\Norm{\nabla \uht}{L^2(\Sn)^d}^2 + \Norm{\nabla \uht}{L^2(\SO)^d}^2 \Big)  + \delta \Norm{\nabla \dpt \uht}{L^2(0, \tn; L^2(\Omega)^d)}^2\\
\nonumber
& \quad + \frac{k}{2} \int_0^{\tn} \int_{\Omega} \dpt \alphaht (\dpt \uht)^2 \dx \dt,
\end{alignat*}
 which, together with the flux-jump identity
\begin{equation*}
\frac12 \jump{w^2}_n - w(\cdot, \tn^+) \jump{w}_n = \frac12 \jump{w}_n^2, \qquad n = 1, \ldots, N - 1,
\end{equation*}
leads to
\begin{alignat*}{3}
\Bhtt(\uht, (\wht^{\star, n},  c^2 \uht(\cdot, 0)))
& = \frac12 \Norm{(1 + k \alphaht(\cdot, \tn))^{\frac12} \dpt \uht(\cdot, \tn^-)}{L^2(\Omega)}^2 + \frac12\Norm{(1 + k \alphaht)^{\frac12} \dpt \uht}{L^2(\SO)}^2\\
\nonumber
& \quad + \sum_{m = 1}^{n  - 1} \Norm{(1 + k \alphaht)(\cdot, \tm))^{\frac12} \jump{\dpt \uht}_m}{L^2(\Omega)}^2 \\
& \quad + \frac{c^2}{2} \Big(\Norm{\nabla \uht}{L^2(\Sn)^d}^2 + \Norm{\nabla \uht}{L^2(\SO)^d}^2 \Big)  + \delta \Norm{\nabla \dpt \uht}{L^2(0, \tn; L^2(\Omega)^d)}^2\\
\nonumber
& \quad + \frac{k}{2} \int_0^{\tn} \int_{\Omega} \dpt \alphaht (\dpt \uht)^2 \dx \dt.
\end{alignat*}
The desired result then follows by using the nondegeneracy assumption~\eqref{eq:nondegeneracy-assumption} of~$\alphaht$.
\end{proof}
\noindent Next, we employ the previous result to derive a partial bound on $\uht$.
\begin{proposition}[Weak partial bound {on} the discrete solution]
\label{prop:weak-continuous-dependence}
Let~$\delta \in [0, \overline{\delta}]$ for some fixed~$\overline{\delta} > 0$, and let Assumption~\ref{asm:nondegeneracy} on the discrete coefficient~$\alphaht$ hold. For~$n = 1, \ldots, N$, the following bound holds:
\begin{alignat}{3}
\nonumber
\frac{(1 - |k| \la)}{4} & \Big(\Norm{\dpt \uht(\cdot, \tn^-)}{L^2(\Omega)}^2 + \sum_{m = 1}^{n - 1} \Norm{\jump{\dpt \uht}_{m}}{L^2(\Omega)}^2 + \Norm{\dpt \uht}{L^2(\SO)}^2 \Big) \\
\nonumber 
& \quad + \frac{c^2}{4} \big(\Norm{\nabla \uht}{L^2(\Sn)^d}^2 + \Norm{\nabla \uht}{L^2(\SO)^d}^2 \big)  + \frac{\delta}{2}\Norm{\nabla \dpt \uht}{L^2(0, \tn; L^2(\Omega)^d)}^2 \\
\nonumber
& \leq \Big(\Norm{f}{L^1(0, \tn; L^2(\Omega))} + \Cinv \Norm{\overline{\tau}^{-1} \xi}{L^1(0, \tn; L^2(\Omega))}\Big) \Norm{\dpt \uht}{L^{\infty}(0, \tn; L^2(\Omega))} \\
\nonumber
& \quad + \frac{|k|}{2} \Norm{\dpt \alphaht}{L^1(0, \tn; L^{\infty}(\Omega))} \Norm{\dpt \uht}{L^{\infty}(0, \tn; L^2(\Omega))}^2 \\
\label{eq:weak-continuous-dependence}
& \quad + \frac{1}{1 - |k| \la}\Norm{(1 + k u_0) u_1}{L^2(\Omega)}^2  + \frac{c^2}{2} \Norm{\nabla u_0}{L^2(\Omega)^d}^2 + \frac{\overline{\delta}}{2} \Norm{\nabla \dpt \mu}{L^2(0, \tn; L^2(\Omega)^d)}^2,
\end{alignat}
where~$\overline{\tau}$ is defined as
$$\overline{\tau}_{|_{I_m}} := \tau_m \quad \text{ for } m = 1, \ldots, N.$$
\end{proposition}
\begin{proof}
Let~$\wht^{\star, n}$ be given by~\eqref{def:wh-star-n}. Using the definition in~\eqref{eq:ell-delta-def} of the linear functional~${\mathcal{L}_{\delta}}(\cdot)$, the polynomial inverse estimate~\eqref{eq:inverse-estimate-time}, and the H\"older and the Cauchy--Schwarz inequalities, we obtain the following bound:
\begin{alignat*}{3}
\nonumber
{\mathcal{L}_{\delta}}(\wht^{\star, n}, c^2 \uht(\cdot, 0)) & = \sum_{m = 1}^n \big(f, \dpt \uht\big)_{Q_m} + \big((1 + k u_0) u_1, \dpt \uht(\cdot, 0) \big)_{\Omega} \\
\nonumber
& \quad + \delta \sum_{m = 1}^n (\nabla \dpt \mu, \nabla \dpt \uht)_{Q_m} + \sum_{m = 1}^n \big(\xi, \dptt \uht \big)_{Q_m} + c^2 (\nabla u_0, \nabla  \uht(\cdot, 0))_{\Omega} \\
\nonumber
& \leq \Norm{f}{L^1(0, \tn; L^2(\Omega))} \Norm{\dpt \uht}{L^{\infty}(0, \tn; L^2(\Omega))} + \Norm{(1 + k u_0) u_1}{L^2(\Omega)} \Norm{\dpt \uht}{L^2(\SO)} \\
\nonumber
& \quad + \delta \Norm{\nabla \dpt \mu}{L^2(0, \tn; L^2(\Omega)^d)} \Norm{\nabla \dpt \uht}{L^2(0, \tn; L^2(\Omega)^d)} \\
\nonumber
& \quad + \Cinv \Big(\sum_{m = 1}^n \Norm{\tau_m^{-1} \xi}{L^1(I_m; L^2(\Omega))} \Big) \Norm{\dpt \uht}{L^{\infty}(0, \tn; L^2(\Omega))} \\
\nonumber
& \quad + c^2 \Norm{\nabla u_0}{L^2(\Omega)^d}\Norm{\nabla \uht}{L^2(\SO)^d}.
\end{alignat*}
Young's inequality can then be used to deduce
\begin{alignat}{3}
{\mathcal{L}_{\delta}}(\wht^{\star, n}, c^2 \uht(\cdot, 0)) 
\nonumber
& \le \Norm{f}{L^1(0, \tn; L^2(\Omega))} \Norm{\dpt \uht}{L^{\infty}(0, \tn; L^2(\Omega))} \\
\nonumber
& \quad + \frac{1}{1 - |k| \la} \Norm{(1 + k u_0) u_1}{L^2(\Omega)}^2 + \frac{(1 - |k| \la)}{4} \Norm{\dpt \uht}{L^2(\SO)}^2 \\
\nonumber
& \quad + \frac{\delta}{2} \Norm{\nabla \dpt \mu}{L^2(0, \tn; L^2(\Omega)^d)}^2 + \frac{\delta}{2} \Norm{\nabla \dpt \uht}{L^2(0, \tn; L^2(\Omega)^d)}^2 \\
\nonumber
& \quad + \Cinv \Norm{\overline{\tau}^{-1} \xi}{L^1(0, \tn; L^2(\Omega))} \Norm{\dpt \uht}{L^{\infty}(0, \tn; L^2(\Omega))} \\
\label{eq:bound-ell-whstar}
& \quad + c^2 \Norm{\nabla u_0}{L^2(\Omega)^d}^2 + \frac{c^2}{4} \Norm{\nabla \uht}{L^2(\SO)^d}^2.
\end{alignat}
The desired result~\eqref{eq:weak-continuous-dependence} then follows by using the definition of the linearized space--time formulation~\eqref{eq:linearized-space-time-formulation}, and bounds \eqref{eq:coercivity-Bhtt-whstar} and~\eqref{eq:bound-ell-whstar}.
\end{proof}

\begin{remark}[Limitations of Proposition~\ref{prop:weak-continuous-dependence}] The terms multiplying~$\Norm{\dpt \uht}{L^{\infty}(0, \tn; L^2(\Omega))}$ on the right-hand side of~\eqref{eq:weak-continuous-dependence} cannot be controlled, as the left-hand side contains terms only related to the energy at the discrete times~$\{\tm\}_{m = 1}^n$. For this reason, we refer to Proposition~\ref{prop:weak-continuous-dependence} as ``weak partial bound" 
on the discrete solution.
We remedy this in the next section by using a nonstandard test function in combination with Proposition~\ref{prop:weak-continuous-dependence}. 
\eremk
\end{remark}

\subsection{Strong continuous dependence on the data\label{sec:strong-continuous-dependence}}
Henceforth, we use the notation~$a \lesssim b$ to indicate the existence of a positive constant~$C$ independent of the
meshsize~$h$, the maximum time step~$\tau$, and the damping parameter~$\delta$ such that~$a \le C b$. Moreover, we write~$a \simeq b$ meaning that~$a \lesssim b$ and~$b \lesssim a$.

\begin{proof}[Proof of Theorem~\ref{thm:strong-continuous-dependence}]
We make use of the auxiliary weight function introduced in~\eqref{eq:polynomial-weight}. For~$n = 1, \ldots, N$, we consider the following test functions:
\begin{equation*}
\wht^{\bullet, n} {}_{|_{Q_m}} := {\begin{cases}
\Pi_{q - 1}^t (\varphi_n \dpt \uht{}_{|_{Q_n}}) & \text{if }
m = n,\\
0 & \text{otherwise},
\end{cases}
}\quad \text{ and } \quad 
\zh^{\bullet, n} := \begin{cases}
c^2 \theta \uht(\cdot, 0) & \text{if } n = 1,\\
0 & \text{otherwise}.
\end{cases}
\end{equation*}
For the sake of clarity, we distinguish two cases~$(n = 1)$ and~$(n > 1)$ in the proof. However, this is not a recursive argument.

\paragraph{(Case~$n = 1$).} We choose as test function~$(\wht^{\bullet, 1}, \zh^{\bullet, 1})$ in~\eqref{eq:linearized-space-time-formulation}, and obtain the identity
\begin{equation*}
\Bhtt(\uht, (\wht^{\bullet, 1}, \zh^{\bullet, 1})) = \mathcal{L}_{\delta}(\wht^{\bullet, 1}, \zh^{\bullet, 1}),
\end{equation*}
where
\begin{alignat}{2}
\nonumber
\Bhtt(\uht, (\wht^{\bullet, 1}, \zh^{\bullet, 1})) & = \big(\dpt((1 + k \alphaht) \dpt \uht), \Pi_{q-1}^t(\varphi_1 \dpt 
\uht)\big)_{Q_1}
\\
\nonumber
& \quad + \big((1 + k \alphaht(\cdot, 0)) \dpt \uht(\cdot, 0), 
\Pi_{q-1}^t(\varphi_1 \dpt \uht)(\cdot, 0) \big)_{\Omega}  \\
\nonumber
& \quad + c^2 \big(\nabla \uht, \nabla \Pi_{q-1}^t(\varphi_1 \dpt \uht)\big)_{Q_1} + c^2 \theta \Norm{\nabla \uht}{L^2(\SO)^d}^2
\\
\nonumber
& \quad + \delta \big(\nabla \dpt \uht, \nabla \Pi_{q-1}^t(\varphi_1 \dpt \uht) 
\big)_{Q_1}\\
\label{eq:Bhtt-wht-bullet-Q1}
& =: L_1 + L_2 + L_3 + L_4 + L_5,
\end{alignat}
and
\begin{alignat}{3}
\nonumber
{\mathcal{L}_{\delta}}(\wht^{\bullet, 1}, \zh^{\bullet, 1})
& = \big(f, \Pi_{q-1}^t (\varphi_1 \dpt \uht) \big)_{Q_1} + \big((1 + k u_0) u_1, 
\Pi_{q-1}^t(\varphi_{1} \dpt \uht)(\cdot, 0) \big)_{\Omega} \\
\nonumber
& \quad + \delta (\nabla \dpt \mu, \nabla \Pi_{q-1}^t(\varphi_1 \dpt \uht) )_{Q_1} + \big(\xi, \dpt \Pi_{q-1}^t (\varphi_1 \dpt \uht) \big)_{Q_1}\\
\nonumber
& \quad + c^2 \theta (\nabla u_0, \nabla \uht(\cdot, 0) )_{\Omega} \\
\label{eq:ell-wht-bullet-Q1}
& =: R_1 + R_2 + R_3 + R_4 + R_5.
\end{alignat}
We first estimate the terms~$\{L_i\}_{i = 1}^5$ on the right-hand side of~\eqref{eq:Bhtt-wht-bullet-Q1}.

\paragraph{Bound {on}~$L_1 + L_2$.} Splitting~$\Pi_{q-1}^t(\varphi \dpt \uht)$ as~$\varphi \dpt \uht - (\Id - \Pi_{q-1}^t)(\varphi \dpt \uht)$ leads to
\begin{alignat}{3}
\nonumber
L_1 + L_2 & = \big(\dpt((1 + k \alphaht) \dpt \uht), \Pi_{q-1}^t(\varphi_1 \dpt 
\uht)\big)_{Q_1}
\\
\nonumber
& \quad + \big((1 + k \alphaht(\cdot, 0)) \dpt \uht(\cdot, 0), 
\Pi_{q-1}^t(\varphi_1 \dpt \uht)(\cdot, 0) \big)_{\Omega} \\
\nonumber
& = \big(\dpt((1 + k \alphaht) \dpt \uht), \varphi_1 \dpt \uht \big)_{Q_1}
\\
\nonumber
& \quad + \theta \big((1 + k \alphaht(\cdot, 0)) \dpt \uht(\cdot, 0), 
\dpt \uht(\cdot, 0) \big)_{\Omega} \\
\nonumber
& \quad - \big(\dpt((1 + k \alphaht) \dpt \uht), (\Id - \Pi_{q-1}^t)(\varphi_1 \dpt \uht) \big)_{Q_1} 
\\
\label{eq:aux-identity-n=1-split-new}
& \quad - \big((1 + k \alphaht(\cdot, 0)) \dpt \uht (\cdot, 0), (\Id 
- \Pi_{q-1}^t) (\varphi_{1} \dpt \uht) (\cdot, 0) \big)_{\Omega},
\end{alignat}
where we have used the identity~$\varphi_1(0) = \theta$.

Since~$\varphi_1(t_1) = \theta - \zeta_q$, using the identity
\begin{equation*}
\dpt((1 + k \alphaht) \dpt \uht) \varphi \dpt \uht = \frac{1}{2} \dpt((1 + k \alphaht) \varphi (\dpt \uht)^2) - \frac12 (1 + k \alphaht) \varphi' (\dpt \uht)^2 + \frac{k \varphi}{2}\dpt \alphaht (\dpt \uht)^2,
\end{equation*}
the nondegeneracy bound {on}~$(1 + k \alphaht)$ in~\eqref{eq:nondegeneracy-assumption}, the uniform bound~\eqref{eq:uniform-bound-phi_n-1} for~$\varphi_n$, the identity~$\varphi_n' = - \lambda_n$, and the H\"older inequality the first two terms on the right-hand side of~\eqref{eq:aux-identity-n=1-split-new} can be bounded as follows:
\begin{alignat*}{3}
& \big(\dpt((1 + k \alphaht) \dpt \uht), \varphi_1 \dpt \uht \big)_{Q_1} 
\\
\nonumber
& \quad + \theta \big((1 + k \alphaht(\cdot, 0)) \dpt \uht (\cdot, 0), \dpt \uht(\cdot, 0) \big)_{\Omega} \\
& = \frac12 \int_{Q_1} \dpt((1 + k \alphaht) \varphi_1 (\dpt \uht)^2) \dV + \frac{\lambda_1}{2} \Norm{(1 + k \alphaht)^{\frac12} \dpt \uht}{L^2(Q_1)}^2 \\
& \quad + \frac{k}{2} \int_{Q_1} \varphi_1 \dpt \alphaht (\dpt \uht)^2 \dV + \theta \Norm{(1 + k \alphaht)^{\frac12} \dpt \uht}{L^2(\SO)}^2 \\
& =  \frac{(\theta - \zeta_q)}{2} \Norm{(1 + k \alphaht(\cdot, t_1))^{\frac12} \dpt \uht(\cdot, t_1^-)}{L^2(\Omega)}^2 + \frac{\theta}{2} \Norm{(1 + k \alphaht(\cdot, 0))^{\frac12} \dpt \uht(\cdot, 0)}{L^2(\Omega)}^2 \\
& \quad + \frac{\lambda_1}{2} \int_{Q_1} (1 + k \alphaht) (\dpt \uht)^2 \dV + \frac{k}{2} \int_{Q_1} \varphi_1 \dpt \alphaht (\dpt \uht)^2 \dV \\
& \geq \frac{(\theta - \zeta_q) (1 - |k| \la) }{2} \Norm{\dpt \uht(\cdot, t_1^-)}{L^2(\Omega)}^2 + \frac{\theta}{2} (1 - |k| \la) \Norm{\dpt \uht(\cdot, 0)}{L^2(\Omega)}^2 \\
& \quad + \frac{(1 - |k|\la) \lambda_1}{2}  \Norm{\dpt \uht}{L^2(Q_1)}^2 - \frac{|k| \theta}{2} \Norm{\dpt \alphaht}{L^{1}(I_1; L^{\infty}(\Omega))} \Norm{\dpt \uht}{L^{\infty}(I_1; L^2(\Omega))}^2.
\end{alignat*}

As for the third term on the right-hand side of~\eqref{eq:aux-identity-n=1-split-new}, we use the orthogonality properties of~$\Pi_{q-1}^t$, the H\"older inequality, estimate~\eqref{eq:estimate-phi_n-1}, the approximation properties of~$\Pi_0^t$ from Lemma~\ref{lemma:polynomial-approx-time}, the inverse estimate~\eqref{eq:inverse-estimate-time}, Lemma~\ref{lemma:Linfty-L2}, and the inequality~$\Norm{w}{L^2(\In)} \le \sqrt{\tau_n} \Norm{w}{L^{\infty}(\In)}$ to obtain the following bound:
\begin{alignat*}{3}
- \big( &\dpt((1 + k \alphaht) \dpt \uht), (\Id - \Pi_{q-1}^t)(\varphi_1 \dpt \uht) \big)_{Q_1} \\
& = - k \big( \dpt\alphaht \dpt \uht, (\Id - \Pi_{q-1}^t)(\varphi_1 \dpt \uht) \big)_{Q_1} - k \big( (\Id - \Pi_0^t)\alphaht \dptt \uht, (\Id - \Pi_{q-1}^t)(\varphi_1 \dpt \uht) \big)_{Q_1} \\
& \geq -|k| \Norm{\dpt \alphaht}{L^1(I_1; L^{\infty}(\Omega))} \Norm{\dpt \uht}{L^{\infty}(I_1; L^2(\Omega))} \Norm{(\Id - \Pi_{q-1}^t)(\varphi_1 \dpt \uht)}{L^{\infty}(I_1; L^2(\Omega))} \\
& \quad - |k| \Norm{(\Id - \Pi_0^t) \alphaht}{L^1(I_1; L^{\infty}(\Omega))} \Norm{\dptt \uht}{L^{\infty}(I_1; L^2(\Omega))} \Norm{(\Id - \Pi_{q-1}^t) (\varphi_1 \dpt \uht)}{L^{\infty}(I_1; L^2(\Omega))} \\
& \geq -|k| (1 + \Cinv) \tau_1^{-\frac12} \Norm{\dpt \alphaht}{L^1(I_1; L^{\infty}(\Omega))} \Norm{\dpt \uht}{L^{\infty}(I_1; L^2(\Omega))} \Norm{(\Id - \Pi_{q-1}^t)(\varphi_1 \dpt \uht)}{L^2(Q_1)} \\
& \quad - |k| \Cp \Cinv (1 + \Cinv) \tau_1^{-\frac12} \Norm{\dpt \alphaht}{L^1(I_1; L^{\infty}(\Omega))} \Norm{\dpt \uht}{L^{\infty}(I_1; L^2(\Omega))} \Norm{(\Id - \Pi_{q-1}^t) (\varphi_1 \dpt \uht)}{L^2(Q_1)} \\
& \geq -|k| \zeta_q (1 + \Cinv) \tau_1^{-\frac12} \Norm{\dpt \alphaht}{L^1(I_1; L^{\infty}(\Omega))} \Norm{\dpt \uht}{L^{\infty}(I_1; L^2(\Omega))} \Norm{\dpt \uht}{L^2(Q_1)} \\
& \quad - |k| \zeta_q \Cp \Cinv (1 + \Cinv) \tau_1^{-\frac12} \Norm{\dpt \alphaht}{L^1(I_1; L^{\infty}(\Omega))} \Norm{\dpt \uht}{L^{\infty}(I_1; L^2(\Omega))} \Norm{\dpt \uht}{L^2(Q_1)} \\
& \geq - \zeta_q |k| (1 + \Cinv) (1 + \Cp \Cinv) \Norm{\dpt \alphaht}{L^1(I_1; L^{\infty}(\Omega))} \Norm{\dpt \uht}{L^{\infty}(I_1; L^2(\Omega))}^2.
\end{alignat*}

The last term on the right-hand side of~\eqref{eq:aux-identity-n=1-split-new} can be bounded using the Cauchy--Schwarz inequality, estimate~\eqref{eq:estimate-phi_n-2}, the nondegeneracy bound {on}~$(1 + k\alphaht)$ in Assumption~\ref{asm:nondegeneracy}, the Young inequality, and the inequality~$\Norm{w}{L^2(\In)} \le \sqrt{\tau_n} \Norm{w}{L^{\infty}(\In)}$ as follows:
\begin{alignat*}{3}
- \big((1 + k \alphaht(\cdot, 0)) & \dpt \uht(\cdot, 0), (\Id 
- \Pi_{q-1}^t) (\varphi_1 \dpt \uht) (\cdot, 0) \big)_{\Omega} \\
& \geq -\frac{(1 - |k| \la) \theta}{4} \Norm{\dpt \uht}{L^2(\SO)}^2  - \frac{1}{\theta} \Big(\frac{1 + |k|\ua}{1 - |k| \la}\Big) \Norm{(\Id - \Pi_{q-1}^t)(\varphi_1 \dpt \uht)(\cdot, 0)}{L^2(\Omega)}^2 \\
& \geq -\frac{(1 - |k| \la) \theta}{4} \Norm{\dpt \uht}{L^2(\SO)}^2  - \frac{\lambda_1}{4 \theta} \Big(\frac{1 + |k|\ua}{1 - |k| \la}\Big) \Norm{\dpt \uht}{L^2(Q_1)}^2 \\
& \geq -\frac{(1 - |k| \la) \theta}{4} \Norm{\dpt \uht}{L^2(\SO)}^2  - \frac{\zeta_q}{4 \theta} \Big(\frac{1 + |k|\ua}{1 - |k| \la}\Big) \Norm{\dpt \uht}{L^{\infty}(I_1; L^2(\Omega))}^2.
\end{alignat*}

\paragraph{Bound {on}~$L_3 + L_4$.} Adding and subtracting suitable terms, we have
\begin{alignat*}{3}
\nonumber
L_3
\nonumber
& = c^2\big(\nabla \uht, \varphi_1 \nabla \dpt \uht \big)_{Q_1}
- c^2\big(\nabla \uht, (\Id - \Pi_{q-1}^t) (\varphi_1 \nabla \dpt \uht) \big)_{Q_1} \\
& =: L_3^{(a)} + L_3^{(b)}.
\end{alignat*}

Using the Leibniz rule, the uniform  bound~\eqref{eq:uniform-bound-phi_n-1} for~$\varphi_1$, {and the fact that~$\varphi_1'(t) = \lambda_1$, $\varphi_1(0) = \theta$, and~$\varphi_1(t_1) = \theta - \zeta_q$}, we get
\begin{alignat*}{3}
L_3^{(a)} + L_4
& = \frac{c^2}{2} \int_{Q_1} \varphi_1 \dpt(|\nabla \uht|^2) \dV + c^2 \theta \Norm{\nabla \uht}{L^2(\SO)^d}^2 \\
& = \frac{c^2}{2} \Big(\int_{Q_1} \dpt(\varphi_1 |\nabla \uht|^2) - \int_{Q_1} \varphi_1' |\nabla \uht|^2 \dV \Big) + c^2 \theta \Norm{\nabla \uht}{L^2(\SO)^d}^2 \\
& = \frac{c^2}{2} \Big(\int_{Q_1} \dpt(\varphi_1 |\nabla \uht|^2) + \lambda_1 \Norm{\nabla \uht}{L^2(Q_1)}^2 \Big) + c^2 \theta \Norm{\nabla \uht}{L^2(\SO)^d}^2 \\
& = \frac{c^2}{2} (\theta - \zeta_q) \Norm{\nabla \uht}{L^2(\Sigma_1)^d}^2 + \frac{c^2 \theta}{2} \Norm{\nabla \uht}{L^2(\SO)^d}^2 + \frac{c^2 \lambda_1 }{2} \Norm{\nabla \uht}{L^2(Q_1)^d}^2.
\end{alignat*}

According to identity~\eqref{eq:Id-Pi-explicit}, there exists a function~$\sigma_1(\bx)$ such that
\begin{equation*}
\uht = (\Id - \Pi_{q - 1}^t) \uht + \Pi_{q - 1}^t \uht = \sigma_1(\bx) L_q(t) + \Pi_{q - 1}^t \uht.
\end{equation*}
Moreover, we have
\begin{equation*}
\dpt \uht = \sigma_1(\bx) L_q'(t) + \dpt \Pi_{q - 1}^t \uht.
\end{equation*}
Therefore, by using the orthogonality properties of~$\Pi_{q - 1}^t$ and simple manipulations, we obtain
\begin{alignat}{3}
\nonumber
L_3^{(b)} 
& = - c^2 \big(\nabla \uht, (\Id - \Pi_{q-1}^t) (\varphi_1 \nabla \dpt \uht) \big)_{Q_1} 
= -c^2 \big((\Id - \Pi_{q - 1}^t) \nabla \uht, \varphi_1 \nabla \dpt \uht \big)_{Q_1} \\
\nonumber
& = -c^2 \big( (\Id - \Pi_{q - 1}^t) \nabla \uht, (\theta - \lambda_1 t) \nabla \sigma_1(\bx)  L_q'(t) + (\theta - \lambda_1 t) \dpt \Pi_{q - 1}^t \nabla \uht\big)_{Q_1} \\
\nonumber
& = c^2 \lambda_1 \big( (\Id - \Pi_{q - 1}^t) \nabla \uht, \nabla \sigma_1(\bx) t L_q'(t) \big)_{Q_1} \\
\nonumber
& = c^2 \lambda_1 \big( \nabla \sigma_1(\bx) L_q(t), \nabla \sigma_1(\bx) t L_q'(t) \big)_{Q_1} \\
\nonumber
&  = c^2 \lambda_1  \Norm{\nabla \sigma_1}{L^2(\Omega)^d}^2 \int_{I_1} t L_q(t) L_q'(t) \dt \\
\label{eq:term-M2}
& = \frac{c^2 q \zeta_q}{2q + 1} \Norm{\nabla \sigma_1}{L^2(\Omega)^d}^2 \geq 0,
\end{alignat}
where, in the last line, we have used the identity
\begin{equation*}
\int_{I_1} t L_q(t) L_q'(t) \dt = \frac12 \int_{I_1} t (L_q^2(t))' \dt = \frac{\tau_1 q}{2q + 1}.
\end{equation*}

\paragraph{Bound {on}~$L_5$.}
The commutativity of the orthogonal projection~$\Pi_{q-1}^t$ and the spatial gradient~$\nabla$, the fact that~$\nabla \dpt \uht(\bx, \cdot)$ belongs to~$\Pp{q-1}{\Tt}^d$ for all~$\bx \in \Omega$, and the uniform bound~\eqref{eq:uniform-bound-phi_n-1} for~$\varphi_1$ lead to
\begin{alignat*}{3}
L_5 = \delta \big(\nabla \dpt \uht, \nabla \Pi_{q-1}^t(\varphi_1 \dpt \uht) \big)_{Q_1}
& = \delta \big(\nabla \dpt \uht, \Pi_{q-1}^t\nabla (\varphi_1 \dpt \uht) \big)_{Q_1} \\
& = \delta \big(\nabla \dpt \uht, \varphi_1 \nabla \dpt \uht \big)_{Q_1} \\
& = \delta \Norm{\sqrt{\varphi_1} \nabla \dpt \uht}{L^2(Q_1)^d}^2 \\
& \geq \delta (\theta - \zeta_q) \Norm{\nabla \dpt \uht}{L^2(Q_1)^d}^2.
\end{alignat*}
We now focus on the terms~$\{R_i\}_{i = 1}^5$ on the right-hand side of~$\eqref{eq:ell-wht-bullet-Q1}$.

\paragraph{Bound {on}~$R_1$.}
We use the orthogonality properties of~$\Pi_{q - 1}^t$, the stability bound in Lemma~\ref{lemma:stability-L2-proj-Lp} for~$\Pi_{q-1}^t$, the H\"older inequality, and the uniform bound~\eqref{eq:uniform-bound-phi_n-1} for~$\varphi_1$ to obtain
\begin{alignat*}{3}
R_1 = (f, \Pi_{q - 1}^t (\varphi_1 \dpt \uht))_{Q_1} 
& \le \Norm{f}{L^1(I_1; L^2(\Omega))} \Norm{\Pi_{q - 1}^t (\varphi_1 \dpt \uht)}{L^{\infty}(I_1; L^2(\Omega))} \\
& \le \CS \theta \Norm{f}{L^1(I_1; L^2(\Omega))} \Norm{\dpt \uht}{L^{\infty}(I_1; L^2(\Omega))}.
\end{alignat*}

\paragraph{Bound {on}~$R_2$.}
The second term on the right-hand side of~\eqref{eq:ell-wht-bullet-Q1} can be bounded using the Cauchy--Schwarz inequality, the polynomial trace inequality~\eqref{eq:polynomial-trace-inequality}, the stability properties of~$\Pi_{q-1}^t$, the uniform bound in~\eqref{eq:uniform-bound-phi_n-1} for~$\varphi_1$, and~the inequality~$\Norm{w}{L^2(I_n)} \le \sqrt{\tau_n} \Norm{w}{L^{\infty}(I_n)}$ as follows:
\begin{alignat*}{3}
R_2 & = \big( (1 + k u_0) u_1, \Pi_{q - 1}^t (\varphi_1 \dpt \uht)(\cdot, 0)\big)_{\Omega} \\
& \le \Norm{(1 + k u_0) u_1}{L^2(\Omega)} \Norm{\Pi_{q - 1}^t (\varphi_1 \dpt \uht)(\cdot, 0)}{L^2(\Omega)} \\
& \le \Ctr^{\star} \theta \tau_1^{-\frac12} \Norm{(1 + k u_0) u_1}{L^2(\Omega)} \Norm{\dpt \uht}{L^2(Q_1)} \\
& \le \Ctr^{\star} \theta \Norm{(1 + k u_0) u_1}{L^2(\Omega)} \Norm{\dpt \uht}{L^{\infty}(I_1; L^2(\Omega))} \\
& \le \frac{(\Ctr^{\star} \theta)^2}{2 \gamma} \Norm{(1 
+ k u_0) u_1}{L^2(\Omega)}^2 + \frac{\gamma}{2} \Norm{\dpt \uht}{L^{\infty}(I_1; L^2(\Omega))}^2,
\end{alignat*}
for all~$\gamma > 0$.

\paragraph{Bound {on}~$R_3$.} Using the commutativity of the spatial gradient operator and~$\Pi_{q-1}^t$, the stability properties of the orthogonal projection~$\Pi_{q-1}^t$, the uniform bound in~\eqref{eq:uniform-bound-phi_n-1} for~$\varphi_n$, and the Young inequality, we get
\begin{alignat*}{3}
R_3 & = \delta (\nabla \dpt \mu, \nabla \Pi_{q-1}^t(\varphi_1 \dpt \uht) )_{Q_1} \\
&  \le \delta \theta \Norm{\nabla \dpt \mu}{L^2(Q_1)^d} \Norm{\nabla \dpt \uht}{L^2(Q_1)^d} \\
& \le \frac{\overline{\delta} \theta^2}{2(\theta - \zeta_q)} \Norm{\nabla \dpt \mu}{L^2(Q_1)^d}^2 + \frac{\delta (\theta - \zeta_q)}{2} \Norm{\nabla \dpt \uht}{L^2(Q_1)^d}^2.
\end{alignat*}

\paragraph{Bound {on}~$R_4$.} We use the H\"older inequality, the polynomial inverse estimate~\eqref{eq:inverse-estimate-time}, the uniform bound in~\eqref{eq:uniform-bound-phi_n-1} for~$\varphi_n$, and the inequality~$\Norm{w}{L^2(\In)} \le \sqrt{\tau_n} \Norm{w}{L^{\infty}(\In)}$ to obtain
\begin{alignat*}{3}
R_4 & = \big(\xi, \dpt \Pi_{q-1}^t (\varphi_1 \dpt \uht) \big)_{Q_1} \\
& \le \Norm{\xi}{L^1(I_1; L^2(\Omega))} \Norm{\dpt \Pi_{q-1}^t (\varphi_1 \dpt \uht)}{L^{\infty}(I_1; L^2(\Omega))} \\
& \le \Cinv \Norm{\tau_1^{-1} \xi}{L^1(I_1; L^2(\Omega))} \Norm{\Pi_{q-1}^t (\varphi_1 \dpt \uht)}{L^{\infty}(I_1; L^2(\Omega))} \\
& \le \Cinv \CS \Norm{\overline{\tau}^{-1}\xi}{L^1(I_1; L^2(\Omega))} \Norm{\varphi_1 \dpt \uht}{L^{\infty}(I_1; L^2(\Omega))} \\
& \le \Cinv \CS \theta \Norm{\overline{\tau}^{-1} \xi}{L^1(I_1; L^2(\Omega))} \Norm{\dpt \uht}{L^{\infty}(I_1; L^2(\Omega))}.
\end{alignat*}

\paragraph{Bound {on}~$R_5$.} By using the Cauchy--Schwarz and the Young inequalities, we have
\begin{alignat*}{3}
R_5 = c^2 \theta (\nabla u_0, \nabla \uht(\cdot, 0) )_{\Omega} \le \frac{c^2 \theta}{2} \Norm{\nabla u_0}{L^2(\Omega)^d}^2 + \frac{c^2 \theta}{2} \Norm{\nabla \uht}{L^2(\SO)^d}^2.
\end{alignat*}

\paragraph{Conclusion of case~$n = 1$.}
From Lemma~\ref{lemma:Linfty-L2}, we deduce that
\begin{equation}
\label{eq:Linfty-L2-dt-uht}
\begin{split}
\frac{\zeta_q}{(1 + \Cinv)^2}\Norm{\dpt \uht}{L^{\infty}(\In; L^2(\Omega))}^2 & \le \lambda_n \Norm{\dpt \uht}{L^2(\Qn)}^2, \\
\frac{\zeta_q}{(1 + \Cinv)^2}\Norm{\nabla \uht}{L^{\infty}(\In; L^2(\Omega)^d)}^2 & \le \lambda_n \Norm{\nabla \uht}{L^2(\Qn)^d}^2,
\end{split}
\end{equation}
which, combined with all the above estimates, yields
\begin{alignat*}{3}
& (\theta - \zeta_q) \Big( \frac{(1 - |k| \la)}{2} \Norm{\dpt \uht(\cdot, t_1^{-})}{L^2(\Omega)}^2 + \frac{c^2}{2} \Norm{\nabla \uht}{L^2(\Sigma_1)^d}^2\Big) + \frac{\theta (1 - |k| \la)}{4} \Norm{\dpt \uht}{L^2(\SO)}^2
\\
& + \frac{\zeta_q}{2} \bigg(\frac{(1 - |k| \la) }{(1 + \Cinv)^2}  - \frac{1}{2 \theta} \Big(\frac{1 + |k| \ua}{1 - |k| \la}\Big) - \frac{\gamma}{2}\bigg) \Norm{\dpt \uht}{L^{\infty}(I_1; L^2(\Omega))}^2 + \frac{\zeta_q c^2}{2(1 + \Cinv)^2} \Norm{\nabla \uht}{L^{\infty}(I_1; L^2(\Omega)^d)}^2 \\
& + \frac{\delta (\theta - \zeta_q)}{2} \Norm{\nabla \dpt \uht}{L^2(Q_1)^d}^2 \\
& \qquad \qquad \le \CS \theta \Norm{f}{L^1(I_1; L^2(\Omega))} \Norm{\dpt \uht}{L^{\infty}(I_1; L^2(\Omega))} + \frac{(\Ctr^{\star} \theta)^2}{2 \gamma} \Norm{(1 + k u_0) u_1}{L^2(\Omega)}^2  + \frac{c^2 \theta}{2} \Norm{\nabla u_0}{L^2(\Omega)^d}^2 \\
& \qquad \qquad \quad + \frac{\overline{\delta} \theta^2}{2(\theta - \zeta_q)} \Norm{\nabla \dpt \mu}{L^2(Q_1)^d}^2  + \Cinv \CS \theta \Norm{\overline{\tau}^{-1} \xi}{L^1(I_1; L^2(\Omega))} \Norm{\dpt \uht}{L^{\infty}(I_1; L^2(\Omega))} \\
& \qquad \qquad \quad + \Big(\frac{\theta}{2} + \zeta_q  (1 + \Cinv) (1 + \Cp \Cinv)\Big)|k| \Norm{\dpt \alphaht}{L^1(I_1; L^{\infty}(\Omega))} \Norm{\dpt \uht}{L^{\infty}(I_1; L^2(\Omega))}^2.
\end{alignat*}
Taking~$\theta$ large enough and~$\gamma$ small enough (independently of $h$, $\tau$, and $\delta$) leads to
\begin{alignat}{3}
\nonumber
& \frac{1}{2}  \Norm{\dpt \uht(\cdot, t_1^-)}{L^2(\Omega)}^2 + \frac{c^2}{2} \Norm{\nabla \uht}{L^2(\Sigma_1)^d}^2 + \Norm{\dpt \uht}{L^2(\SO)}^2
\\
\nonumber
& + \frac12 \Norm{\dpt \uht}{L^{\infty}(I_1; L^2(\Omega))}^2  + \frac{c^2}{2} \Norm{\nabla \uht}{L^{\infty}(I_1; L^2(\Omega)^d)}^2 + \overline{\delta} \Norm{\nabla \dpt \uht}{L^2(Q_1)}^2 \\
\nonumber
& \quad \lesssim \Big(\Norm{f}{L^1(I_1; L^2(\Omega))} +  \Norm{\overline{\tau}^{-1} \xi}{L^1(I_1; L^2(\Omega))} \Big) \Norm{\dpt \uht}{L^{\infty}(I_1; L^2(\Omega))}\\
\nonumber
& \qquad + |k| \Norm{\dpt \alphaht}{L^1(I_1; L^{\infty}(\Omega))} \Norm{\dpt \uht}{L^{\infty}(I_1; L^2(\Omega))}^2\\
\nonumber
& \qquad + \Norm{(1 + k u_0) u_1}{L^2(\Omega)}^2 + c^2 \Norm{\nabla u_0}{L^2(\Omega)^d}^2 + \overline{\delta} \Norm{\nabla \dpt \mu}{L^2(Q_1)^d}^2,
\end{alignat}
where the hidden constant is independent of the time interval~$I_1$.

\paragraph{(Case~$n > 1$).}
The rest of the proof consists in proving that, for~$n = 2, \ldots, N$, it holds
\begin{equation}
\label{eq:case-m+1}
\begin{split}
& \frac{1}{2} \Norm{\dpt \uht(\cdot, t_{n}^-)}{L^2(\Omega)}^2 + \frac{c^2}{2} \Norm{\nabla \uht}{L^2(\Sigma_{n})^d}^2 + \frac12 \Norm{\jump{\dpt \uht}_{n - 1}}{L^2(\Omega)}^2
\\
& \quad + \frac12 \Norm{\dpt \uht}{L^{\infty}(I_{n}; L^2(\Omega))}^2 + \frac{c^2}{2} \Norm{\nabla \uht}{L^{\infty}(I_{n}; L^2(\Omega)^d)}^2 + \overline{\delta} \Norm{\nabla \dpt \uht}{L^2(Q_{n})}^2 \\
& \quad \lesssim \big(\Norm{f}{L^1(0, \tn; L^2(\Omega))} + \Norm{\overline{\tau}^{-1} \xi}{L^1(0, \tn; L^2(\Omega))} \big) \Norm{\dpt \uht}{L^{\infty}(0, \tn; L^2(\Omega))} \\
& \qquad + |k| \Norm{\dpt \alphaht}{L^1(0, \tn; L^{\infty}(\Omega))} \Norm{\dpt \uht}{L^{\infty}(0, \tn; L^2(\Omega))}^2\\
& \qquad + \Norm{(1 + k u_0) u_1}{L^2(\Omega)}^2 + c^2 \Norm{\nabla u_0}{L^2(\Omega)^d}^2 + \overline{\delta} \Norm{\nabla \dpt \mu}{L^2(0, \tn; L^2(\Omega)^d)}^2  ,
\end{split}
\end{equation}
with a hidden constant independent of the time interval~$\In$.

We choose in~\eqref{eq:linearized-space-time-formulation} as test function~$(\wht^{\bullet, n}, \zh^{\bullet, n})$. Since~$\zh^{\bullet, n} = 0$, with an abuse of notation, we omit it below. We obtain the identity
\begin{equation*}
\Bhtt(\uht, \wht^{\bullet, n}) = {\mathcal{L}_{\delta}}(\wht^{\bullet, n}),
\end{equation*}
where
\begin{alignat*}{2}
\nonumber
\Bhtt(\uht, \wht^{\bullet, n}) & = \big(\dpt((1 + k \alphaht) \dpt \uht), \Pi_{q-1}^t(\varphi_{n}  \dpt 
\uht)\big)_{\Qn}
\\
\nonumber
& \quad - \big((1 + k \alphaht(\cdot, \tnmo)) \jump{\dpt \uht}_{n - 1}, 
\Pi_{q-1}^t(\varphi_{n} \dpt \uht)(\cdot, \tnmo^+) \big)_{\Omega}  \\
\nonumber
& \quad + c^2 \big(\nabla \uht, \nabla \Pi_{q-1}^t(\varphi_{n} \dpt \uht)\big)_{\Qn}
\\
\nonumber
& \quad + \delta \big(\nabla \dpt \uht, \nabla \Pi_{q-1}^t(\varphi_{n} \dpt \uht) 
\big)_{\Qn}\\
& =: L_1^{(n)} + L_2^{(n)} + L_3^{(n)} + L_5^{(n)},
\end{alignat*}
and
\begin{alignat*}{3}
\nonumber
{\mathcal{L}_{\delta}}(\wht^{\bullet, n})
& = \big(f, \Pi_{q-1}^t (\varphi_{n} \dpt \uht) \big)_{\Qn} + \delta (\nabla \dpt \mu, \nabla \Pi_{q-1}^t(\varphi_{n} \dpt \uht) )_{\Qn} \\
\nonumber
& \quad + \big(\xi, \dpt \Pi_{q-1}^t (\varphi_{n} \dpt \uht) \big)_{\Qn} \\
& =: R_1^{(n)} + R_3^{(n)} + R_4^{(n)}.
\end{alignat*}
Most of the terms can be bounded as in the case~$n = 1$, so we just focus on the slight differences.

\paragraph{Bound {on}~$L_1^{(n)} + L_2^{(n)}$.} By splitting~$\Pi_{q-1}^t(\varphi \dpt \uht)$ as~$\varphi \dpt \uht - (\Id - \Pi_{q-1}^t)(\varphi \dpt \uht)$, we obtain
\begin{alignat}{3}
\nonumber
L_1^{(n)} + L_2^{(n)}
\nonumber
& = \big(\dpt((1 + k \alphaht) \dpt \uht), \varphi_{n} \dpt \uht \big)_{\Qn}
\\
\nonumber
& \quad - \theta \big((1 + k \alphaht(\cdot, \tnmo)) \jump{\dpt \uht}_{n-1}, 
\dpt \uht(\cdot, \tnmo^+) \big)_{\Omega} \\
\nonumber
& \quad - \big(\dpt((1 + k \alphaht) \dpt \uht), (\Id - \Pi_{q-1}^t)(\varphi_{n} \dpt \uht) \big)_{\Qn} 
\\
\label{eq:aux-identity-m-split-new}
& \quad + \big((1 + k \alphaht(\cdot, \tnmo)) \jump{\dpt \uht}_{n-1}, (\Id 
- \Pi_{q-1}^t) (\varphi_{n} \dpt \uht) (\cdot, \tnmo^+) \big)_{\Omega}.
\end{alignat}
Since~$\varphi_{n}(\tnmo) = \theta$ and~$\varphi_{n}(\tn) = \theta - \zeta_q$, proceeding as for the case~$n = 1$, and using the identity
\begin{equation*}
\frac12 w(\cdot, \tnmo^+)^2 + \jump{w}_{n-1} w(\cdot, \tnmo^+) = -\frac12 \jump{w}_{n-1}^2 + \frac12 w(\cdot, \tnmo^-)^2,
\end{equation*}
we have
\begin{alignat*}{3}
& \big(\dpt((1 + k \alphaht) \dpt \uht), \varphi_{n} \dpt \uht \big)_{\Qn} 
\\
\nonumber
& \quad - \theta \big((1 + k \alphaht(\cdot, \tnmo)) \jump{\dpt \uht}_{n-1}, \dpt \uht(\cdot, \tnmo^+) \big)_{\Omega} \\
& = \frac12 \int_{\Qn} \dpt((1 + k \alphaht) \varphi_{n}  (\dpt \uht)^2) \dV + \frac{\lambda_{n}}{2} \Norm{(1 + k \alphaht)^{\frac12} \dpt \uht}{L^2(\Qn)}^2 \\
& \quad + \frac{k}{2} \int_{\Qn} \varphi_{n} \dpt \alphaht (\dpt \uht)^2 \dV - \theta \big((1 + k \alphaht(\cdot, \tnmo)) \jump{\dpt \uht}_{n-1}, \dpt \uht(\cdot, \tnmo^+) \big)_{\Omega} \\
& = \frac{(\theta - \zeta_q)}{2} \Norm{(1 + k \alphaht(\cdot, \tn))^{\frac12} \dpt \uht(\cdot, \tn^-)}{L^2(\Omega)}^2 \\
& \quad - \theta \int_{\Omega} (1 + k \alphaht(\cdot, \tnmo)) \Big(\frac12 (\dpt \uht(\cdot, \tnmo^+))^2 + \jump{\dpt \uht}_{n-1} \dpt \uht(\cdot, \tnmo^+) \Big) \dx \\
& \quad + \frac{\lambda_{n}}{2} \int_{\Qn} (1 + k \alphaht) (\dpt \uht)^2 \dV + \frac{k}{2} \int_{\Qn} \varphi_{n} \dpt \alphaht (\dpt \uht)^2 \dV \\
& \geq \frac{(\theta - \zeta_q) (1 - |k| \la)}{2} \Norm{\dpt \uht(\cdot, \tn^-)}{L^2(\Omega)}^2 + \frac{(1 - |k|\la) \lambda_{n}}{2}  \Norm{\dpt \uht}{L^2(\Qn)}^2  \\
& \quad + \frac{\theta(1 - |k| \la) }{2} \Norm{\jump{\dpt \uht}_{n-1}}{L^2(\Omega)}^2 - \frac{\theta(1 + |k| \ua)}{2}\Norm{\dpt \uht(\cdot, \tnmo^-)}{L^2(\Omega)}^2  \\
& \quad - \frac{|k| \theta}{2} \Norm{\dpt \alphaht}{L^{1}(\In; L^2(\Omega))} \Norm{\dpt \uht}{L^{\infty}(\In; L^2(\Omega))}^2.
\end{alignat*}
The third term on the right-hand side of~\eqref{eq:aux-identity-m-split-new} can be bounded as for the case~$n = 1$ as follows:
\begin{alignat*}{3}
- \big( &\dpt((1 + k \alphaht) \dpt \uht), (\Id - \Pi_{q-1}^t)(\varphi_{n} \dpt \uht) \big)_{\Qn} \\
& \geq -\zeta_q |k| (1 + \Cinv) \Norm{\dpt \alphaht}{L^1(\In; L^{\infty}(\Omega))} (1 + \Cp \Cinv ) \Norm{\dpt \uht}{L^{\infty}(\In; L^2(\Omega))}^2.
\end{alignat*}
Similarly, the following bound {on} the fourth term on the right-hand side of~\eqref{eq:aux-identity-m-split-new} can be obtained:
\begin{alignat*}{3}
\big((1 + k \alphaht(\cdot, \tnmo)) & \jump{\dpt \uht}_{n-1}, (\Id 
- \Pi_{q-1}^t) (\varphi_{n} \dpt \uht) (\cdot, \tnmo^+) \big)_{\Omega} \\
& \geq -\frac{(1 - |k| \la) \theta}{4} \Norm{\jump{\dpt \uht}_{n-1}}{L^2(\Omega)}^2  - \frac{\zeta_q}{4 \theta} \Big(\frac{1 + |k|\ua}{1 - |k| \la}\Big) \Norm{\dpt \uht}{L^{\infty}(\In; L^2(\Omega))}^2. 
\end{alignat*}

\paragraph{Bound {on}~$L_3^{(n)}$.} Adding and subtracting suitable terms, we have
\begin{alignat*}{3}
\nonumber
L_3^{(n)} & = c^2\big(\nabla \uht, \nabla \Pi_{q - 1}^t (\varphi_{n} \dpt \uht) \big)_{\Qn} \\
\nonumber
& = c^2\big(\nabla \uht, \varphi_{n} \nabla \dpt \uht \big)_{\Qn}
- c^2\big(\nabla \uht, (\Id - \Pi_{q-1}^t) (\varphi_{n} \nabla \dpt \uht) \big)_{\Qn} \\
& \geq c^2\big(\nabla \uht, \varphi_{n} \nabla \dpt \uht \big)_{\Qn},
\end{alignat*}
as the last term in the second line is negative by the same reasoning used in~\eqref{eq:term-M2} for~$L_3^{(b)}$ above. Then, proceeding as for the case~$n = 1$, we have
\begin{alignat*}{3}
L_3^{(n)} & \geq \frac{c^2}{2} (\theta - \zeta_q) \Norm{\nabla \uht}{L^2(\Sigma_{n})^d}^2 - \frac{c^2 \theta}{2} \Norm{\nabla \uht}{L^2(\Sigma_{n-1})^d}^2 + \frac{c^2 \lambda_{n} }{2} \Norm{\nabla \uht}{L^2(\Qn)^d}^2.
\end{alignat*}
The remaining terms can be treated as in the case~$n = 1$ as follows:
\begin{alignat*}{3}
L_5^{(n)} & \geq \delta (\theta - \zeta_q) \Norm{\nabla \dpt \uht}{L^2(\Qn)^d}^2, \\
R_1^{(n)} & \le \CS \theta \Norm{f}{L^1(\In; L^2(\Omega))} \Norm{\dpt \uht}{L^{\infty}(\In; L^2(\Omega))}, \\
R_3^{(n)} & \le \frac{\overline{\delta} \theta^2}{2(\theta - \zeta_q)} \Norm{\nabla \dpt \mu}{L^2(\Qn)^d}^2 + \frac{\delta (\theta - \zeta_q)}{2} \Norm{\nabla \dpt \uht}{L^2(\Qn)^d}^2, \\
R_4^{(n)} & \le \Cinv \CS \theta \Norm{\overline{\tau}^{-1} \xi}{L^1(\In; L^2(\Omega))} \Norm{\dpt \uht}{L^{\infty}(\In; L^2(\Omega))}.
\end{alignat*}

\paragraph{Conclusion of case~$n > 1$.} Collecting the estimates for the case~$n > 1$ and using~\eqref{eq:Linfty-L2-dt-uht}, we have
\begin{alignat*}{3}
& (\theta - \zeta_q) \Big(\frac{(1 - |k| \la)}{2} \Norm{\dpt \uht(\cdot, \tn^{-})}{L^2(\Omega)}^2 + \frac{c^2}{2} \Norm{\nabla \uht}{L^2(\Sigma_{n})^d}^2\Big) + \frac{\theta (1 - |k| \la)}{4} \Norm{\jump{\dpt \uht}_{n-1}}{L^2(\Omega)}^2 \\
& + \frac{\zeta_q}{2}\bigg(\frac{(1 - |k| \la) }{(1 + \Cinv)^2}   - \frac{1}{2 \theta} \Big(\frac{1 + |k| \ua}{1 - |k| \la}\Big)\bigg) \Norm{\dpt \uht}{L^{\infty}(I_1; L^2(\Omega))}^2  + \frac{\zeta_q c^2}{2(1 + \Cinv)^2} \Norm{\nabla \uht}{L^{\infty}(\In; L^2(\Omega)^d)}^2 \\
& + \frac{\delta (\theta - \zeta_q)}{2} \Norm{\nabla \dpt \uht}{L^2(\Qn)^d}^2 \\
& \qquad \qquad \le \frac{\theta(1 + |k| \ua)}{2} \Norm{\dpt \uht(\cdot, \tnmo^-)}{L^2(\Omega)}^2 + \frac{c^2 \theta}{2} \Norm{\nabla \uht}{L^2(\Sigma_{n-1})^d}^2 \\
& \qquad \qquad \quad + \CS \theta \Norm{f}{L^1(\In; L^2(\Omega))} \Norm{\dpt \uht}{L^{\infty}(0, \tn; L^2(\Omega))}\\
& \qquad \qquad \quad + \frac{\overline{\delta} \theta^2}{2(\theta - \zeta_q)} \Norm{\nabla \dpt \mu}{L^2(\Qn)^d}^2 + \Cinv \CS \theta  \Norm{\overline{\tau}^{-1} \xi}{L^1(\In; L^2(\Omega))} \Norm{\dpt \uht}{L^{\infty}(0, \tn; L^2(\Omega))} \\
& \qquad \qquad \quad + \Big(\frac{\theta}{2} + \zeta_q  (1 + \Cinv) (1 + \Cp \Cinv) \Big) |k| \Norm{\dpt \alphaht}{L^1(\In; L^{\infty}(\Omega))} \Norm{\dpt \uht}{L^{\infty}(0, \tn; L^2(\Omega))}^2.
\end{alignat*}
The same parameter~$\theta$ used for the case~$n = 1$ leads to
\begin{alignat}{3}
\nonumber
\frac{1}{2} & \Norm{\dpt \uht(\cdot, \tn^-)}{L^2(\Omega)}^2 + \frac{c^2}{2} \Norm{\nabla \uht}{L^2(\Sigma_{n})^d}^2 + \frac12 \Norm{\jump{\dpt \uht}_{n-1}}{L^2(\Omega)}^2 + \frac{c^2}{2} \Norm{\nabla \uht}{L^2(\Sigma_{n-1})^d}^2 \\
\nonumber
& + \frac12 \Norm{\dpt \uht}{L^{\infty}(\In; L^2(\Omega))}^2 + \frac{c^2}{2} \Norm{\nabla \uht}{L^{\infty}(\In; L^2(\Omega)^d)}^2 + \overline{\delta} \Norm{\nabla \dpt \uht}{L^2(\Qn)}^2 \\
\nonumber
& \lesssim \frac{1}{2} \Norm{\dpt \uht(\cdot, \tnmo^-)}{L^2(\Omega)}^2 + \frac{c^2}{2} \Norm{\nabla \uht}{L^2(\Sigma_{n-1})^d}^2 \\
\nonumber
& \quad + \big(\Norm{f}{L^1(\In; L^2(\Omega))}  + \Norm{\overline{\tau}^{-1} \xi}{L^1(\In; L^2(\Omega))} \big) \Norm{\dpt \uht}{L^{\infty}(0, \tn; L^2(\Omega))}\\
\label{eq:aux-stab-m}
& \quad + \frac{|k|\theta}{2} \Norm{\dpt \alphaht}{L^1(\In; L^2(\Omega))} \Norm{\dpt \uht}{L^{\infty}(0, \tn; L^2(\Omega))}^2 + \overline{\delta} \Norm{\nabla \dpt \mu}{L^2(\Qn)^d}^2,
\end{alignat}
with a hidden constant independent of the time interval~$\In$. 

By combining~\eqref{eq:aux-stab-m} with the estimate in Proposition~\ref{prop:weak-continuous-dependence} to bound $\Norm{\dpt \uht(\cdot, \tnmo^-)}{L^2(\Omega)}^2$, and using the additive properties of the~$L^1$- and~$L^2$-norms, we obtain~\eqref{eq:case-m+1}.

\paragraph{Conclusion of the proof.} For each~$n \in \{1, \ldots, N\}$, summing~\eqref{eq:case-m+1} and the estimate in Proposition~\ref{prop:weak-continuous-dependence} for~$n = N$, and neglecting some positive terms, we arrive at
\begin{equation}
\label{eq:stab-global}
\begin{split}
& \frac12 \Norm{\dpt \uht}{L^{\infty}(\In; L^2(\Omega))}^2 + \frac{c^2}{2} \Norm{\nabla \uht}{L^{\infty}(\In; L^2(\Omega)^d)}^2   + \SemiNorm{\dpt \uht}{\sf J}^2 \\
& \quad + \frac{c^2}{2} \big(\Norm{\nabla \uht}{L^2(\ST)^d}^2 + \Norm{\nabla \uht}{L^2(\SO)^d}^2 \big)  + \overline{\delta} \Norm{\nabla \dpt \uht}{L^2(\Qn)^d}^2 \\
& \qquad \lesssim \big(\Norm{f}{L^1(0, T; L^2(\Omega))} + \Norm{\overline{\tau}^{-1} \xi}{L^1(0, T; L^2(\Omega))} \big) \Norm{\dpt \uht}{L^{\infty}(0, T; L^2(\Omega))} \\
& \qquad\quad  + \Ca \Norm{\dpt \uht}{L^{\infty}(0, T; L^2(\Omega))}^2 + \Norm{(1 + k u_0) u_1}{L^2(\Omega)}^2 + c^2 \Norm{\nabla u_0}{L^2(\Omega)^d}^2 + \overline{\delta} \Norm{\nabla \dpt \mu}{L^2(\QT)^d}^2.
\end{split}
\end{equation}
Taking~$n$ as the index where the left-hand side of~\eqref{eq:stab-global} takes its maximum value and~$\Ca$ small enough, we obtain the desired result~\eqref{eq:strong-continuous-dependence}.
\end{proof}

\section{{Convergence analysis
of the linearized problem~\label{sect:a-priori-linearized}}}
In this section, we derive~\emph{a priori} bounds for the discrete error $\Piht u - \uht$ in the energy norm, where~$\Pi_{h\tau}$ is the combined projection introduced in Section~\ref{Sec: Space-time projection} below, and then estimate $u-\uht$ in the norms
\begin{equation}
\label{eq:reduced-energy-norm} 
\Norm{\dpt v}{L^{\infty}(0, T; L^2(\Omega))} \quad \text{ and } \quad \Norm{\nabla v}{L^{\infty}(0, T; L^2(\Omega)^d)}.
\end{equation}

Henceforth, we further assume that the family of spatial meshes~$\{\Th\}_{h>0}$ is quasi-uniform, i.e., there exists a positive constant~$C_{\mathrm{qu}}$ independent of~$h$ such that~$\diam(K) \le C_{\mathrm{qu}} \diam(K')$ for all~$K, K' \in \Th$. This is required to obtain optimal convergence rates for the error of the Ritz projection~$\Rh$ in the~$L^2(\Omega)$ norm (see~Lemmas~\ref{lemma:estimates-Rh} and~\ref{lemma:stab-Rh}). \purple{In addition, we also need} the following assumption on the \purple{spatial} domain~$\Omega$.
\begin{assumption}[Elliptic regularity for~$\Omega$]
\label{asm:elliptic-regularity}
The spatial domain~$\Omega$ is such that
\begin{equation*}
\text{ if } z \in H_0^1(\Omega) \text{ and } \Delta z \in L^2(\Omega), \text{ then } z \in H^2(\Omega).
\end{equation*}
\end{assumption}

Furthermore, we assume that the solution~$u$ to the Westervelt IBVP belongs to the following space for all $\delta \in [0, \overline{\delta}]$: 
\begin{equation}
\begin{aligned}
\label{eq:regularity-Westervelt}
\Ulm :=  W_1^m(0, T; H^2(\Omega)) &\cap W_{\infty}^{m + 1}(0, T; H^1(\Omega)) \cap W_1^2 (0, T; H^{\ell + 1}(\Omega)) \\ &\cap C^1([0, T]; H^{\ell + 1}(\Omega) \cap H_0^1(\Omega))  \cap C^2([0,T]; \Linf),
\end{aligned}
\end{equation}
for some~${2} \le m \le q$ and~$1 \le \ell \le p$, where~$p$ and~$q$ are the degrees of approximation in space and time, respectively. 
In addition, we assume that the initial data have (at least) the following regularity:
\begin{equation}
\label{eq:regularity-initial-data}
u_0 \in H^{\ell + 1}(\Omega) \cap H_0^1(\Omega) \quad \text{ and } \quad u_1 \in H^{\ell + 1}(\Omega)\cap H_0^1(\Omega).
\end{equation}
A~$\delta$-uniform well-posedness result for the Westervelt IBVP can be found in~\cite[Thm.\ 4.1]{kaltenbacher2022parabolic}. For sufficiently smooth and small initial data (and possibly small final time), it has been shown that the problem has a solution 
\begin{equation*}
\begin{aligned}
u \in \,L^\infty(0,T; \Honethree) \cap W^{1}_{\infty}(0,T; H^2(\Omega)\cap H_0^1(\Omega) ) 
\cap H^2(0,T; H^1(\Omega)),
\end{aligned}
\end{equation*}
where $\Honethree=\, \left\{u\in H^3(\Omega)\,:\, \mbox{tr}_{\partial\Omega} u = 0, \  \mbox{tr}_{\partial\Omega} \Delta u = 0\right\}.$ The analysis in~\cite{kaltenbacher2022parabolic} is geared toward minimal regularity assumptions on data; we expect, however, that the techniques in~\cite{kaltenbacher2022parabolic} can be extended to allow showing uniform well-posedness in $\Ulm$.

Next theorem is the main result in this section and is proven in Section~\ref{subsect:a-priori-estimates-linearized} below.
\begin{theorem}
[\emph{A priori} estimates for the linearized problem]\label{thm:linear-error-bound}
Assume that~$2 \le m \le q$ and~$1 \le \ell \le p$,
and that the spatial domain~$\Omega$ satisfies Assumption~\ref{asm:elliptic-regularity}. Assume further that~$\Th$ is quasi-uniform, $\delta \in [0, \overline{\delta}]$ for some fixed~$\overline{\delta} > 0$,  and that Assumption~\ref{asm:nondegeneracy} holds. Let~$\uht$ be the solution to the linearized problem~\eqref{eq:linearized-space-time-formulation} with~$\mu = 0$, $\xi = 0$, and~$\alphaht(\cdot, 0) = \Rh u_0$.
If bound~\eqref{eq:bound-alpha-W-1-infty} holds with~$\Ca$ such that estimate~\eqref{eq:strong-continuous-dependence} holds, the solution~$u$ to the Westervelt IBVP 
belongs to $\Ulm$, and~$u_0, u_1$ satisfy~\eqref{eq:regularity-initial-data}, then the following estimates hold:
\begin{alignat}{3}
\nonumber
\Norm{\dpt \eu}{L^{\infty}(0, T; L^2(\Omega))} & \leq \Clinone \Bigl\{ {|k|\Norm{\dpt (u - \alphaht)}{L^{\infty}(0, T; L^2(\Omega))} \Norm{\dpt u}{L^1(0, T; L^{\infty}(\Omega))} }\\
\nonumber
& \qquad \quad + {|k| \Norm{u - \alphaht}{C^0([0, T]; L^2(\Omega))} \Norm{\dptt u}{L^1(0, T; L^{\infty}(\Omega))} }
\\
\nonumber
& \qquad \quad + \tau^m \Big(\Norm{\Delta {\dpt^{(m)} u}}{L^1(0, T; L^2(\Omega))} {+ \Norm{{\dpt^{(m + 1)} u}}{L^{\infty}(0, T; L^2(\Omega))}}  \\
\nonumber
& \qquad \qquad\qquad  + {|k|} \Norm{\nabla {\dpt^{(m + 1)} u}}{L^{\infty}(0, T; L^2(\Omega)^d)} + {\overline{\delta}} \Norm{\nabla {\dpt^{(m + 1)} u}}{L^2(0, T; L^2(\Omega)^d)}\Big) \\
\nonumber
& \qquad \quad + h^{\ell + 1} \Big(|k| \Norm{\dpt u}{L^1(0, T; H^{\ell + 1}(\Omega))} + {(1 + |k| \ua)} \Norm{\dptt u}{L^1(0, T; H^{\ell + 1}(\Omega))} \\
\nonumber
& \qquad \qquad \qquad + {|k| \Norm{u_1}{L^{\infty}(\Omega)}}\Norm{u_0}{H^{\ell + 1 }(\Omega)} +  (1 + |k| \Norm{u_0}{H^2(\Omega)} ) \SemiNorm{u_1}{H^{\ell + 1}(\Omega)} \\
\label{eq:error-dpt}
& \qquad \qquad \qquad + \Norm{\dpt u}{C^0([0, T]; H^{\ell + 1}(\Omega))}  \Big) \Bigr\}, \\
\nonumber
\Norm{\nabla \eu}{L^{\infty}(0, T; L^2(\Omega)^d)} 
& \leq \Clintwo \Bigl\{ {|k|\Norm{\dpt (u - \alphaht)}{L^{\infty}(0, T; L^2(\Omega))} \Norm{\dpt u}{L^1(0, T; L^{\infty}(\Omega))} }\\
\nonumber
& \qquad \quad + {|k| \Norm{u - \alphaht}{C^0([0, T]; L^2(\Omega))} \Norm{\dptt u}{L^1(0, T; L^{\infty}(\Omega))} } \\
\nonumber
& \qquad \quad + \tau^m \Big(\Norm{\Delta {\dpt^{(m)} u}}{L^1(0, T; L^2(\Omega))} + {(|k| + \tau)}\Norm{\nabla {\dpt^{(m + 1)} u}}{L^{\infty}(0, T; L^2(\Omega)^d)} \\
\nonumber
& \qquad \qquad \qquad  + {\overline{\delta}} \Norm{\nabla {\dpt^{(m + 1)} u}}{L^2(0, T; L^2(\Omega)^d)}\Big) \\
\nonumber
& \qquad \quad + h^{\ell} \Big(\Norm{u}{L^{\infty}(0, T; H^{\ell + 1}(\Omega))} + {|k|} h \Norm{\dpt u}{L^1(0, T; H^{\ell + 1}(\Omega))} \\
\nonumber
& \qquad \qquad \qquad 
+ {(1 + |k| \ua)} h \Norm{\dptt u}{L^1(0, T; H^{\ell + 1}(\Omega))} \\
\label{eq:error-nabla}
& \qquad \qquad \qquad 
+ {|k| h \Norm{u_1}{L^{\infty}(\Omega)}}\SemiNorm{u_0}{H^{\ell + 1 }(\Omega)} +  (1 + |k| \Norm{u_0}{H^2(\Omega)} ) h \SemiNorm{u_1}{H^{\ell + 1}(\Omega)}  \Bigl) \Bigr\}. 
\end{alignat}
\end{theorem}

\begin{remark}[Error estimates for the linear damped wave equation with constant coefficients] 
Setting~$k = 0$ and~$\delta > 0$, and assuming that, 
in addition to~\eqref{eq:regularity-Westervelt}, $u$ belongs to~$W_1^{m + 1}(0, T; H^2(\Omega))$, the following error estimates follow from Theorem~\ref{thm:linear-error-bound}: 
\begin{alignat*}{3}
\Norm{\dpt \eu}{L^{\infty}(0, T; L^2(\Omega))} & \lesssim \tau^m \big(\Norm{\Delta {\dpt^{(m)} u}}{L^1(0, T; L^2(\Omega))} + \Norm{{\dpt^{(m + 1)} u}}{L^{\infty}(0, T; L^2(\Omega))} + \delta \Norm{\nabla {\dpt^{(m + 1)} u}}{L^2(0, T; L^2(\Omega)^d)}\big) \\
\nonumber
& \quad + h^{\ell + 1} \big(\Norm{\dptt u}{L^1(0, T; H^{\ell + 1}(\Omega))} + \Norm{\dpt u}{C^0([0, T]; H^{\ell + 1}(\Omega))}  \big), \\
\Norm{\nabla \eu}{L^{\infty}(0, T; L^2(\Omega)^d)} & 
\lesssim \tau^m \big(\tau \Norm{\Delta {\dpt^{(m + 1)} u}}{L^1(0, T; L^2(\Omega))} + \tau \Norm{\nabla {\dpt^{(m + 1)} u}}{L^{\infty}(0, T; L^2(\Omega)^d)} \\
& \qquad\quad  + \delta \Norm{\nabla {\dpt^{(m + 1)} u}}{L^2(0, T; L^2(\Omega)^d)} \big) \\
\nonumber
& \quad + h^{\ell} \big(\Norm{u}{L^{\infty}(0, T; H^{\ell + 1}(\Omega))} 
+ h \Norm{\dptt u}{L^1(0, T; H^{\ell + 1}(\Omega))} \big).
\end{alignat*}
The estimate of~$\dpt \eu$ is optimal, whereas the estimate of~$\nabla \eu$ is optimal with respect to~$h$, and it is suboptimal by one order with respect to~$\tau$
unless the damping parameter is small ($\delta \lesssim \tau$), which is the case we are interested in this work.
\eremk
\end{remark}
\subsection{{Combined space--time projection}} \label{Sec: Space-time projection}
We first recall the definition of the auxiliary projection in~\cite[Def.~5.1]{Walkington:2014}.
\begin{definition}[Projection~$\Pt$]
\label{DEF:Pt}
Let~$q \in \IN$ with~$q \geq 2$. Given a partition~$\Tt$ of the time interval~$(0, T)$, the projection operator~$\Pt : C^1(0, T) + \Vtq \rightarrow \Vtq$ is defined for all~$v \in C^1(0, T) + \Vtq$ as follows:
\begin{subequations}
\label{eq:proj-Pt-def}
\begin{alignat}{3} 
\label{eq:proj-Pt-def-1}
\Pt v(0) & = v(0),\\
(\Pt v)' (\tn^-) & = v'(\tn^-) & & \qquad \text{ for } n = 1, \ldots, N, \\
\big( (\Pt v)' - v', \, p_{q-2}\big)_{\In} & = 0 & & \qquad \text{ for } n = 1, \ldots, N, \ \forall p_{q - 2} \in \Pp{q-2}{\In}.
\end{alignat}
\end{subequations}
\end{definition}

Moreover, we denote by~$\Rh : H_0^1(\Omega) \rightarrow \Vhp$ the Ritz projection, defined for any~$z \in H_0^1(\Omega)$ as the solution to the following variational problem:
\begin{equation*}
(\nabla (\Rh z - z), \nabla \vh)_{\Omega} = 0 \qquad \forall \vh \in \Vhp.
\end{equation*}

In what follows, the projection operators~$\Pt$ and~$\Rh$ are to be understood, respectively, as applied pointwise in space and pointwise in time.
We also denote by~$\Piht$ the operator~$\Pt \circ \Rh$. We next state its approximation properties that will be exploited in the error analysis.

\begin{lemma}[Estimates for~$\Piht$]
\label{lemma:estimates-Piht}
Let~$v$ belong to the space
$$W_{\infty}^{m + 1}(0, T; L^2(\Omega)) \cap C^1([0, T]; H^{\ell + 1} \cap H_0^1(\Omega)) \cap L^{\infty}(0, T; H^{\ell + 1}(\Omega)) \cap W^{m + 1}_{\infty}(0, T; H^1(\Omega)),$$
for~$2 \le m \le q$ and~$1 \le \ell \le p$. Then, the following estimates hold:
\begin{subequations}
\begin{alignat}{3}
\label{eq:interpolation-error-dpt}
\Norm{\dpt (\Id - \Piht) v}{L^{\infty}(0, T; L^2(\Omega))} & \lesssim \tau^m \Norm{{\dpt^{(m + 1)} v}}{L^{\infty}(0, T; L^2(\Omega))} + h^{\ell + 1} \Norm{\dpt v}{C^0([0, T]; H^{\ell + 1}(\Omega))}, \\
\label{eq:interpolation-error-nabla}
\Norm{\nabla (\Id - \Piht) v}{L^{\infty}(0, T; L^2(\Omega)^d)} & \lesssim h^{\ell} \Norm{v}{L^{\infty}(0, T; H^{\ell + 1}(\Omega))} + \tau^{m + 1} \Norm{{\nabla \dpt^{(m + 1)} v}}{L^{\infty}(0, T; L^2(\Omega)^d)}, \\
\nonumber
\sum_{n = 1}^N \big(\tau_n^{-1} \Norm{\dpt(\Piht v - v)}{L^2(\Qn)}^2 & + \tau_n \Norm{\dptt (\Piht v - v)}{L^2(\Qn)}^2 \big)^{\frac12} \\
\label{eq:combined-interpolation-error}
& \lesssim \tau_{\min}^{-\frac12} h^{\ell + 1} \Norm{v}{C^0([0, T]; H^{\ell + 1}(\Omega))} + \tau^{m - \frac12} \Norm{{\dpt^{(m + 1)}v}}{L^2(0, T; L^2(\Omega))}.
\end{alignat}
\end{subequations}
\end{lemma}
\begin{proof}
We use the stability and approximation properties of the operators~$\Pt$ and~$\Rh$ from Lemmas~\ref{lemma:stab-Pt}, \ref{lemma:estimates-Pt}, and~\ref{lemma:estimates-Rh} to obtain 
\begin{alignat*}{3}
\nonumber
\Norm{\dpt (\Id - \Piht) v}{L^{\infty}(0, T; L^2(\Omega))} & \le \Norm{\dpt (\Id - \Pt) v}{L^{\infty}(0, T; L^2(\Omega))} + \Norm{\dpt \Pt (\Id - \Rh) v}{L^{\infty}(0, T; L^2(\Omega))} \\
\nonumber
& \lesssim \tau^m \Norm{{\dpt^{(m + 1)}v}}{L^{\infty} (0, T; L^2(\Omega))}+ \Norm{\dpt (\Id - \Rh) v}{C^0([0, T]; L^2(\Omega))} \\
\nonumber
&  \lesssim \tau^m \Norm{{\dpt^{(m + 1)} v}}{L^{\infty}(0, T; L^2(\Omega))} + h^{\ell + 1} \Norm{\dpt v}{C^0([0, T]; H^{\ell + 1}(\Omega))}, \\[2mm]
\nonumber
\Norm{\nabla (\Id - \Piht) v}{L^{\infty}(0, T; L^2(\Omega)^d)} & \le \Norm{\nabla (\Id - \Rh) v}{L^{\infty}(0, T; L^2(\Omega)^d)} + \Norm{\nabla \Rh (\Id - \Pt) v}{L^{\infty}(0, T; L^2(\Omega)^d)} \\
\nonumber
& \lesssim h^{\ell} \Norm{v}{L^{\infty}(0, T; H^{\ell + 1}(\Omega))} + \tau^{m + 1} \Norm{\nabla {\dpt^{(m + 1)} v}}{L^{\infty}(0, T; L^2(\Omega)^d)}.
\end{alignat*}
The estimate~\eqref{eq:combined-interpolation-error} follows in an analogous way.
\end{proof}

\subsection{Bounds on the discrete error in the linearized problem}
Let~$u$ be the solution to the Westervelt IBVP~\eqref{eq:Westervelt-IBVP} and~$\uht$ be the solution of the linearized space--time formulation~\eqref{eq:linearized-space-time-formulation} with~$\mu = 0$, $\xi = 0$, and~$\alphaht(\cdot, 0) = \Rh u_0$. We consider the following error functions:
\begin{equation*}
\eu := u - \uht = \epi - \Piht \eu \quad \text{ and } \quad \epi = u - \Piht u.
\end{equation*}
The next result establishes a bound on $\Piht \eu$ in the $\Tnormd{\cdot}$ norm, which we recall is defined in \eqref{eq:energy-norm-delta}.
\begin{proposition}
[Estimates of the discrete error]
\label{prop:a-priori-bounds-linearized} Let~$\Th$ be quasi-uniform, $\delta \in [0, \overline{\delta}]$ for some fixed~$\overline{\delta} > 0$, $q \in \IN$ with~$q \geq 2$, and let also Assumption~\ref{asm:nondegeneracy} hold. Let $\uht$ be the solution of~\eqref{eq:linearized-space-time-formulation} with~$\mu = 0$, $\xi = 0$, and~$\alphaht(\cdot, 0) = \Rh u_0$.
{If bound~\eqref{eq:bound-alpha-W-1-infty} holds with~$\Ca$ such that estimate~\eqref{eq:strong-continuous-dependence} holds,} the initial conditions~$(u_0, u_1) \in \left(H_0^1(\Omega) \cap \Linf\right)  \times \left(H_0^1(\Omega)\cap \Linf\right)$, the source term~$f \in {L^1(0, T; L^2(\Omega))}$, and the solution~$u$ to the Westervelt IBVP~\eqref{eq:Westervelt-IBVP} belongs to~${W^{2}_{1}(0, T; L^{\infty}(\Omega)) \cap W^1_{\infty}(0, T; L^2(\Omega)) 
\cap C^1(0, T; H^2(\Omega) \cap H_0^1(\Omega))}$, then the following estimate holds: 
\begin{equation}
\label{eq:projected-error-bound}
\begin{split}
\Tnormd{\Piht \eu} & \lesssim {|k|}\Norm{\dpt (u - \alphaht)}{L^{\infty}(0, T; L^2(\Omega))} \Norm{\dpt u}{L^1(0, T; L^{\infty}(\Omega))} \\
& \quad + {|k|} \Norm{u - \alphaht}{C^0([0, T]; L^2(\Omega))} \Norm{\dptt u}{L^1(0, T; L^{\infty}(\Omega))} \\
& \quad + {|k|} \Norm{(\Id - \Rh) \dpt u}{L^1(0, T; L^2(\Omega))} + {(1 + |k| \ua)} \Norm{(\Id - \Rh) \dptt u}{L^1(0, T; L^2(\Omega))} \\
& \quad + \Norm{(\Id - \Pt) \Delta u}{L^1(0, T; L^2(\Omega))} + {|k| \Norm{u_1}{L^{\infty}(\Omega)}} \Norm{(\Id - \Rh) u_0}{L^2(\Omega)} \\
& \quad + (1 + |k| \Norm{\Rh u_0}{\Linf}) \Norm{(\Id - \Rh) u_1}{\Ltwo} \\
& \quad + |k| \sum_{n = 1}^N \tau_n^{-1} \Norm{(\Id - \Pi_0^t) \alphaht}{L^{1}(I_n; L^{\infty}(\Omega))} \Norm{\dpt (\Id - \Pt) {\Rh} u}{L^{\infty}(\In; L^2(\Omega))} \\
& \quad + \overline{\delta} \Norm{\dpt (\Id - \Pt) \nabla u}{L^2(\QT)^d},
\end{split}
\end{equation}
where the hidden constant is independent of~$h$, $\tau$, and~$\delta$.
\end{proposition}
\begin{proof}
Since the space--time variational formulation~\eqref{eq:linearized-space-time-formulation} is consistent except for the linearized term, and~$\uht(\cdot, 0) = \alphaht(\cdot, 0) = \Rh u_0$, the following error equation holds:
\begin{alignat*}{3}
\Bhtt(\eu, (\wht, \zh)) & = - \sum_{n = 1}^N k (\dpt((u - \alphaht) \dpt u), \wht)_{\Qn} \\
& \qquad - k \big( (u_0 - \Rh u_0) u_1, \wht(\cdot, 0)  \big)_{\Omega} 
& & \qquad \forall (\wht, \zh) \in \Vht \times \Vhp,
\end{alignat*}
whence
\begin{alignat}{3}
\nonumber
\Bhtt(\Piht \eu, (\wht, \zh)) & = \Bhtt(\epi, (\wht, \zh)) \\
\nonumber
& \quad + \sum_{n = 1}^N k (\dpt((u - \alphaht) \dpt u), \wht)_{\Qn} \\
\label{eq:linearized-problem-projected-error}
& \quad + k \big( (u_0 - \Rh u_0) u_1, \wht(\cdot, 0) \big)_{\Omega}  \qquad \forall (\wht, \zh) \in \Vht \times \Vhp.
\end{alignat}
Due {to} the choice of the discrete initial condition~$\uht(\cdot, 0) = \Rh u_0$ (see Remark~\ref{rem:discrete-initial-condition}), the above identity is independent of~$\zh$.
Therefore, with an abuse of notation, we omit the test function~$\zh$ in the subsequent steps.

We wish to represent~$\Piht \eu$ as a solution to a linearized space--time formulation of the form of~\eqref{eq:linearized-space-time-formulation} with a suitable choice of the functions $f$, $\mu$, and $\xi$. To this end, we further rewrite~$\Bhtt(\epi, \wht)$ in~\eqref{eq:linearized-problem-projected-error}. By the definition of the bilinear form~$\Bhtt(\cdot, \cdot)$ in~\eqref{eq:Bhtt-def}, we have
\begin{alignat*}{3}
\nonumber
\Bhtt(\epi, \wht) & = \sum_{n = 1}^N \big(\dpt((1 + k \alphaht) \dpt \epi), \wht\big)_{\Qn} \\
\nonumber
& \quad - \sum_{n = 1}^{N  - 1} \big((1 + k \alphaht(\cdot, \tn)) \jump{\dpt \epi}_n, \wht(\cdot, \tn^+) \big)_{\Omega} + \big((1 + k \alphaht) \dpt \epi, \wht \big)_{\SO}\\
\nonumber
& \quad + c^2 \big(\nabla \epi, \nabla \wht\big)_{\QT} + \delta \big(\nabla \dpt \epi, \nabla \wht \big)_{\QT} \\
& =: J_1 + J_2 + J_3 + J_4 + J_5.
\end{alignat*}
We now use the properties of the projection~$\Piht$ to simplify each term~$\{J_i\}_{i = 1}^6$.

By splitting~$\epi$ as~$(u - \Pt) u + \Pt(u - \Rh u)$, we have
\begin{alignat*}{3}
\nonumber
J_1 + J_2 + J_3 & = \sum_{n = 1}^N \big(\dpt((1 + k \alphaht) \dpt \epi), \wht\big)_{\Qn}  - \sum_{n = 1}^{N  - 1} \big((1 + k \alphaht(\cdot, \tn)) \jump{\dpt \epi}_n, \wht(\cdot, \tn^+) \big)_{\Omega} \\
\nonumber
& \quad + \big((1 + k \alphaht) \dpt \epi, \wht \big)_{\SO} \\
\nonumber
& = \sum_{n = 1}^N \big(\dpt((1 + k \alphaht) \dpt (\Id - \Pt) u, \wht\big)_{\Qn} \\
\nonumber
& \quad - \sum_{n = 1}^{N  - 1} \big((1 + k \alphaht(\cdot, \tn)) \jump{\dpt (\Id - \Pt) u}_n, \wht(\cdot, \tn^+) \big)_{\Omega} \\
\nonumber
& \quad + \big((1 + k \alphaht) \dpt (\Id - \Pt) u, \wht \big)_{\SO} \\
\nonumber
& \quad + \sum_{n = 1}^N \big(\dpt((1 + k \alphaht) \dpt \Pt (\Id - \Rh) u, \wht\big)_{\Qn} \\
\nonumber
& \quad - \sum_{n = 1}^{N  - 1} \big((1 + k \alphaht(\cdot, \tn)) \jump{\dpt \Pt (\Id - \Rh) u}_n, \wht(\cdot, \tn^+) \big)_{\Omega} \\
\nonumber
& \quad + \big((1 + k \alphaht) \dpt \Pt (\Id - \Rh) u, \wht \big)_{\SO} \\
& =: J_{1,\tau} + J_{2, \tau} + J_{3, \tau} + J_{1, h} + J_{2, h} + J_{3, h}.
\end{alignat*}
Integrating by parts in time, and using the definition in~\eqref{eq:proj-Pt-def} of the projection~$\Pt$ and the identity
\begin{equation*}
v(\cdot, \tn^+) \jump{u}_n  + u(\cdot, \tn^{-}) \jump{v}_n = \jump{uv}_n \quad \text{ for } n = 1, \ldots, N - 1,
\end{equation*}
we deduce
\begin{alignat}{3}
\nonumber
J_{1, \tau} + J_{2, \tau} + J_{3, \tau} & = \sum_{n = 1}^N \big(\dpt((1 + k \alphaht) \dpt (\Id - \Pt) u, \wht\big)_{\Qn} \\
\nonumber
& \quad - \sum_{n = 1}^{N  - 1} \big((1 + k \alphaht(\cdot, \tn)) \jump{\dpt (\Id - \Pt) u}_n, \wht(\cdot, \tn^+) \big)_{\Omega} \\
\nonumber
& \quad + \big((1 + k \alphaht) \dpt (\Id - \Pt) u, \wht \big)_{\SO} \\
\nonumber
& = - \sum_{n = 1}^N \big((1 + k \alphaht) \dpt (\Id - \Pt) u, \dpt \wht \big)_{\Qn} + \big((1 + k \alphaht \dpt (\Id - \Pt) u, \wht \big)_{\ST} \\
\nonumber
& \quad + \sum_{n = 1}^{N - 1}  \int_{\Omega} (1 + k \alphaht(\cdot, \tn)) \big(\jump{\wht \dpt (\Id - \Pt) u}_n - \wht(\cdot, \tn^+) \jump{\dpt \uht}_n \big) \dx \\
\nonumber
& = - \sum_{n = 1}^{N} \big((1 + k \alphaht) \dpt (\Id - \Pt) u, \dpt \wht \big)_{\Qn} + \big((1 + k \alphaht) \dpt (\Id - \Pt) u, \wht \big)_{\ST} \\
\nonumber
& \quad + \sum_{n = 1}^{N - 1} \big( (1 + k \alphaht(\cdot, \tn)) \dpt (\Id - \Pt) u(\cdot, \tn^-), \jump{\wht}_n \big)_{\Omega} \\
\label{eq:identity-J1t-J2t-J3t}
& = -\sum_{n = 1}^N \big(k (\alphaht - \Pi_0^t \alphaht) \dpt(\Id - \Pt) u, \dpt \wht \big)_{\Qn},
\end{alignat}
where, in the last step, we have used the fact that~$(\Pi_0^t \alphaht \dpt \wht)(\bx, \cdot)  \in \Pp{q-2}{\Tt}$ for all~$\bx \in \Omega$.

Integrating by parts in time and using similar arguments as above, we get
\begin{alignat*}{3}
\nonumber
J_{1, h} + J_{2, h} + J_{3, h} & = \sum_{n = 1}^N \big(\dpt((1 + k \alphaht) \dpt \Pt (\Id - \Rh) u, \wht\big)_{\Qn} \\
\nonumber
& \quad - \sum_{n = 1}^{N  - 1} \big((1 + k \alphaht(\cdot, \tn)) \jump{\dpt \Pt (\Id - \Rh) u}_n, \wht(\cdot, \tn^+) \big)_{\Omega} \\
\nonumber
& \quad + \big((1 + k \alphaht) \dpt \Pt (\Id - \Rh) u, \wht \big)_{\SO} \\
\nonumber
& = - \sum_{n = 1}^N \big( ( 1 + k \alphaht) \dpt \Pt (\Id - \Rh) u, \dpt \wht \big)_{\Qn} \\
\nonumber 
& \quad + \big((1 + k \alphaht) \dpt \Pt (\Id - \Rh) u, \wht \big)_{\ST} \\
\nonumber
& \quad + \sum_{n = 1}^{N - 1} \big((1 + k \alphaht(\cdot, \tn)) \dpt \Pt(\Id - \Rh) u(\cdot, \tn^-), \jump{\wht}_n \big)_{\Omega}.
\end{alignat*}
Moreover, using the definition of~$\Pt$ and integrating by parts back in time, we obtain
\begin{alignat}{3}
\nonumber
J_{1, h} + J_{2, h} + J_{3, h}
& = - \sum_{n = 1}^N \big( ( 1 + k \alphaht) \dpt  (\Id - \Rh) u, \dpt \wht \big)_{\Qn} + \big((1 + k \alphaht) \dpt  (\Id - \Rh) u, \wht \big)_{\ST} \\
\nonumber
& \quad + \sum_{n = 1}^{N - 1} \big((1 + k \alphaht(\cdot, \tn)) \dpt (\Id - \Rh) u(\cdot, \tn^-), \jump{\wht}_n \big)_{\Omega} \\
\nonumber
&\quad + \sum_{n = 1}^N \big(k \alphaht \dpt (\Id - \Pt)(\Id - \Rh) u, \dpt \wht \big)_{\Qn}
\\
\nonumber
& = \sum_{n = 1}^N \Big[\big(k \dpt \alphaht (\Id - \Rh) \dpt u, \wht \big)_{\Qn} + \big( (1 + k \alphaht) (\Id - \Rh) \dptt u, \wht \big)_{\Qn} \\
\nonumber
& \quad + \big(k \alphaht \dpt (\Id - \Pt)(\Id - \Rh) u, \dpt \wht \big)_{\Qn}\Big] \\
\label{eq:identity-J1h-J2h-J3h}
& \quad {+ \big( (1 + k \alphaht(\cdot, 0)) (\Id - \Rh) \dpt u (\cdot, 0), \wht \big)_{\Omega}}.
\end{alignat}
Finally, since~$\alphaht(\cdot, 0) = \Rh u_0$ and~$\dpt u(\cdot, 0) = u_1$, the last term in the above equation is equal to
\begin{equation*}
\big( (1 + k \Rh u_0) (\Id - \Rh) u_1, \wht(\cdot, 0) \big)_{\Omega}.
\end{equation*}

As for the term~$J_4$, we split~$\epi$ as~$(\Id - \Rh)u + \Rh(\Id - \Pt) u$, and use the orthogonality properties of~$\Rh$, the commutativity of~$\Pt$ and the spatial gradient operator~$\nabla$, and integration by parts in space 
to obtain
\begin{alignat}{3}
\nonumber
J_4 & = c^2(\nabla \epi, \nabla \wht)_{\QT} = c^2 (\nabla (\Id - \Rh) u, \nabla \wht)_{\QT} + c^2(\nabla \Rh (\Id - \Pt) u, \nabla \wht)_{\QT} \\
\label{eq:identity-J4}
& \quad = c^2\big((\Id - \Pt) \nabla u, \nabla \wht \big)_{\QT} = -c^2\big( (\Id - \Pt) \Delta u, \wht\big)_{\QT}.
\end{alignat}
Similarly, we also have
\begin{alignat}{3}
\nonumber
J_5 & = \delta \big(\nabla \dpt \epi, \nabla \wht \big)_{\QT} = \delta \big( \nabla (\Id - \Rh) \dpt u, \nabla \wht)_{\QT} + \delta \big(\nabla \Rh \dpt (\Id - \Pt) u, \nabla \wht \big)_{\QT} \\
\label{eq:identity-J5}
& = \delta \big(\nabla \dpt(\Id - \Pt) u, \nabla \wht \big)_{\QT}.
\end{alignat}

From~\eqref{eq:linearized-problem-projected-error} and the simplified expressions of the terms~$\{J_i\}_{i = 1}^5$ in~\eqref{eq:identity-J1t-J2t-J3t},  \eqref{eq:identity-J1h-J2h-J3h}, \eqref{eq:identity-J4}, and~\eqref{eq:identity-J5}, we deduce that~$\Piht \eu$ solves a linearized space--time formulation of the form of~\eqref{eq:linearized-space-time-formulation} with
\begin{alignat*}{3}
f & = k \dpt \alphaht (\Id - \Rh) \dpt u + (1 + k \alphaht)(\Id - \Rh) \dptt u - c^2 (\Id - \Pt) \Delta u, \\
\xi & = -k(\Id - \Pi_0^t) \alphaht \dpt (\Id - \Pt) u - k (\Id - \Pi_0^t) \alphaht \dpt (\Id - \Pt) \Rh u, \\
\mu & = (\Id - \Pt) u.
\end{alignat*}
Therefore, the continuous dependence on the data in Theorem~\ref{thm:strong-continuous-dependence} and the H\"older inequality yield the desired result. 
\end{proof}

\subsection{\emph{A priori} error estimates for the linearized problem \label{subsect:a-priori-estimates-linearized}}
We are now in a position to prove Theorem~\ref{thm:linear-error-bound}.
\begin{proof}[Proof of Theorem~\ref{thm:linear-error-bound}]
We first further estimate the discrete error by bounding the right-hand side terms in~\eqref{eq:projected-error-bound}. By using 
the approximation properties of~$\Pi_0^t$, $\Pt$, and~$\Rh$ from Lemmas~\ref{lemma:polynomial-approx-time}, \ref{lemma:estimates-Pt}, \ref{lemma:estimates-Rh}, and \ref{lemma:stab-Rh}, we obtain
\begin{alignat*}{3}
{|k|} \Norm{(\Id - \Rh) \dpt u}{L^1(0, T; L^2(\Omega))} & \lesssim {|k|} h^{\ell + 1} \Norm{\dpt u}{L^1(0, T; H^{\ell + 1}(\Omega))}, \\
{(1 + |k| \ua)}\Norm{(\Id - \Rh) \dptt u}{L^1(0, T; L^2(\Omega))} & \lesssim {(1 + |k| \ua)} h^{\ell + 1} \Norm{\dptt u}{L^1(0, T; 
H^{\ell + 1}(\Omega))}, \\
\Norm{(\Id - \Pt) \Delta u}{L^1(0, T; L^2(\Omega))} & \lesssim \tau^{m} \Norm{\Delta {\dpt^{(m)} u}}{L^1(0, T; L^2(\Omega))}, \\
{|k| \Norm{u_1}{L^{\infty}(\Omega)}} \Norm{(\Id - \Rh) u_0}{L^2(\Omega)} & \lesssim h^{\ell + 1} {|k| \Norm{u_1}{L^{\infty}(\Omega)}} \SemiNorm{u_0}{H^{\ell + 1}(\Omega)}, \\
(1 + |k|\, \Norm{\Rh u_0}{\Linf}) \Norm{(\Id - \Rh) u_1}{\Ltwo} & \lesssim (1 + |k| \Norm{u_0}{H^2(\Omega)} ) h^{\ell + 1} \SemiNorm{u_1}{H^{\ell + 1}(\Omega)}. \end{alignat*}
Furthermore,
\begin{alignat*}{3}
|k| \sum_{n = 1}^N \tau_n^{-1} \Norm{(\Id - \Pi_0^t) \alphaht}{L^{1}(I_n; L^{\infty}(\Omega))} & \Norm{\dpt (\Id - \Pt)\Rh u}{L^{\infty}(\In; L^2(\Omega))}\\ \vspace*{-2mm}
& \le |k| \Cp \Norm{\dpt (\Id - \Pt) \Rh u}{L^{\infty}(0, T; L^2(\Omega))} \sum_{n = 1}^N \Norm{\dpt \alphaht}{L^1(I_n; L^{\infty}(\Omega))} \\
& \lesssim |k| \Norm{\dpt \alphaht}{L^1(0, T; L^{\infty}(\Omega))} \Norm{\dpt (\Id - \Pt) {\nabla} u}{L^{\infty}(0, T; L^2(\Omega)^d)}\\
& \lesssim \tau^m {|k|} \Norm{\nabla {\dpt^{(m + 1)} u}}{L^{\infty}(0, T; L^2(\Omega)^d)},\\
\overline{\delta} \Norm{\dpt(\Id - \Pt) \nabla u}{L^2(\QT)} & \lesssim \overline{\delta} \tau^m \Norm{\nabla {\dpt^{(m + 1)} u}}{L^2(0, T; L^2(\Omega)^d)}.
\end{alignat*}
Therefore, the following estimate of the projected error is obtained:
\begin{alignat}{3}
\nonumber
\Tnormd{\Piht \eu} & \lesssim 
{|k|\Norm{\dpt (u - \alphaht)}{L^{\infty}(0, T; L^2(\Omega))} \Norm{\dpt u}{L^1(0, T; L^{\infty}(\Omega))} }\\
\nonumber
& \quad + {|k| \Norm{u - \alphaht}{C^0([0, T]; L^2(\Omega))} \Norm{\dptt u}{L^1(0, T; L^{\infty}(\Omega))} } \\
\nonumber
& \quad + \tau^m \Big(\Norm{\Delta {\dpt^{(m)} u}}{L^1(0, T; L^2(\Omega))} + |k| \Norm{\nabla {\dpt^{(m + 1)}u} }{L^{\infty}(0, T; L^2(\Omega)^d)} + {\overline{\delta}} \Norm{\nabla {\dpt^{(m + 1)} u}}{L^2(0, T; L^2(\Omega)^d)}\Big) \\
\nonumber
& \quad + h^{\ell + 1} \Big({|k|} \Norm{\dpt u}{L^1(0, T; H^{\ell + 1}(\Omega))} + {(1 + |k| \ua)} \Norm{\dptt u}{L^1(0, T; H^{\ell + 1}(\Omega))} \\
\label{eq:estimate-projected-error}
& \qquad \qquad + {|k| \Norm{u_1}{L^{\infty}(\Omega)}} \SemiNorm{u_0}{H^{\ell + 1 }(\Omega)}+  (1 + |k| \Norm{u_0}{H^2(\Omega)} ) h \SemiNorm{u_1}{H^{\ell + 1}(\Omega)} 
 \Big).
\end{alignat}
Estimates~\eqref{eq:error-dpt} and~\eqref{eq:error-nabla} then follow using the triangle inequality, and combining estimates~\eqref{eq:estimate-projected-error} and Lemma~\ref{lemma:estimates-Piht}.
\end{proof}

Due to condition~\eqref{eq:proj-Pt-def-1} in the definition of~$\Pt$, the following bound can be used to derive an error estimate of~$\eu$:
\begin{equation*}
\Norm{\eu}{C^{0}([0, T]; L^2(\Omega))} \le \Norm{(\Id - \Rh) u_0}{L^2(\Omega)} + T \Norm{\dpt \eu}{L^{\infty}(0, T; L^2(\Omega))}.
\end{equation*}
We use this estimate in Section~\ref{Sec: error analysis nonlinear} below; cf.~\eqref{est alpht}.

\section{\emph{A priori} analysis of the nonlinear problem} \label{Sec: error analysis nonlinear}

In this section, we prove the well-posedness and \emph{a priori} bounds for the nonlinear problem \eqref{eq:space-time-formulation} by relying on the results of Sections~\ref{sec:stab-linearized-problem} and~\ref{sect:a-priori-linearized}. To streamline the presentation, we introduce the short-hand notation for (semi)norms involved in Theorem~\ref{thm:linear-error-bound}: 
\begin{equation} \label{m ell norms}
	\begin{aligned}
		\Norm{u}{m} := &\,\begin{multlined}[t]\Norm{\Delta {\dpt^{(m)} u}}{L^1(0, T; L^2(\Omega))} + \Norm{{\dpt^{(m + 1)} u}}{L^{\infty}(0, T; L^2(\Omega))} +|k|\Norm{\nabla {\dpt^{(m + 1)} u}}{L^{\infty}(0, T; L^2(\Omega)^d)}\\+ \overline{\delta} \Norm{\nabla {\dpt^{(m + 1)} u}}{L^2(0, T; L^2(\Omega)^d)}, \end{multlined} \\
		\Norm{u}{m, \tau} := &\, \Norm{\Delta {\dpt^{(m)} u}}{L^1(0, T; L^2(\Omega))} + {(|k| + \tau)}\Norm{\nabla {\dpt^{(m + 1)} u}}{L^{\infty}(0, T; L^2(\Omega)^d)}   + \overline{\delta} \Norm{\nabla {\dpt^{(m + 1)} u}}{L^2(0, T; L^2(\Omega)^d)}, \\
		\Norm{u}{\ell} := &\, \begin{multlined}[t] |k| \Norm{\dpt u}{L^1(0, T; H^{\ell + 1}(\Omega))} + (1 + |k| \ua) \Norm{\dptt u}{L^1(0, T; H^{\ell + 1}(\Omega))}  +|k| \Norm{u_1}{L^{\infty}(\Omega)}\Norm{u_0}{H^{\ell + 1 }(\Omega)}\\ + \Norm{\dpt u}{C^0([0, T]; H^{\ell + 1}(\Omega))}+  (1 + |k| \Norm{u_0}{H^2(\Omega)} ) h \SemiNorm{u_1}{H^{\ell + 1}(\Omega)},
        \end{multlined}\\
		\Norm{u}{\ell,h} := &\,\begin{multlined}[t] \Norm{u}{L^{\infty}(0, T; H^{\ell + 1}(\Omega))} + |k| h \Norm{\dpt u}{L^1(0, T; H^{\ell + 1}(\Omega))} 
+ {(1 + |k| \ua)} h \Norm{\dptt u}{L^1(0, T; H^{\ell + 1}(\Omega))} \\
+ {|k| h \Norm{u_1}{L^{\infty}(\Omega)}}\SemiNorm{u_0}{H^{\ell + 1 }(\Omega)}+  (1 + |k| \Norm{u_0}{H^2(\Omega)} ) h \SemiNorm{u_1}{H^{\ell + 1}(\Omega)}.	
		\end{multlined}
	\end{aligned}
\end{equation}
Recall that the variable coefficient~$\alphaht$ in the previous sections can be seen as a place holder for a fixed-point iterate. We next set up this fixed-point mapping to connect the linear results from Section~\ref{sect:a-priori-linearized} to the nonlinear problem \eqref{eq:space-time-formulation}.
\paragraph{The fixed-point operator.} We define the mapping $\mapping:  \ball \ni \alphaht \mapsto \uht$, where $\alphaht$ is taken from
\begin{equation*}
	\begin{aligned}
		\ball =\, \Bigl\{ \alphaht \in \Vht:\
		& \Norm{\dpt (u-\alphaht) }{L^{\infty}(0, T; L^2(\Omega))} 
		\leq \Cballone \left( \tau^m \Norm{u}{m}
		+ h^{\ell + 1} \Norm{u}{\ell}\right), \\
		& \Norm{\nabla (u-\alphaht)}{L^{\infty}(0, T; L^2(\Omega)^d)} 
		\leq \Cballtwo \left(\tau^m \|u\|_{m,\tau}+ h^{\ell} \|u\|_{\ell,h}\right),\\[1mm]
		& \alphaht(\cdot, 0) = \Rh u_0
        \Bigr\},
	\end{aligned}
\end{equation*}
and $\uht$ solves the corresponding linearized problem
\[
\Bhtt(\uht, (\wht, \zh))  = \ell(\wht, \zh) \qquad \forall {(\wht, \zh) \in \Wht \times \Vhp}.
\]
The constants $\Cballone$ and $\Cballtwo$ will be chosen
large enough by the upcoming proofs, but independent
of the discretization parameters and $\delta$.

The analysis of the nonlinear problem will follow by Banach's fixed-point theorem through two steps, where we first prove that $\calT$ is a well-defined self-mapping and then that it is strictly contractive.

\subsection{The self-mapping property} 
In the first step, we focus on proving the self-mapping property.
\begin{proposition}\label{prop: selfmapping}
Let the assumptions of Theorem~\ref{thm:linear-error-bound} hold, and let $\tau \lesssim h^{\frac{d}{2m}}$, $2 \le m \le q$, and~$1 \le \ell \le p$. There exist
	large enough $\Cballone>0$ and $\Cballtwo>0$ (independent of $h$, $\tau$, and $\delta$) in the definition of $\ball$ and small enough $\Cwellpone$ (relative to $\Cballone$), where
	\begin{equation} \label{assumpt selfmapping}
		\begin{aligned}
		\begin{multlined}[t]
        |k|\left(\|u\|_{m}+\Norm{u}{\ell} + \Norm{u}{W^1_{\infty}(0, T; H^{\ell + 1}(\Omega))} + \Norm{\dpt u}{L^1(0, T; L^{\infty}(\Omega))}\right.  \left.+\Norm{\dptt u}{L^1(0, T; L^{\infty}(\Omega))}  \right)\leq \Cwellpone,
        \end{multlined}
		\end{aligned}
	\end{equation}
	such that $\ball$ is non-empty, the mapping $\calT$ is well defined, and $\calT(\ball) \subset \ball$.
\end{proposition}
\begin{proof}
The ball~$\ball$ is non-empty since $\Piht u$ belongs to it as long as 
\begin{equation*}
    \Cballone \geq \Cpione, \quad \Cballtwo \geq \Cpitwo,
\end{equation*}
where~$\Cpione$ and~$\Cpitwo$ are the hidden constants in~\eqref{eq:interpolation-error-dpt} and~\eqref{eq:interpolation-error-nabla}, respectively. \vspace{1mm}

Let $\alphaht \in \ball$. The rest of the statement will follow by Theorem~\ref{thm:linear-error-bound} once we check that the assumptions made there on $\alphaht$ are satisfied. \\[1mm]

\noindent $\bullet$ {\bf Nondegeneracy and boundedness of the variable coefficient}. We first verify Assumption~\ref{asm:nondegeneracy}. It is sufficient to show that
\begin{equation} \label{smallness alphaht}
	|k| \|\alphaht\|_{L^{\infty}(0, T; \Linf)} \leq   |k|\ua <1.
\end{equation}
Then $\la =\ua$ in \eqref{eq:nondegeneracy-assumption}.

We use the triangle inequality, the stability and approximation properties of the interpolant~$\Ih{}$ in Lemmas~\ref{lemma:stab-Ih} and~\ref{lemma:approx-Ih}, and the polynomial inverse estimate~\eqref{eq:inverse-estimate-space-L-infty} to obtain
\begin{alignat*}{3}
\nonumber
\Norm{\alphaht}{L^{\infty}(0, T; \Linf)} 
& \le \Norm{\alphaht - \Ih{} u}{L^{\infty}(0, T; \Linf)} + \Norm{\Ih{} u}{L^{\infty}(0, T; \Linf)}\\
\nonumber
& \le \Cinv h^{-d/2} \Norm{\alphaht - \Ih{} u}{L^{\infty}(0, T; \Ltwo)} + \CstabI \Norm{u}{L^{\infty}(0, T; \Linf)} \\
\nonumber
& \le \Cinv h^{-d/2} \big(\Norm{\alphaht - u}{L^{\infty}(0, T; \Ltwo)} + \Norm{\Ih{} u - u}{L^{\infty}(0, T; \Ltwo)} \big) \\
\nonumber
& \quad + \CstabI \Norm{u}{L^{\infty}(0, T; \Linf)} \\
\nonumber
& \le \Cinv h^{-d/2} \big(\Norm{\alphaht - u}{L^{\infty}(0, T; \Ltwo)} + \CapproxI h^{\ell + 1} \Norm{u}{L^{\infty}(0, T; H^{\ell + 1}(\Omega))} \big) \\
& \quad + \CstabI \Norm{u}{L^{\infty}(0, T; \Linf)}.
\end{alignat*}
By \eqref{est: Rh L2} and the fact that $\alphaht \in \ball$, we have
\begin{alignat}{3}
\nonumber
& h^{-d/2} \Norm{\alphaht - u}{C^0([0, T]; \Ltwo)} \\ \nonumber
& \leq h^{-d/2}(\Norm{(\Id - \Rh) u_0}{L^2(\Omega)} + T \Norm{\dpt(u - \alphaht)}{L^{\infty}(0, T; L^2(\Omega))}) \\
& \leq \begin{multlined}[t] \CapproxRh h^{\ell+1-d/2}|u_0|_{\Hlplusone}+T  h^{-d/2} \Cballone \left(\tau^m \Norm{u}{m}  + h^{\ell + 1}\Norm{u}{\ell} \right). \label{est alpht}
		  \end{multlined}
\end{alignat}
Due to the condition~$\tau \lesssim h^{\frac{d}{2m}}$, and the fact that~$m \geq 2$ and~$\ell \geq 1$,  we have uniform bounds $h^{l+1-d/2} \leq C$ and $\tau^m h^{-d/2} \leq C$. Since $L^{\infty}(0, T; H^{\ell + 1}(\Omega)) \hookrightarrow L^{\infty}(0, T; \Linf)$, the smallness condition in \eqref{smallness alphaht} can be ensured to hold for small enough (relative to $\Cballone$)
\begin{equation*} 
\begin{aligned}
	\begin{multlined}[t]
|k|\left(\SemiNorm{u_0}{\Hlplusone} + \Norm{u}{m} + \Norm{u}{\ell} +\Norm{u}{L^{\infty}(0, T; H^{\ell + 1}(\Omega))}  \right).
\end{multlined}
\end{aligned}
\end{equation*}
\indent Next, we check that condition \eqref{eq:bound-alpha-W-1-infty} also holds; that is, we derive the uniform bound:
\begin{equation} \label{boundedness}
	|k|\Norm{\dpt \alphaht}{L^1(0, T; L^{\infty}(\Omega))} \le \Ca.
\end{equation}
We first have
\begin{equation*}
		|k|\Norm{\dpt \alphaht}{L^1(0, T; L^{\infty}(\Omega))}  
		\leq\, |k|{T}\Norm{\dpt \alphaht}{L^\infty(0, T; L^{\infty}(\Omega))}
\end{equation*}
and we can further bound the right-hand side as follows:
\begin{alignat*}{3}
\Norm{\dpt \alphaht}{L^{\infty}(0, T; \Linf)} 
& \le \Norm{\dpt (\alphaht - \Ih{} u)}{L^{\infty}(0, T; \Linf)} + \Norm{\Ih{} \dpt u}{L^{\infty}(0, T; \Linf)} \\
& \le \Cinv h^{-d/2} \big(\Norm{\dpt (\alphaht - u)}{L^{\infty}(0, T; \Ltwo)} + \Norm{\dpt (u - \Ih{} u)}{L^{\infty}(0, T; \Ltwo)} \big) \\
& \quad + \CstabI \Norm{\dpt u}{L^{\infty}(0, T; \Linf)} \\
& \le \Cinv h^{-d/2} \big(\Norm{\dpt (\alphaht - u)}{L^{\infty}(0, T; L^2(\Omega))} + \CapproxI h^{\ell + 1} \Norm{\dpt u}{L^{\infty}(0, T; H^{\ell + 1}(\Omega)}  \big) \\
& \quad + \CstabI \Norm{\dpt u}{L^{\infty}(0, T; L^{\infty}(\Omega))}.
\end{alignat*}
Moreover, since~$\alphaht \in \ball$, we have
\begin{equation*}
h^{-d/2} \Norm{\dpt (\alphaht - u)}{L^{\infty}(0, T; L^2(\Omega))} \le \Cballone h^{-d/2} \big(\tau^m \Norm{u}{m} + h^{\ell + 1} \Norm{u}{\ell}\big). 
\end{equation*}
Thus, we can guarantee that \eqref{boundedness} holds as long as
\begin{equation*}
	|k|\left(\|u\|_{m}  + \Norm{u}{\ell} + \Norm{\dpt u}{L^{\infty}(0, T; H^{\ell + 1}(\Omega))} \right)
\end{equation*}
is sufficiently small (in other words, $\Cwellpone$ is sufficiently small), relative to $\Cballone$. \\[2mm]

Altogether, we conclude that the assumptions of {Proposition}~\ref{prop:a-priori-bounds-linearized} are satisfied. By {Proposition}~\ref{prop:a-priori-bounds-linearized}, we thus know that the mapping $\calT$ is well defined. Furthermore, by \eqref{eq:error-dpt} and \eqref{eq:error-nabla}, and the fact that $\alphaht \in \ball$, we have
\begin{equation*}
\begin{aligned}
	& \Norm{\dpt \eu}{L^{\infty}(0, T; L^2(\Omega))}\\
    & \qquad \leq \Clinone (\tau^m \Norm{u}{m} + h^{\ell + 1} \Norm{u}{\ell} ) \big(1 + |k|\Cballone \Norm{\dpt u}{L^1(0, T; L^{\infty}(\Omega))} +  |k|\Cballone\Norm{\dptt u}{L^1(0, T; L^{\infty}(\Omega))} \big),
       \end{aligned} 
\end{equation*}
and
\begin{equation*}
\begin{aligned}
	& \Norm{\nabla \eu}{L^{\infty}(0, T; L^2(\Omega)^d)} \\
    & \qquad \leq\Clintwo ( \tau^m \Norm{u}{1,\tau}   + h^{\ell} \Norm{u}{2,h}) \big(1 +  |k|\Cballone\Norm{\dpt u}{L^1(0, T; L^{\infty}(\Omega))} +  |k|\Cballone\Norm{\dptt u}{L^1(0, T; L^{\infty}(\Omega))} \big),
   \end{aligned} 
\end{equation*}
where we recall that~$\Clinone$ and~$\Clintwo$ are the constants in Theorem~\ref{thm:linear-error-bound}.

By~\eqref{assumpt selfmapping}, we have 
\[
\Norm{\dpt u}{L^1(0, T; L^{\infty}(\Omega))} + \Norm{\dptt u}{L^1(0, T; L^{\infty}(\Omega))} \leq \Cwellpone,
\]
uniformly in $\delta$. Then,
{it follows that}~$\uht \in \ball$ as long as
\begin{equation*}
	\begin{aligned}
		\Cballone \geq \Clinone (1+  |k|\Cballone \Cwellpone), \quad \Cballtwo \geq \Clintwo (1+  |k|\Cballone \Cwellpone).
	\end{aligned}
\end{equation*}
In other words, as long as $\Cwellpone$ is small enough so that
\begin{equation*}
\max\{\Clinone, \Clintwo\}\, |k| \Cwellpone <1,
\end{equation*}
and~$\Cballone$, $\Cballtwo$ {are} large enough so that
\begin{equation*}
	\Cballone \geq \frac{\Clinone}{1- \Clinone |k| \Cwellpone}, \quad 	\Cballtwo \geq \frac{\Clintwo}{1- \Clintwo |k| \Cwellpone}.
\end{equation*}
Thus, for properly calibrated constants, we can conclude that $\calT(\ball) \subset \ball$.
\end{proof}
We observe that the condition  $\tau \lesssim h^{\frac{d}{2m}}$ is invoked in the above proof by the need to use an inverse estimate in space to ensure the non-degeneracy of the problem and uniformly bound $\|\partial_t \alphaht\|_{L^\infty(0,T; \Linf)}$.

\begin{remark}[On the smallness of $u$]
    We note that the smallness imposed on $u$ via \eqref{smallness alphaht} (which can be enforced via the smallness of data) could be somewhat relaxed by complementing it with a smallness assumption on $|k|T(h^{\ell+1-d/2}+\tau^m h^{-d/2})$ and $|k|\| u\|_{L^\infty(0,T; \Linf)}$, especially for high-order approximations and sufficiently smooth solutions. 
    This might be relevant for ultrasound applications, where data is often smooth but not necessarily small in higher-order norms; see~\cite[\S4]{kaltenbacher2022parabolic} for further details. We also point out that all the smallness conditions imposed in this work are mitigated by the fact that the parameter $k$ is small in practice as it is proportional to $1/c^2$; see, e.g.,~\cite[Ch.~5]{kaltenbacher2007numerical}.
    \eremk
\end{remark}
\subsection{Strict contractivity of the operator} 
In the next step, we prove contractivity of the mapping. Compared to Proposition~\ref{prop: selfmapping}, we strengthen here the assumption on the lower bound on~$\ell$ to~$\ell \geq d/2$.
\begin{proposition}\label{prop: contractivity}
Let the partitions~$\Th$ and~$\Tt$ be quasi-uniform, and the assumptions of Proposition~\ref{prop: selfmapping} hold. Let also~{$h^{\ell + 1 - \frac{d}{2}} \lesssim \tau \lesssim h^{\frac{d}{2{(m-1)}}}$, $2 \le m \le q$, ${1} \le \ell \le p$}, and~{$\ell + 1 - d/2 \geq d/(2(m-1))$}.
There exists $\Cwellptwo>0$ such that, if 
    \begin{equation*}
        |k|         {T^{\frac32}} \Tnorm{u} \leq \Cwellptwo,
    \end{equation*}
where
\begin{equation} \label{def Tnorm u}
    \begin{aligned}
        \Tnorm{u}=&\, \begin{multlined}[t] 
         \Norm{u}{m}+\Norm{u}{\ell}+\Norm{\dpt u}{L^1(0, T; H^{\ell + 1}(\Omega))}+ \Norm{{\dpt^{(m + 1)} u}}{L^{\infty}(0, T; \Ltwo)}\\+\Norm{\dptt u}{C^0([0, T]; \Linf)}+\Norm{\dpt u}{L^{\infty}(0, T; H^{\ell + 1}(\Omega))}+\Norm{\dpt u}{L^1(0, T; \Linf)}
         \\+\Norm{u}{C^0([0, T]; H^{\ell + 1}(\Omega))} + \Norm{{\dpt^{(m + 1)}u} }{L^2(0, T; L^2(\Omega))},
        \end{multlined}
    \end{aligned}
\end{equation}
then the mapping $\calT$ is strictly contractive.
\end{proposition}
\begin{proof}
Take $\alphahtone$, $\alphahttwo \in \ball$ and denote $\balphaht=\alphahttwo-\alphahtone$. Let $\uhttwo = \calT \alphahttwo$ {and}~$\uhtone= \calT \alphahtone \in \ball$. The difference $\buht= \uhttwo-\uhtone$ then solves
\begin{alignat}{3}
 \nonumber
\Bhtt(\buht, (\wht, 0)) & = - \sum_{n = 1}^N \big(k \dpt \balphaht \dpt \uhtone + k  \balphaht \dpt^2 \uhtone, \wht\big)_{\Qn} \\
\label{diff problem}
& \quad + \sum_{n = 1}^{N - 1} \big( k \balphaht(\cdot, \tn) \jump{\dpt \uhtone}_n, \wht(\cdot, \tn^+) \big)_{\Omega}
            \qquad \forall \wht \in \Wht,
\end{alignat}
where $\alphaht$ in the definition of $\Bhtt{(\cdot, \cdot)}$ in \eqref{eq:Bhtt-def} is replaced by $\alphahttwo$ here.

By defining the following lifting operator~$\mathfrak{L}_{\alpha} : \Vht \rightarrow \Wht$: 
\begin{equation}
	\label{def:L-alpha}
	\int_{\QT} \mathfrak{L}_{\alpha} \uht \wht \dV = \sum_{n = 1}^{N - 1} \big(k \balphaht (\cdot, \tn) \jump{\dpt \uht}_n, \wht(\cdot, \tn^+) \big)_{\Omega} \qquad \forall \wht \in \Wht,
\end{equation}
the variational problem~\eqref{diff problem} can be written as
\begin{equation*}
	\begin{aligned}
		\begin{multlined}[t]	\Bhtt(\buht, (\wht, 0)) = - \big(k \dpt \balphaht \dpt \uhtone+k  \balphaht \dpt^2 \uhtone, \wht\big)_{\QT}
			+ \big(\mathfrak{L}_{\alpha} \uhtone, \wht\big)_{\QT}.
		\end{multlined}
	\end{aligned}
\end{equation*}
We can then rely on an analogous stability estimate to~\eqref{eq:strong-continuous-dependence} with zero initial and Dirichlet boundary data,  $\mu= 0$, and~$\xi=0$; in other words,
 \begin{equation} \label{est delta norm}
	\Tnormd{\buht} \lesssim \sum_{n = 1}^{N} \Norm{f}{L^1(\In; L^2(\Omega))},
\end{equation} 
where $f =  - k  \balphaht \dpt^2 \uhtone-k \dpt \balphaht \dpt \uhtone +\mathfrak{L}_{\alpha} \uhtone$ (recall the definition of $\Tnormd{\cdot}$ in \eqref{eq:energy-norm-delta}).\\
\indent Toward estimating $f$, we note that
\begin{alignat*}{3}
\nonumber
|k|\sum_{n = 1}^N \Norm{\balphaht \dptt \uhtone}{L^1(0, T; \Ltwo)}
	& \leq |k|\sum_{n = 1}^N  \Norm{\balphaht}{L^{\infty}(\In; \Ltwo)} \Norm{\dptt \uhtone}{L^1(\In; \Linf)}\\
	& \leq |k| \Norm{\balphaht}{C^0([0, T]; \Ltwo)} \sum_{n = 1}^N \Norm{\dptt \uhtone}{L^1(\In; \Linf)}.
\end{alignat*}

 Further, due to~$\balphaht(\cdot, 0) = 0$, we have
\begin{equation*}
\Norm{\balphaht}{C([0,T]; \Ltwo)} =\left \| \int_0^t \dpt \balphaht \ds \right \|_{L^\infty(0,T; \Ltwo)} \leq T \Norm{\dpt \balphaht}{L^{\infty}(0, T; \Ltwo)}. 
\end{equation*}
Next, using the triangle inequality, the polynomial inverse estimates~\eqref{eq:inverse-estimate-time} and~\eqref{eq:inverse-estimate-space-L-infty}, and the stability and approximation properties of~$\Pt$ and~$\Ih$ in Lemmas~\ref{lemma:stab-Pt}, \ref{lemma:estimates-Pt}, \ref{lemma:stab-Ih}, and~\ref{lemma:approx-Ih}, we obtain
\begin{equation} \label{bound utt discrete}
\begin{aligned}
\sum_{n = 1}^N & \Norm{\dptt \uhtone}{L^1(\In; \Linf)} \\
& \le \sum_{n = 1}^N \Norm{\dptt \uhtone - \dpt (\Pt \circ \Ih{}) \dpt u}{L^1(\In; \Linf)} + \Norm{\dpt (\Pt \circ \Ih{}) \dpt u}{L^1(0, T; \Linf)} \\
& \le \Cinv^2 \sum_{n = 1}^N \tau_n^{-1} h^{-d/2} \Norm{\dpt \uhtone - (\Pt \circ \Ih) \dpt u}{L^1(\In; L^2(\Omega))} \\
& \quad + \CstabPt \Norm{\Ih \dptt u}{C^0([0, T]; \Linf)} \\
& \le \Cinv^2 \sum_{n = 1}^N \tau_n^{-1} h^{-d/2} \Big(\Norm{\dpt (\uhtone - u)}{L^1(\In; \Ltwo)} + \Norm{\dpt u - \Ih \dpt u}{L^1(\In; \Ltwo)} \\
& \quad + \Norm{\Ih (\dpt u - \Pt \dpt u)}{L^1(\In; \Ltwo)} \Big) + \CstabPt \CstabI \Norm{\dptt u}{C^0([0, T]; \Linf)} \\
& \le \Cinv^2 \tau^{-1} h^{-d/2} T \Norm{\dpt( \uhtone - u)}{L^{\infty}(0, T; L^2(\Omega))}  + \Cinv^2 \tau^{-1} h^{\ell + 1 - d/2} \Norm{\dpt u}{L^1(0, T; H^{\ell + 1}(\Omega))} \\
& \quad + \Cinv^2 \CstabI \tau^{-1} h^{-d/2} \Norm{\dpt u - \Pt \dpt u}{L^1(\In; L^2(\Omega))} + \CstabPt \CstabI \Norm{\dptt u}{C^0([0, T]; \Linf)} \\
& \le \Cinv^2 \tau^{-1} h^{-d/2} T \Cballone \big(\tau^m \Norm{u}{m} + h^{\ell + 1} \Norm{u}{\ell} \big) + \Cinv^2 \tau^{-1} h^{\ell + 1 - d/2} \Norm{\dpt u}{L^1(0, T; H^{\ell + 1}(\Omega))} \\
& \quad + \Cinv^2 \CstabI \tau^{m - 1} h^{-d/2} \Norm{{\dpt^{(m + 1)} u}}{L^{\infty}(0, T; \Ltwo)} + \CstabPt \CstabI \Norm{\dptt u}{C^0([0, T]; \Linf)},
\end{aligned}
\end{equation}
where we have used the quasi-uniformity of~$\Tt$.

The assumption~$h^{\ell + 1 - \frac{d}2} \lesssim \tau \lesssim h^{\frac{d}{2{(m-1)}}}$ implies the existence of a positive constant~$C$, independent of~$h$ and~$\tau$ such that

\begin{equation}
\label{eq:tau-h-relation}
    \tau^{m - 1} h^{-\frac{d}{2}}
    \le C \quad \text{ and } \quad \tau^{-1} h^{\ell + 1 - d/2} \le C.
\end{equation}
Therefore, we have
\begin{equation*}
\begin{aligned}
&|k| \sum_{n = 1}^N \Norm{\balphaht \dptt \uhtone}{L^1(0, T; \Ltwo)}\\
& \qquad \lesssim \,\begin{multlined}[t] |k| T (\Norm{u}{m}+\Norm{u}{\ell}+\Norm{\dpt u}{L^1(0, T; H^{\ell + 1}(\Omega))} + \Norm{{\dpt^{(m + 1)} u}}{L^{\infty}(0, T; \Ltwo)}\\
+ \Norm{\dptt u}{C^0([0, T]; \Linf)}) \Norm{\balphaht}{L^{\infty}(0, T; \Ltwo)}.
\end{multlined}
\end{aligned}
\end{equation*}

The second term within $f$ can be treated as follows:
\begin{alignat*}{3}
|k| \Norm{\dpt \balphaht \dpt \uhtone}{L^1(0, T; \Ltwo)} & \le |k| \Norm{\dpt \balphaht}{L^{\infty}(0, T; \Ltwo)} \Norm{\dpt \uhtone}{L^1(0, T; \Linf)}.
\end{alignat*}
We can reason that $\| \dpt \uhtone\|_{L^\infty(0,T; L^\infty(\Omega))}$ can be made small enough by reducing $u$ in suitable norms similarly to above using the fact that $\uhtone \in \ball$:
\begin{alignat*}{3}
\Norm{\dpt \uhtone}{L^1(0,T; \Linf)} & \le \Norm{\dpt \uhtone - \dpt \Ih u}{L^1(0, T; \Linf)} + \Norm{\Ih \dpt u}{L^1(0, T; \Linf)} \\
& \le \Cinv h^{-\frac{d}{2}} T \Norm{\dpt \uhtone - \dpt \Ih u}{L^{\infty}(0, T; \Ltwo)} + \CstabI \Norm{\dpt u}{L^1(0, T; \Linf)} \\
& \le \Cinv h^{-\frac{d}{2}} T \big(\Norm{\dpt \uhtone - \dpt u}{L^{\infty}(0, T; \Ltwo)} + \Norm{\dpt u - \Ih \dpt u}{L^{\infty}(0, T; \Ltwo)} \big) \\
& \quad + \CstabI \Norm{\dpt u}{L^1(0, T; \Linf)} \\
& \le \Cinv h^{-\frac{d}{2}} T \Big(\Cballone (\tau^m \Norm{u}{m} + h^{\ell + 1} \Norm{u}{\ell}) + h^{\ell + 1} \Norm{\dpt u}{L^{\infty}(0, T; H^{\ell + 1}(\Omega))}\Big) \\
& \quad + \CstabI \Norm{\dpt u}{L^1(0, T; \Linf)}.
\end{alignat*}
Thus,
\begin{equation*}
	\begin{aligned}
        &|k| \Norm{\dpt \balphaht \dpt \uhtone}{L^1(0, T; \Ltwo)} \\
		& \qquad \leq C |k| 
        T (\Norm{u}{m}+\Norm{u}{\ell}+\Norm{\dpt u}{L^{\infty}(0, T; H^{\ell + 1}(\Omega))}+\Norm{\dpt u}{L^1(0, T; \Linf)}) \Norm{\dpt \balphaht }{L^{\infty}(0, T; \Ltwo)}.
	\end{aligned}
\end{equation*}
\indent It remains to estimate the $\mathfrak{L}_{\alpha}$ term:
\[
\|\mathfrak{L}_{\alpha} \uhtone\|_{L^1(0,T; \Ltwo)} \leq \sqrt{T} \|\mathfrak{L}_{\alpha} \uhtone\|_{L^2(0,T; \Ltwo)}.
\]
Taking~$\wht = \mathfrak{L}_{\alpha} \uhtone$ in~\eqref{def:L-alpha}, and using the inverse estimates in Lemma~\ref{lemma:inverse-estimate}, the trace inequalities in Lemma~\ref{lemma:trace-inequality}, and the continuity in time of~$\dpt u$, we obtain
\begin{alignat*}{3}
\Norm{\mathfrak{L}_{\alpha} \uhtone}{L^2(\QT)}^2 
& = \sum_{n = 1}^{N - 1} \big(k \balphaht(\cdot, \tn) \jump{\dpt \uhtone}_n, \mathfrak{L}_{\alpha} \uhtone (\cdot, \tn^+) \big)_{\Omega} \\
& \le |k| \sum_{n = 1}^{N - 1} \Norm{\balphaht(\cdot, \tn) \jump{\dpt \uhtone}_n}{L^2(\Omega)} \Norm{\mathfrak{L}_{\alpha} \uhtone(\cdot, \tn^+)}{L^2(\Omega)} \\
& \le |k| {T} \Norm{{\dpt \balphaht}}{L^{\infty}(0, T; L^2(\Omega))} \sum_{n = 1}^{N - 1} \Norm{\jump{\dpt \uhtone}_n}{L^{\infty}(\Omega)} \Norm{\mathfrak{L}_{\alpha} \uht(\cdot, \tn^+)}{L^2(\Omega)} \\
& \le \Cinv |k|{T} \Norm{{\dpt \balphaht}}{L^{\infty}(0, T; L^2(\Omega))} \sum_{n = 1}^{N - 1} h^{-\frac{d}{2}} \Norm{\jump{\dpt \uhtone}_n}{L^2(\Omega)} \Norm{\mathfrak{L}_{\alpha} \uht(\cdot, \tn^+)}{L^2(\Omega)} \\
& \le \Cinv |k| {T} \Norm{{\dpt \balphaht}}{L^{\infty}(0, T; L^2(\Omega))} \sum_{n = 1}^{N - 1} h^{-\frac{d}{2}} \Big(\Norm{\jump{\dpt (\uhtone -  \Piht u)}_n}{L^2(\Omega)} \\
& \qquad + \Norm{\jump{\dpt (\Piht u - u)}_n}{L^2(\Omega)}\Big) \Norm{\mathfrak{L}_{\alpha} \uhtone (\cdot, \tn^+)}{L^2(\Omega)} \\
& \le \Ctr^{\star} \Cinv |k| {T} \Norm{{\dpt \balphaht}}{L^{\infty}(0, T; L^2(\Omega))} \tau^{-\frac12} h^{-\frac{d}{2}} \sum_{n = 1}^{N - 1}  \Big( \Norm{\jump{\dpt (\uhtone - \Piht u)}_n}{L^2(\Omega)} \\
& \qquad +  \Norm{\jump{\dpt (\Piht u - u)}_n}{L^2(\Omega)} \Big) \Norm{\mathfrak{L}_{\alpha} \uht}{L^2(Q_{n + 1})} \\
& \le \sqrt{2} \Ctr^{\star} \Cinv |k| {T} \Norm{{\dpt \balphaht}}{L^{\infty}(0, T; L^2(\Omega))} \tau^{-\frac12} h^{-\frac{d}{2}} \Big( \big|\dpt (\uhtone - \Piht u) \big|_{\sf J}  \\
& \qquad + \Ctr \Big(\sum_{n = 1}^{N} \tau_n^{-1} \Norm{\dpt(\Piht u - u)}{L^2(\Qn)}^2 + \tau_n \Norm{\dptt (\Piht u - u)}{L^2(\Qn)}^2 \Big)^{\frac12} \Big) \Norm{\mathfrak{L}_{\alpha} \uht}{L^2(\QT)}.
\end{alignat*}
Thus,
\begin{alignat}{3}
\nonumber
      \Norm{\mathfrak{L}_{\alpha} \uht}{L^1(0, T; \Ltwo)} &  \le \sqrt{2} \Ctr^{\star} \Cinv |k| T^{\frac32} \Norm{\dpt \balphaht}{L^{\infty}(0, T; L^2(\Omega))} \tau^{-\frac12} h^{-\frac{d}{2}} \Big\{ \big|\dpt (\uhtone - \Piht u) \big|_{\sf J} \\
    \label{est L alpha}
& \quad + \Ctr \Big(\sum_{n = 1}^{N} \tau_n^{-1} \Norm{\dpt(\Piht u - u)}{L^2(\Qn)}^2 + \tau_n \Norm{\dptt (\Piht u - u)}{L^2(\Qn)}^2 \Big)^{\frac12} \Big\}.
\end{alignat}
Further, using~\eqref{eq:estimate-projected-error}, we have
\begin{equation*}
    \begin{aligned}
        \big|\dpt (\uhtone - \Piht u) \big|_{\sf J}  \lesssim&\, \tau^m \Norm{u}{m} 
  \big(1 + |k|\Cballone \Norm{\dpt u}{L^1(0, T; L^{\infty}(\Omega))} + |k|\Cballone\Norm{\dptt u}{L^1(0, T; L^{\infty}(\Omega))}\big)\\
& \quad + h^{\ell + 1} \Norm{u}{\ell}
 \big(1 + |k|\Cballone\Norm{\dpt u}{L^1(0, T; L^{\infty}(\Omega))} + |k|\Cballone\Norm{\dptt u}{L^1(0, T; L^{\infty}(\Omega))} \big),
    \end{aligned}
\end{equation*}
which, combined with~\eqref{eq:combined-interpolation-error}, \eqref{est L alpha}, and the assumption~$h^{\ell + 1 - d/2} \lesssim \tau \lesssim h^{d/(2m)}$, gives
\begin{equation*}
\begin{aligned}
\Norm{\La \uht}{L^1(0, T; \Ltwo)} 
\lesssim&\, |k| T^{\frac32} \Bigl\{ \Norm{u}{m} 
  \big(1 + |k|\Cballone \Norm{\dpt u}{L^1(0, T; L^{\infty}(\Omega))} + |k|\Cballone \Norm{\dptt u}{L^1(0, T; L^{\infty}(\Omega))}\big) \\
  &+  \Norm{u}{\ell}(1 + |k|\Cballone \Norm{\dpt u}{L^1(0, T; L^{\infty}(\Omega))} + |k|\Cballone \Norm{\dptt u}{L^1(0, T; L^{\infty}(\Omega))} \big)\\
  & +\Norm{u}{C^0([0, T]; H^{\ell + 1}(\Omega))} + \Norm{{\dpt^{(m + 1)} u}}{L^2(0, T; L^2(\Omega))}\Bigr\} \Norm{\dpt \balphaht}{L^{\infty}(0, T; \Ltwo)}.
\end{aligned}
\end{equation*}

By combining the derived estimates with \eqref{est delta norm}, we arrive at
\begin{equation*}
\begin{aligned}
\Tnormd{\buht} 
    \lesssim & \, |k| T^{\frac32} \Tnorm{u} (1 + |k|\Cballone \Norm{\dpt u}{L^1(0, T; L^{\infty}(\Omega))} + |k|\Cballone \Norm{\dptt u}{L^1(0, T; L^{\infty}(\Omega))})  \Tnormd{\balphaht}\\
    \lesssim&\, (1 + |k|\Cballone \Norm{\dpt u}{L^1(0, T; L^{\infty}(\Omega))} + |k|\Cballone \Norm{\dptt u}{L^1(0, T; L^{\infty}(\Omega))}) \Cwellptwo \Tnormd{\balphaht}.
   \end{aligned} 
\end{equation*} 
We can thus guarantee strict contractivity of the mapping by reducing $\Cwellptwo$ (relative to $\Cballone$).
\end{proof}
The above proof of in Proposition~\ref{prop: contractivity} involves estimating $\sum_{n = 1}^N  \Norm{\dptt \uhtone}{L^1(\In; \Linf)} $ in \eqref{bound utt discrete}. As this second-derivative term is not present in the discrete energy, we are forced to use inverse estimates in both time and space.  This invokes the assumption ~$h^{\ell + 1 - \frac{d}{2}} \lesssim \tau \lesssim h^{\frac{d}{2{(m-1)}}}$ (strengthened from $\tau \lesssim h^{\frac{d}{2m}}$ needed for proving the self-mapping property). In the discretization approaches for nonlinear acoustic models that employ finite-element discretization in space with standard time-stepping schemes in~\cite{Dorich_Nikolic:2024, Nikolic:2023} a similar issue is circumvented by considering time-differentiated (semi)discrete problems; this approach, however, does not seem easily transferrable to the present setting. \\
\indent We next prove the well-posedness and accuracy of the nonlinear discrete problem. 
\begin{theorem} [\emph{A priori} estimates for the discrete Westervelt equation] \label{thm: h tau conv nonlinear}
	Let the assumptions of Propositions~\ref{prop: selfmapping} and~\ref{prop: contractivity} hold with~$h^{\ell + 1 - \frac{d}{2}} \lesssim \tau \lesssim h^{\frac{d}{2{(m-1)}}}$, $2 \le m \le q$, ${1} \le \ell \le p$, and~{$\ell + 1 - d/2 \geq d/(2(m-1))$}. There exist
	 $\Cballone>0$ and $\Cballtwo>0$, independent of $h$, $\tau$, and $\delta$, and small enough $\Cwellp$ (relative to $\Cballone$) such that, if the exact solution satisfies
     \begin{equation*}
         |k| \Tnorm{u} \leq \Cwellp
     \end{equation*}
(where $\Tnorm{\cdot}$ is defined in \eqref{def Tnorm u}), then problem~\eqref{eq:space-time-formulation}, that is,
 \begin{equation*} 
	\Bht(\uht, (\wht, \zh)) = \mathcal{L}(\wht, \zh) \quad \forall (\wht, \zh) \in \Wht \times \Vhp,
\end{equation*}
 has a unique solution $\uht \in \Vht$, which satisfies
 \begin{equation*}
     \begin{aligned}
         \Norm{\dpt (u-\uht) }{L^{\infty}(0, T; L^2(\Omega))} 
		\leq&\, \Cballone \left( \tau^m \Norm{u}{m}
		+ h^{\ell + 1} \Norm{u}{\ell}\right), \\
		 \Norm{\nabla (u-\uht)}{L^{\infty}(0, T; L^2(\Omega)^d)} 
		\leq&\, \Cballtwo \left(\tau^m \|u\|_{m,\tau}+ h^{\ell} \|u\|_{\ell,h}\right).
     \end{aligned}
 \end{equation*}
\end{theorem}
\begin{proof}
   The statement follows by Propositions~\ref{prop: selfmapping} and~\ref{prop: contractivity} and Banach's fixed-point theorem.
\end{proof}

\begin{remark}[Assumption on~$\tau$]
The assumption~$h^{\ell + 1 - \frac{d}{2}} \lesssim \tau \lesssim h^{\frac{d}{2{(m-1)}}}$ determines just the scaling of~$\tau$, but it is not a CFL condition in the classical sense, as it is only necessary that the hidden constants are independent of~$h$, $\tau$, and~$\delta$.
For~$d = 3$ and the minimal regularity~$\ell = 2$ and~$m = 2$, such an assumption becomes~${\tau \simeq h^{\frac{3}{{2}}}}$.
Moreover,
it relaxes for higher approximations in space and time and sufficiently smooth solutions. Note however, that the hidden constants in $h^{\ell + 1 - \frac{d}{2}} \lesssim \tau \lesssim h^{\frac{d}{2{(m-1)}}}$  influence the smallness needed from the exact data; see, e.g., \eqref{est alpht} and \eqref{bound utt discrete}.

We also note that, since the need for the condition $h^{\ell + 1 - d/2} \lesssim \tau$ arises in the proof of contractivity, we expect that it might still be possible to show existence of accurate solutions without this resctriction by resorting to Schauder's fixed-point theorem instead of Banach's theorem; see, e.g.,~\cite[Thm. 4.1]{Nikolic:2023} for similar arguments used in the analysis of a finite element semi-discretization of Westervelt's equation with nonlocal damping.
\eremk
\end{remark}

\section{The vanishing dissipation parameter analysis} \label{sec:vanishing delta analysis}
In this section, we discuss the behavior of discrete solutions as $\delta \searrow 0$. To emphasize the dependence on $\delta$, we use here $u^{(\delta)}$ to denote the exact solution to the Westervelt IBVP for $\delta \in (0,\overline{\delta}]$ and $u^{(0)}$ to denote the exact solution to the inviscid problem. It is known that, as the dissipation vanishes, the exact dissipative solutions converge in the energy norm to the exact solution to the inviscid problem at a linear rate; see~\cite[Thm.~5.1]{kaltenbacher2022parabolic}. In this section, we establish sufficient conditions under which this asymptotic behavior is preserved in the present discrete setting.\\
\indent Consider the family $\{\uhtdelta\}_{\delta \in (0,\bdelta]}$ of solutions to
\begin{equation} \label{nonlinear delta problem}
	\Bht(\uhtdelta, (\wht, \zh)) = \mathcal{L}(\wht, \zh) \quad \forall (\wht, \zh) \in \Wht \times \Vhp.
\end{equation}
Denote the discrete solution corresponding to $\delta=0$ (that is, of the inviscid discrete problem) by $\uhtzero$. We have the following result that relates the two discrete problems.
\begin{theorem} \label{thm: delta conv nonlinear}
    Let the assumptions of Theorem~\ref{thm: h tau conv nonlinear} hold and let $\delta \in (0,\bdelta]$. 
    {Let also~$\tau \lesssim h^{\vartheta}$ with~$\vartheta = \max\{\frac{d}{2(m-1)},\, \frac{1}{m}\}$.}
    Then the family of discrete solutions $\{\uhtdelta\}_{\delta \in (0,\bdelta]}$ of \eqref{nonlinear delta problem} converges to the solution to the discrete inviscid problem at a linear rate {as follows}:
\begin{equation*} 
\begin{aligned}
\begin{multlined}[t]\Norm{\dpt (\uhtdelta-\uhtzero)}{L^{\infty}(0, T; L^2(\Omega))} +  \Norm{\nabla (\uhtdelta-\uhtzero)}{L^{\infty}(0, T; L^2(\Omega)^d)} + \SemiNorm{\dpt (\uhtdelta-\uhtzero)}{\sf J} \\ \hspace*{4.5cm}
+ \Norm{\nabla (\uhtdelta-\uhtzero)}{L^2(\ST)^d}  \leq C \delta.
\end{multlined}
\end{aligned}
\end{equation*}
\end{theorem}
\begin{proof}
Let $\buht=\uhtdelta-\uhtzero$. Since $(\nabla \buht(\cdot, 0), \nabla \zh)_{\Omega}=0$ and $\big( k \buht \dpt \uhtzero, \wht \big)_{\SO}=0$, this difference solves
\begin{equation}
\label{eq:aux-delta-zero}
	\begin{aligned}
		\begin{multlined}[t]	\Bhtt(\buht, (\wht, 0)) = \delta \big(\Delta_h \dpt \uhtdelta, \wht \big)_{\QT} - \big(k \dpt \buht \dpt \uhtzero + k  \buht \dpt^2 \uhtzero, \wht \big)_{\QT}
			\\ + \big(\La \uhtzero, \wht\big)_{\QT} \qquad {\forall \wht \in \Wht},
		\end{multlined}
	\end{aligned}
\end{equation}
where~$\Delta_h: H_0^1(\Omega) \to \Vhp$
is the discrete Laplacian operator, defined for any~$v \in H_0^1(\Omega)$ as follows: 
\[
(-\Delta_h v, \zh)_{\Omega} = (\nabla v, \nabla \zh)_{\Omega} \qquad   \forall \zh \in \Vhp.
\]
The last two terms on the right-hand side of~\eqref{eq:aux-delta-zero} can be treated as in the proof of Proposition~\ref{prop: contractivity}.  
Therefore, we obtain
\begin{equation*}
    \Tnormd{\buht} \lesssim \delta \Norm{\Delta_h \uhtdelta}{L^1(0,T; \Ltwo)},
\end{equation*}
where $C>0$ does not depend on $\delta$, $h$ or~$\tau$. The right-hand side term above remains to be estimated. To treat it, we use the following inverse estimate: 
\begin{equation*}
    \Norm{\Delta_h v}{\Ltwo} \lesssim h^{-1} \Norm{\nabla v}{\Ltwod} \qquad \forall v \in H_0^1(\Omega),
\end{equation*}
together with the stability bound~$\Norm{\Delta_h v}{\Ltwo} \le \Norm{\Delta v}{\Ltwo}$ for all~$v \in H_0^1(\Omega) \cap H^2(\Omega)$. Thus,
\begin{alignat*}{3}
\delta \Norm{\Delta_h \uhtdelta}{L^1(0, T; \Ltwo)} & \le \delta T \Norm{\Delta_h \uht}{L^{\infty}(0, T; \Ltwo)} \\
& \lesssim \delta T \big(\Norm{\Delta_h (\uhtdelta - u^{(\delta)})}{L^{\infty}(0, T; \Ltwo)} + \Norm{\Delta_h u^{(\delta)}}{L^{\infty}(0, T; \Ltwo)} \big) \\
& \lesssim \delta T \big(h^{-1} \Norm{\nabla (\uhtdelta - u^{(\delta)}}{L^{\infty}(0, T; \Ltwo^d)} + \Norm{\Delta u^{(\delta)}}{L^{\infty}(0, T; \Ltwo)} \big).
\end{alignat*}
Since $\uhtdelta \in \ball$, we then have
\begin{alignat*}{3}
\delta \Norm{\Delta_h \uhtdelta}{L^1(0, T; \Ltwo)} \lesssim \delta T {h^{-1}}\big(\tau^m \Norm{u^{(\delta)}}{m, \tau} + h^{\ell} \Norm{u^{(\delta)}}{\ell, h} + \Norm{u^{(\delta)}}{L^{\infty}(0, T; H^2(\Omega))}\big),
\end{alignat*}
where the hidden constant does not depend on~$h$, $\tau$, or~$\delta$. The statement then follows by {using the assumption~$\tau \lesssim h^{1/m}$ and} noting that
    \begin{equation*}
\begin{aligned}
\begin{multlined}[t]\Norm{\dpt \buht}{L^{\infty}(0, T; L^2(\Omega))} +  \Norm{\nabla \buht}{L^{\infty}(0, T; L^2(\Omega)^d)} + \SemiNorm{\dpt \buht}{\sf J} 
+ \Norm{\nabla \buht}{L^2(\ST)^d}  \lesssim \Tnormd{\buht}.
\end{multlined}
\end{aligned}
\end{equation*}
\end{proof}
{We note that $\tau \lesssim h^{1/m}$ is a more restrictive assumption than $\tau \lesssim h^{\frac{d}{2(m-1)}}$ only in a one-dimensional setting, that is, for $d=1$. The present theory is, however, not optimized for that setting anyway, as there we can use, for example, the embedding $H^1(\Omega) \hookrightarrow L^\infty(\Omega)$ in place of inverse estimates in space in Section~\ref{Sec: error analysis nonlinear}.}
\section{Numerical experiments} \label{sec:numerical-results}
In this section, we validate our theoretical results on the~$(h, \tau)$-convergence in Theorem~\ref{thm: h tau conv nonlinear}, and the~$\delta$-convergence in Theorem~\ref{thm: delta conv nonlinear}. Moreover, we numerically assess the convergence of the~$(p, q)$-version of the proposed method.

An object-oriented MATLAB implementation of the space--time method in~\eqref{eq:space-time-formulation} for~$(2 + 1)$-dimensional problems
was developed to carry out the numerical experiments\footnote{The code used for the numerical tests are available in the GitHub repository~\cite{Gomez:code}.}.

\subsection{Implementation details}
Since no continuity in time is required in the test space~$\Wht$, the space--time formulation~\eqref{eq:space-time-formulation} reduces to solving the following sequence of nonlinear systems of equations on each time slab: 
\begin{itemize}
\item For~$n = 1$, $\uht(\cdot, 0) = \Rh u_0$ and
\begin{equation}
\label{eq:nonlinear-systems-1}
\begin{split}
\big(\dpt (( 1 + k \uht) & \dpt \uht), \wht\big)_{Q_1} + \big((1 + k \uht) \dpt \uht, \wht\big)_{\SO} \\
+ c^2 & (\nabla \uht, \nabla \wht)_{Q_1} 
+ \delta(\nabla \dpt \uht, \nabla \wht)_{Q_1}\\
& \qquad = (f, \wht)_{Q_1} + ((1 + k u_0) u_1, \wht(\cdot, 0))_{\Omega} \quad \forall \wht \in \Pp{q-1}{I_1; \Vhp}.
\end{split}
\end{equation}
\item For~$n = 2, \ldots, N$, $\uht(\cdot, \tnmo^+) = \uht(\cdot, \tnmo^-)$ and
\begin{equation}
\label{eq:nonlinear-systems-2}
\begin{split}
& \big(\dpt (( 1 + k \uht) \dpt \uht), \wht\big)_{\Qn} + \big((1 + k \uht(\cdot, \tnmo)) \dpt \uht(\cdot, \tnmo^+), \wht(\cdot, \tnmo^+)\big)_{\Omega} \\
& + c^2 (\nabla \uht, \nabla \wht)_{\Qn} 
+ \delta(\nabla \dpt \uht, \nabla \wht)_{\Qn}\\
& = (f, \wht)_{\Qn} + \big((1 + k \uht(\cdot, \tnmo)) \dpt \uht(\cdot, \tnmo^-), \wht(\cdot, \tnmo^+)\big)_{\Omega} \quad \forall \wht \in \Pp{q-1}{I_n; \Vhp}.
\end{split}
\end{equation}
\end{itemize}

Moreover, inspired by the {linearized fixed-point iteration}
proposed in~\cite[\S5.4.2]{kaltenbacher2007numerical}  (used there to solve the nonlinear problems stemming from the Newmark scheme applied to the Westervelt equation), we approximate the solution to the nonlinear systems of equations~\eqref{eq:nonlinear-systems-2}
using the following fixed-point iteration:
for all~$\wht \in \Pp{q-1}{I_n; \Vhp}$,
\begin{equation}
\label{eq:nonlinear-solver}
\begin{split}
& \big(\dptt \uht^{(s + 1)}, \wht\big)_{\Qn} + \big(\dpt \uht^{(s + 1)}(\cdot, \tnmo^+), \wht(\cdot, \tnmo^+)\big)_{\Omega} \\
& + c^2 (\nabla \uht^{(s + 1)}, \nabla \wht)_{\Qn} 
+ \delta(\nabla \dpt \uht^{(s + 1)}, \nabla \wht)_{\Qn}\\
& \qquad = (f, \wht)_{\Qn} + \big((1 + k \uht(\cdot, \tnmo)) \dpt \uht(\cdot, \tnmo^-), \wht(\cdot, \tnmo^+)\big)_{\Omega} \\
& \qquad\quad  - k \big(\dpt(\uht^{(s)} \dpt \uht^{(s)}), \wht\big)_{\Qn} - k \big(\uht(\cdot, \tnmo) \dpt \uht^{(s)}(\cdot, \tnmo^+), \wht(\cdot, \tnmo^+) \big)_{\Omega},
\end{split}
\end{equation}
whose associated linear system of equations has a unique solution due to Remark~\ref{rem:well-posedness-linearized-problem}. The nonlinear system~\eqref{eq:nonlinear-systems-1} is solved analogously.

In the numerical experiments below, for the fixed-point iteration~\eqref{eq:nonlinear-solver}, we use as initial guess~$\uht^{(0)}$ the solution to problem~\eqref{eq:nonlinear-systems-2} for~$k = 0$. Moreover, we set a maximum number of linear iterations~$s_{\max} = 15$ and a tolerance~$tol = 10^{-12}$.

\subsection{\texorpdfstring{$h$}{h}-convergence\label{sect:h-convergence}}
We consider a manufactured~$(2 + 1)$-dimensional problem on the space--time domain~$\QT = (0, 1)^2 \times (0, 1)$ with~$k = 0.5$, $c = 1$, $\delta = 6\times 10^{-9}$, homogeneous Dirichlet boundary conditions, and initial data and source term chosen so that the exact solution to~\eqref{eq:Westervelt-IBVP} is given by
\begin{equation}
\label{eq:smooth-sol}
u(x, y, t) = A \sin(\omega t) \sin(\ell x) \sin (\ell y),
\end{equation}
where~$A = 10^{-2}$, $\omega = \pi/3$, and~$\ell = \pi$.

In order to numerically assess the~$h$-convergence rates of the proposed scheme, we use approximations in time of degree~$q = 5$ and a fixed partition of the time interval with uniform time step~$\tau = 2\times 10^{-1}$, so as to let the spatial error dominate. 
In Figure~\ref{fig:h-convergence}, we show (in \emph{log-log} scale) the results obtained for a set of structured simplicial meshes and approximations of degree~$p = 1, 2, 3$. Optimal convergence rates of order~$\mathcal{O}(h^{p})$ and~$\mathcal{O}(h^{p + 1})$ are observed, respectively, for the following errors:
\begin{equation}
\label{eq:errors-numerical-exp}
\Norm{\dpt(u - \uht)}{L^{\infty}(0, T; L^2(\Omega))} \quad \text{ and } \quad \Norm{\nabla(u - \uht)}{L^{\infty}(0, T; L^2(\Omega)^2)},
\end{equation}
which are approximated by taking the maximum of the errors measured on uniformly distributed points on each time interval. 
\begin{figure}[!htb]
\centering
\includegraphics[width = 2.8in]{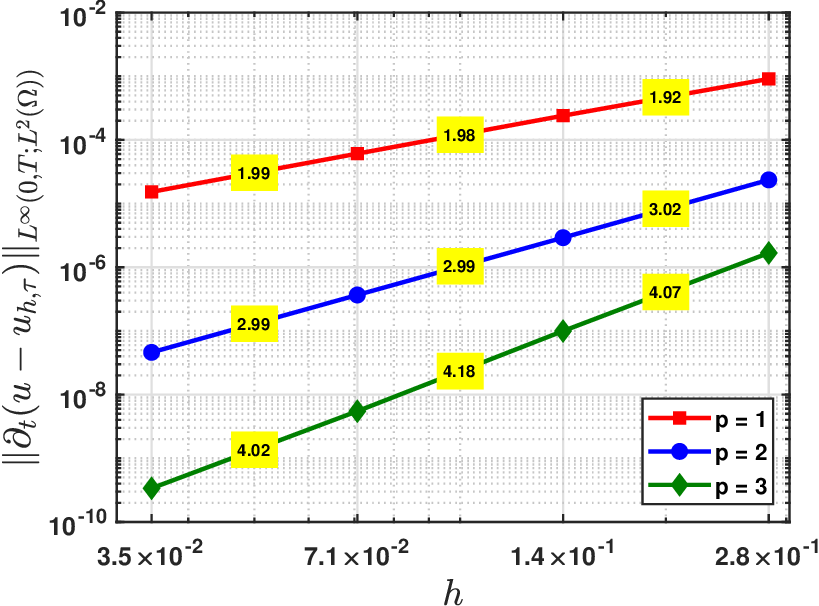}
\hspace{0.3in}
\includegraphics[width = 2.8in]{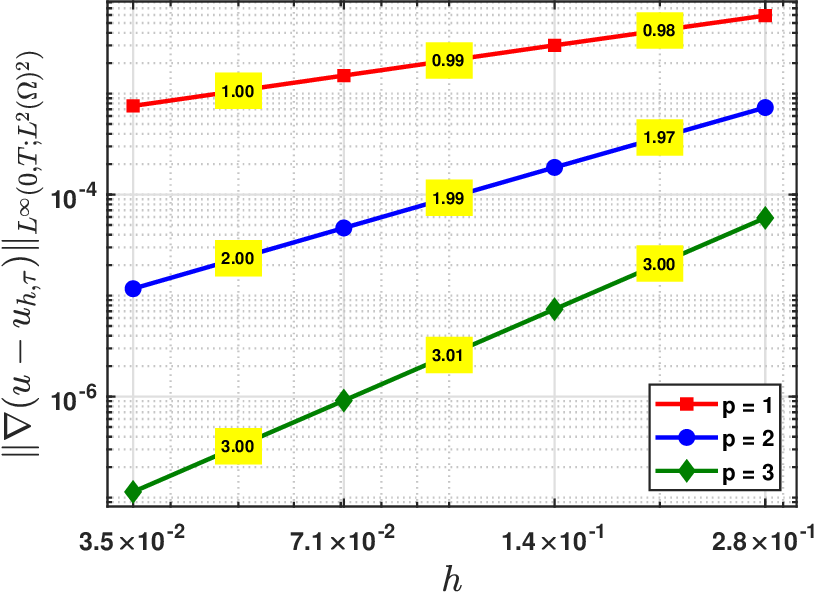}
\caption{$h$-convergence (in \emph{log-log} scale) of the errors in~\eqref{eq:errors-numerical-exp} for the problem in Section~\ref{sect:h-convergence} with exact solution~$u$ in~\eqref{eq:smooth-sol} with~$A = 10^{-2}$, $\omega = \pi/3$, and~$\ell = \pi$. The numbers in the yellow rectangles are the empirical convergence rates.
\label{fig:h-convergence}}
\end{figure}

\subsection{\texorpdfstring{$\tau$}{tau}-convergence\label{sect:tau-convergence}}
We now assess the~$\tau$-convergence of the scheme. To do so, we consider the problem from the previous section with exact solution given by~$u$ in~\eqref{eq:smooth-sol} with~$A = 10^{-2}$, $\omega = 9\pi/2$, and~$\ell 
= \pi$. 
For this value of~$\omega$, the solution~$u$ oscillates faster in time, thus increasing the influence of the temporal error.

We use approximations in space of degree~$p = 6$, and a fixed coarse spatial mesh with~$h \approx 2.82 \times 10^{-1}$. In Figure~\ref{fig:tau-convergence}, we show (in \emph{log-log} scale) the results obtained for a set of uniform partitions of the time interval with~$\tau = 0.5 \times 2^{-i}$, $i = 1, \ldots, 4$, and approximations of degree~$q = 2, 3, 4$. Optimal convergence rates of order~$\mathcal{O}(\tau^{q})$ and~$\mathcal{O}(\tau^{q + 1})$ are observed, respectively, for the errors in~\eqref{eq:errors-numerical-exp}.
\begin{figure}[!htb]
\centering
\includegraphics[width = 2.8in]{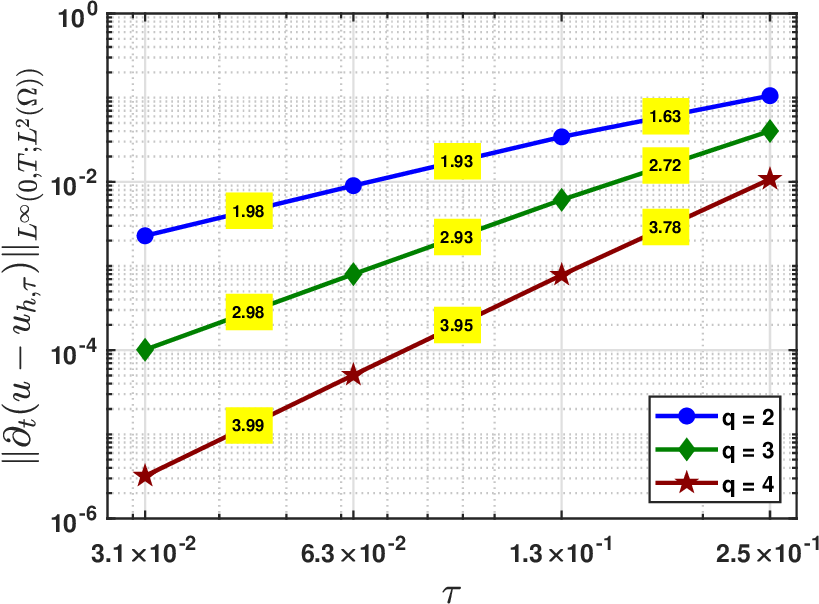}
\hspace{0.3in}
\includegraphics[width = 2.8in]{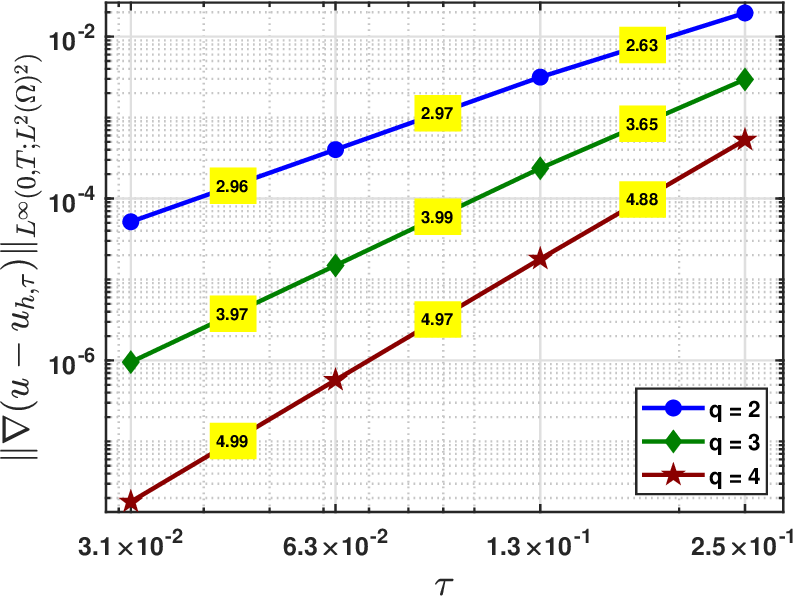}
\caption{$\tau$-convergence (in \emph{log-log} scale) of the errors in~\eqref{eq:errors-numerical-exp} for the problem in Section~\ref{sect:tau-convergence} with exact solution~$u$ in~\eqref{eq:smooth-sol} with~$A = 10^{-2}$, $\omega = 9\pi/2$, and~$\ell = \pi$.
\label{fig:tau-convergence}}
\end{figure}


\subsection{Assessment of the constraints on space--time meshes}
We now assess numerically the necessity of assumption~$h^{\ell + 1 - \frac{d}{2}} \lesssim \tau \lesssim h^{\frac{d}{2(m-1)}}$ in Theorem~\ref{thm: h tau conv nonlinear}, which, for~$d = 2$, $\ell = p$, and~$m = q$, simplifies to~$h^p \lesssim \tau \lesssim h^{\frac{1}{q-1}}$. To this end, we study both constraints~$h^p \lesssim \tau$ and~$\tau \lesssim h^{\frac{1}{q-1}}$ separately.

We first consider the test case in Section~\ref{sect:h-convergence}, using a final time~$T = 100$, time step~$\tau  = 20$, and degree of approximation in time~$q = 2$. In Figure~\ref{fig:h-convergence-CFL}, we present the results obtained for various spatial meshes and degrees of approximation in space~$p = 1,\, 2,\, 3$, where we observe that the errors remain bounded even when  the constraint~$\tau \lesssim h$ is not satisfied.

Similarly, we consider the test case in Section~\ref{sect:tau-convergence}, with final time~$T = 1$, a coarse spatial mesh, and degree of approximation in space~$p = 1$. In Figure~\ref{fig:tau-convergence-CFL}, we show the results obtained for a set of time partitions and degrees of approximation in time~$q = 2,\, 3,\, 4$, which show that the errors remain bounded even when the constraint~$h \lesssim \tau$ is not satisfied. 
\begin{figure}[!htb]
\centering
\includegraphics[width = 2.8in]{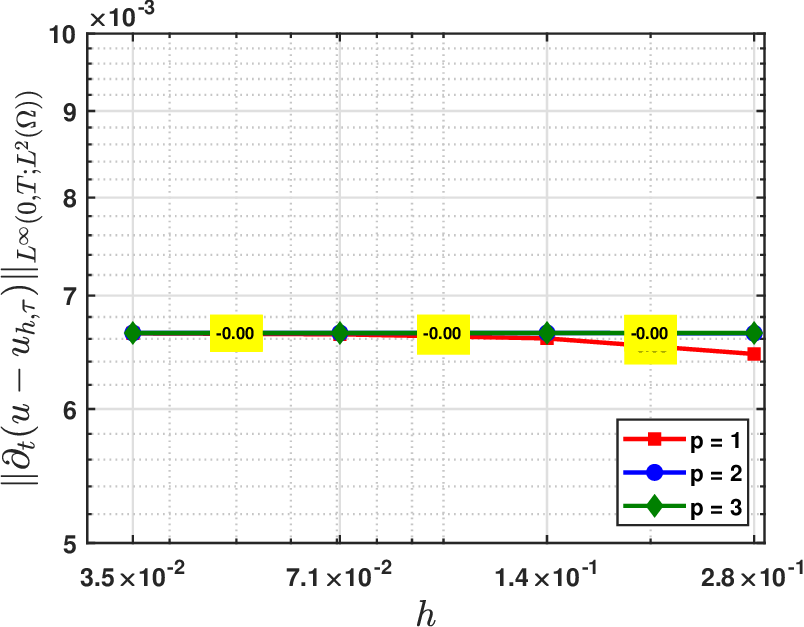}
\hspace{0.3in}
\includegraphics[width = 2.8in]{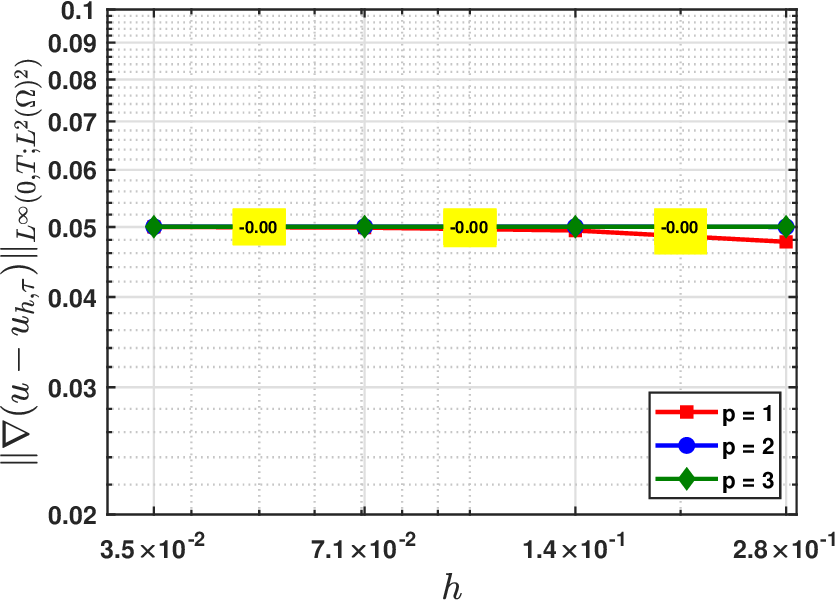}
\caption{
{Behavior (in \emph{log-log} scale) of the errors in~\eqref{eq:errors-numerical-exp} for the problem in Section~\ref{sect:h-convergence} with final time~$T = 100$. The results correspond to~$q = 2$ and $\tau = 20$.}
\label{fig:h-convergence-CFL}}
\end{figure}

\begin{figure}[!htb]
\centering
\includegraphics[width = 2.8in]{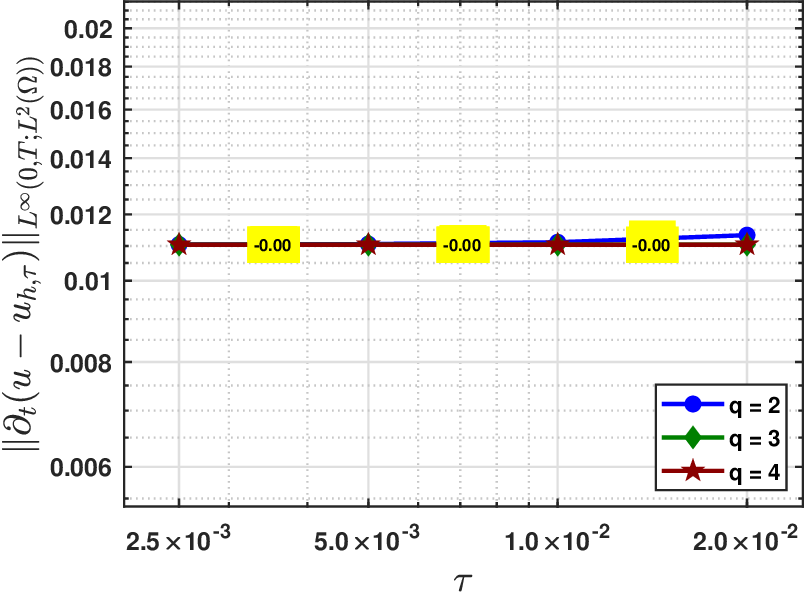}
\hspace{0.3in}
\includegraphics[width = 2.8in]{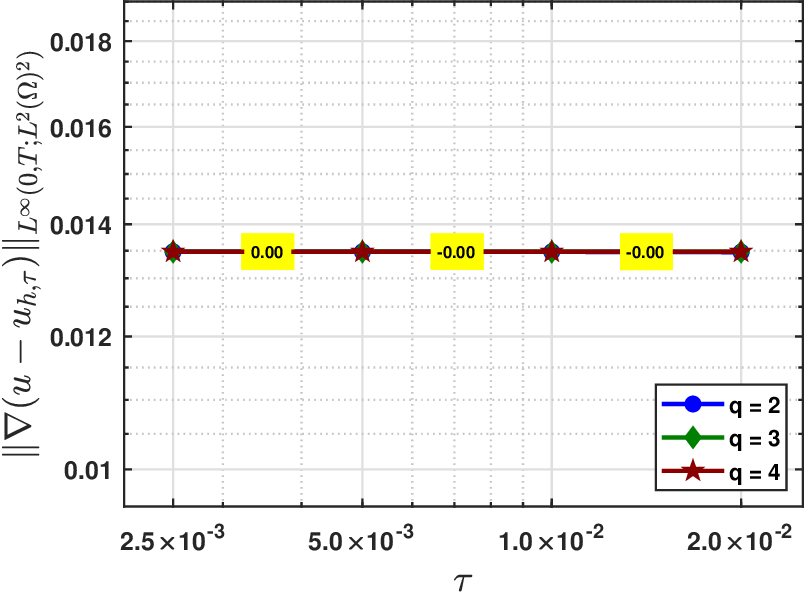}
\caption{
{Behavior (in \emph{log-log} scale) of the errors in~\eqref{eq:errors-numerical-exp} for the problem in Section~\ref{sect:tau-convergence} with final time~$T = 1$.
The results correspond to~$p = 1$ and the coarsest spatial mesh. }
\label{fig:tau-convergence-CFL}}
\end{figure}

{These numerical results suggest that the present theoretical constraint on the space--time meshes may be too restrictive, at least for solution fields of the type in~\eqref{eq:smooth-sol} that do not exhibit nonlinear effects, such as steepening of the wavefront.}

\subsection{\texorpdfstring{$(p, q)$}{(p,q)}-convergence\label{sect:p-convergence}}
Taking advantage of the simultaneous high-order convergence in space and time provided by the proposed method, we numerically assess the~$(p, q)$-version of the method. More precisely, given a fixed space--time mesh with~$h = \sqrt{2} \tau \approx 2.82 \times 10^{-1}$, we increase the approximation degrees in space and time; for simplicity, we set~$q = p$.
In Figure~\ref{fig:p-convergence}, we show (in \emph{semilogy} scale) the errors obtained using approximations of degree~$p = 2, \ldots, 9$ for the problem in Section~\ref{sect:tau-convergence}, and observe exponential convergence of order~$\mathcal{O}(e^{-b \sqrt[3]{N_{\mathrm{DoFs}}}})$ for the errors in~\eqref{eq:errors-numerical-exp}, where~$N_{\mathrm{DoFs}}$ denotes the total number of degrees of freedom. Such an exponential convergence was also observed in~\cite[\S4.1]{Dong_Mascotto_Wang:2024} for the DG--CG scheme applied to the linear wave equation, and it follows the expected decay for analytical solutions with~$p$-FEM approximations~\cite{Schwab:1998}.
\begin{figure}[!htb]
\centering
\includegraphics[width = 2.8in]{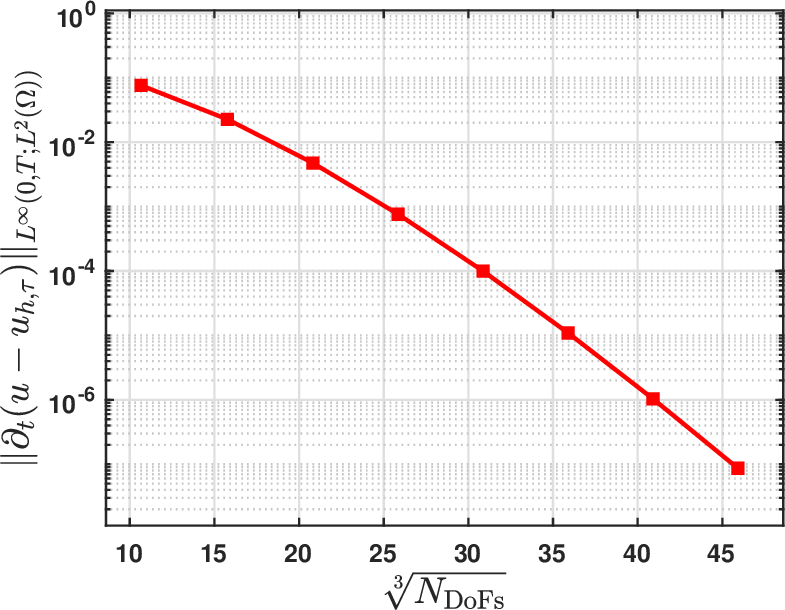}
\hspace{0.3in}
\includegraphics[width = 2.8in]{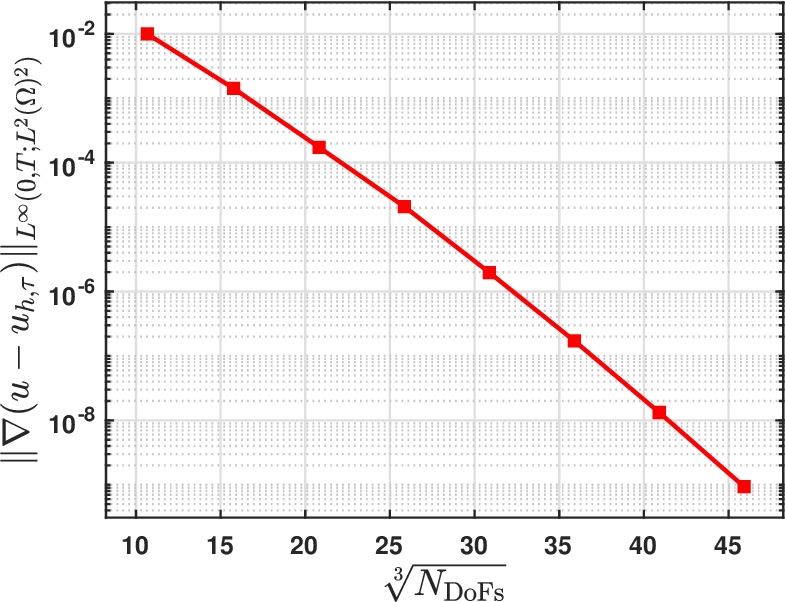}
\caption{$p$-convergence (in~\emph{semilogy} scale) of the errors in~\eqref{eq:errors-numerical-exp} for the problem in Section~\ref{sect:p-convergence} with exact solution~$u$ in~\eqref{eq:smooth-sol} with~$A = 10^{-2}$, $\omega = 9\pi/2$, and~$\ell = \pi$.
\label{fig:p-convergence}}
\end{figure}

\subsection{\texorpdfstring{$\delta$}{delta}-convergence}
We now validate the~$\delta$-convergence of the discrete solution to the inviscid discrete solution proven in Theorem~\ref{thm: delta conv nonlinear}.
To do so, we just have to compute the difference~$\uht^{(\delta)} - \uht^{(0)}$, so it is not necessary to know the exact solution of the problem. 
\subsubsection{Homogeneous source term\label{sect:zero-f}}
We first consider the Westervelt equation~\eqref{eq:Westervelt-IBVP} on the space--time domain~$\QT = (0, 1)^2 \times (0, 1)$ with~$k = 0.3$, $c = 1$, homogeneous Dirichlet boundary conditions and source term, and initial data given by
\begin{equation}
\label{eq:initial-data-zero-f}
u_0(x, y) = 10^{-2} \sin(\pi x) \sin(\pi y) \quad \text{ and } \quad u_1(x, y) = \sin(\pi x) \sin(\pi y).
\end{equation}

We set a space--time mesh with~$h = \sqrt{2} \tau \approx 1.41 \times 10^{-1}$, and use approximations in time of degree~$q = 4$. In Figure~\ref{fig:delta-convergence-f=0}, we show (in \emph{log-log} scale) the decay of the following errors for~$\delta = 10^{-2i}$, $i = 1, \ldots, 5$, and~$p = 1, 2, 3$:
\begin{equation}
\label{eq:delta-errors}
\Norm{\dpt (\uht^{(\delta)} - \uht^{(0)}) }{L^{\infty}(0, T; \Ltwo)} \quad \text{ and } \quad \Norm{\nabla (\uht^{(\delta)} - \uht^{(0)}) }{L^{\infty}(0, T; \Ltwo^2)},
\end{equation}
where we observe the expected linear convergence in~$\delta$ predicted by Theorem~\ref{thm: delta conv nonlinear}.
\begin{figure}[!htb]
    \centering
    \includegraphics[width=2.8in]{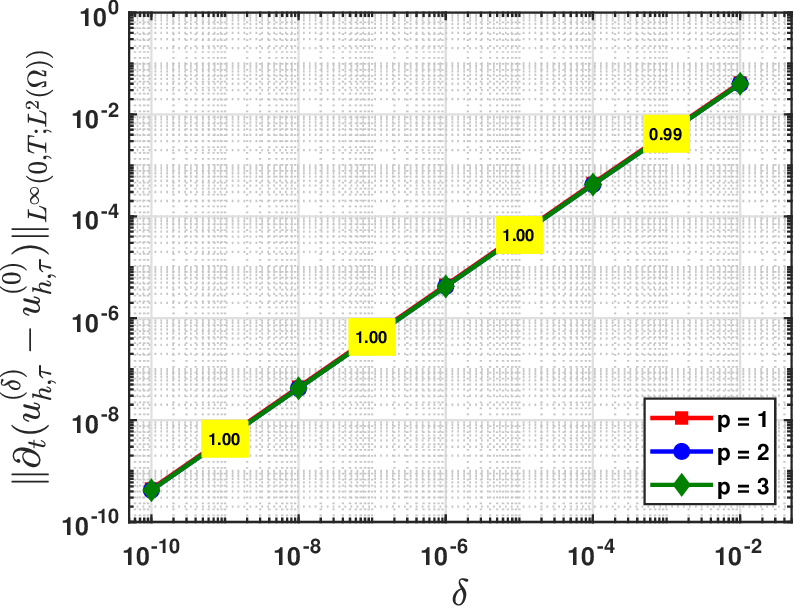}
    \hspace{0.3in}
    \includegraphics[width=2.8in]{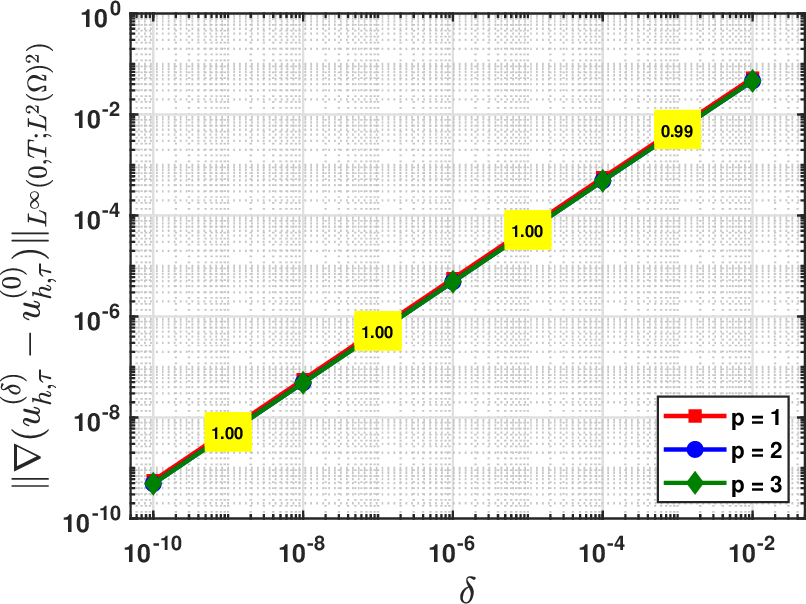}
    \caption{$\delta$-convergence (in~\emph{log-log} scale) for the problem in Section~\ref{sect:zero-f} with~$f = 0$ and initial data in~\eqref{eq:initial-data-zero-f}.}
    \label{fig:delta-convergence-f=0}
\end{figure}

\subsubsection{Nonhomogeneous source term\label{sect:non-zero-f}}
Finally, we consider the Westervelt equation~\eqref{eq:Westervelt-IBVP} on the space--time domain~$\QT = (0, 1) \times (0, 2 \times 10^{-4})$, with~$k = -10$, $c = 2000$, homogeneous Dirichlet boundary conditions and initial data, and a source term given by
\begin{equation}
\label{eq:source-term-delta}
f(x, y, t) = \frac{a}{\sqrt{\sigma}} \exp(-\alpha t) \exp\Big(- \frac{(x - 1/2)^2 + (y - 1/2)^2}{2 \sigma^2}\Big),
\end{equation}
with~$a = 400$, $\alpha = 5 \times 10^4$, and~$\sigma = 3 \times 10^{-2}$. 

We use the same spatial mesh and approximation degrees as in the previous numerical experiment, and set the time step as~$\tau = 10^{-5}$.
In Figure~\ref{fig:snapshots-f-not-zero}, we show the discrete solution corresponding to~$\delta = 6 \times 10^{-9}$ obtained at~$t = 10^{-4}$ and~$t = 2 \times 10^{-4}$, where the formation of a wavefront is observed.
\begin{figure}[!htb]
    \centering
    \includegraphics[width=2.8in]{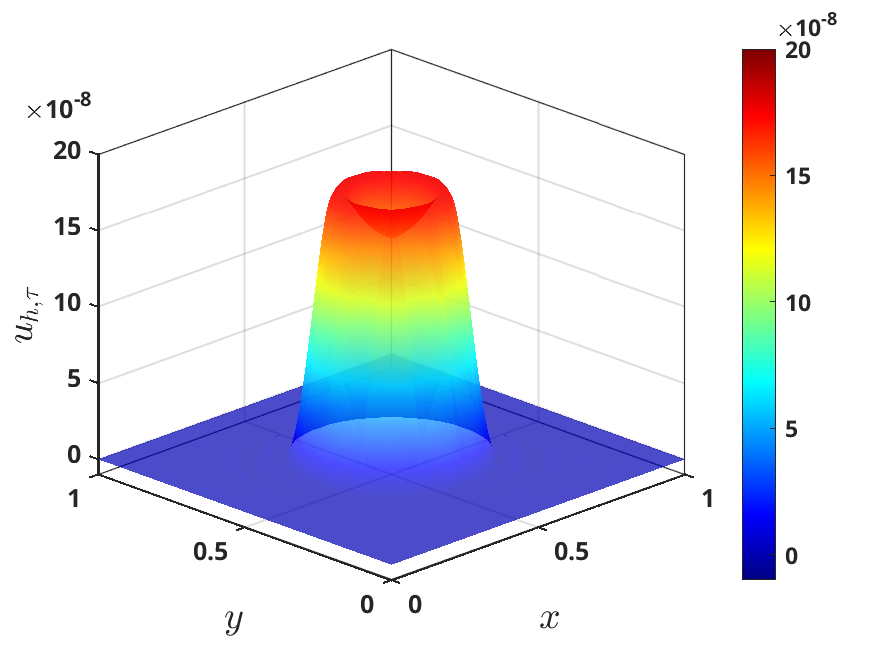}
    \hspace{0.2in}
    \includegraphics[width=2.8in]{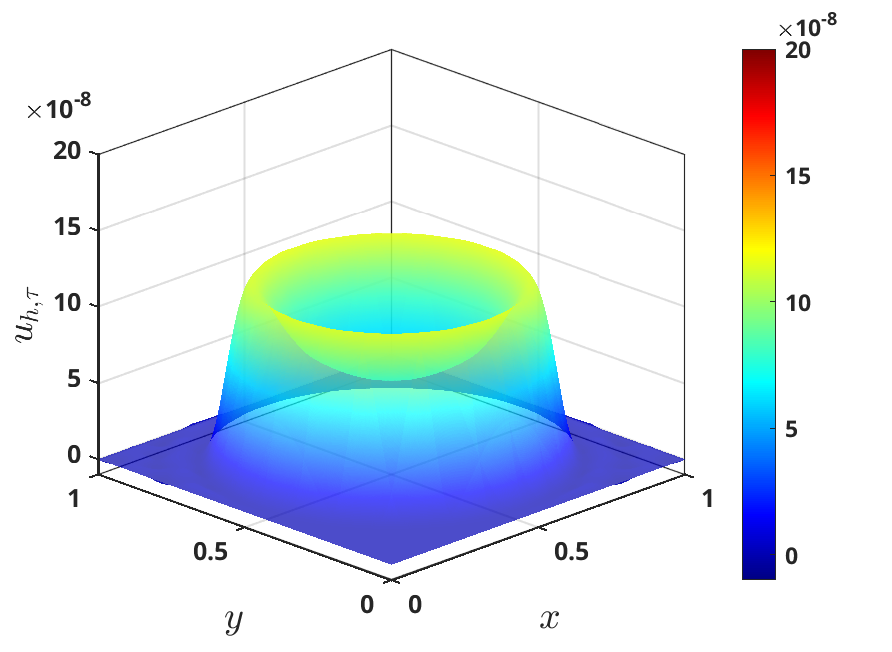}
    \caption{Snapshopts at~$t = 10^{-4}$ and~$t = 2 \times 10^{-4}$ of the discrete solution for the problem in Section~\ref{sect:non-zero-f} with~$\delta = 6 \times 10^{-9}$, zero initial data, and source term~$f$ in~\eqref{eq:source-term-delta}.\label{fig:snapshots-f-not-zero}}
\end{figure}

Since~$f$ does not depend on~$\delta$, convergence to the inviscid discrete solution when~$\delta \to 0^+$ is expected. In Figure~\ref{fig:delta-convergence-f-not-0}, we show the errors in~\eqref{eq:delta-errors} for~$\delta = 10^{-2i}$, $i = 1, \ldots, 5$, and~$p = 1, 2, 3$, and observe convergence rates of order~$\mathcal{O}(\delta)$, as expected.
\begin{figure}[!htb]
    \includegraphics[width=2.8in]{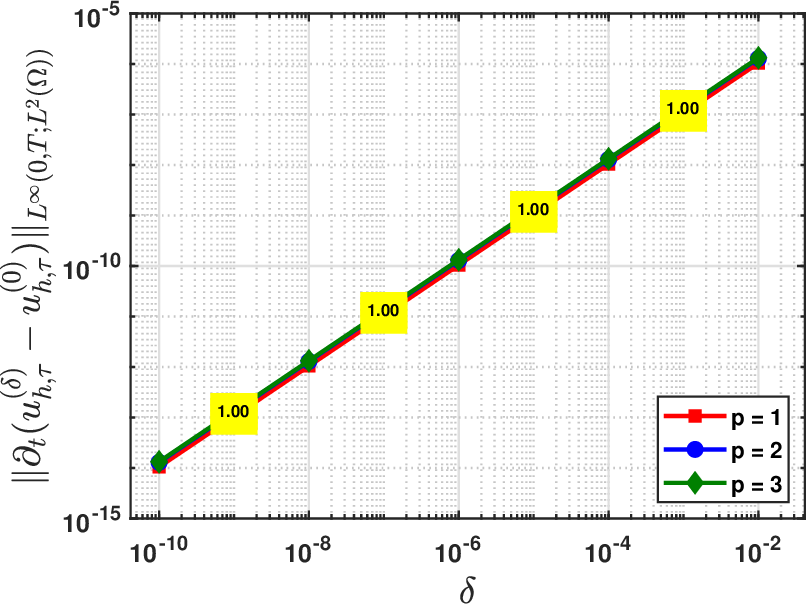}
    \hspace{0.3in}
    \includegraphics[width=2.8in]{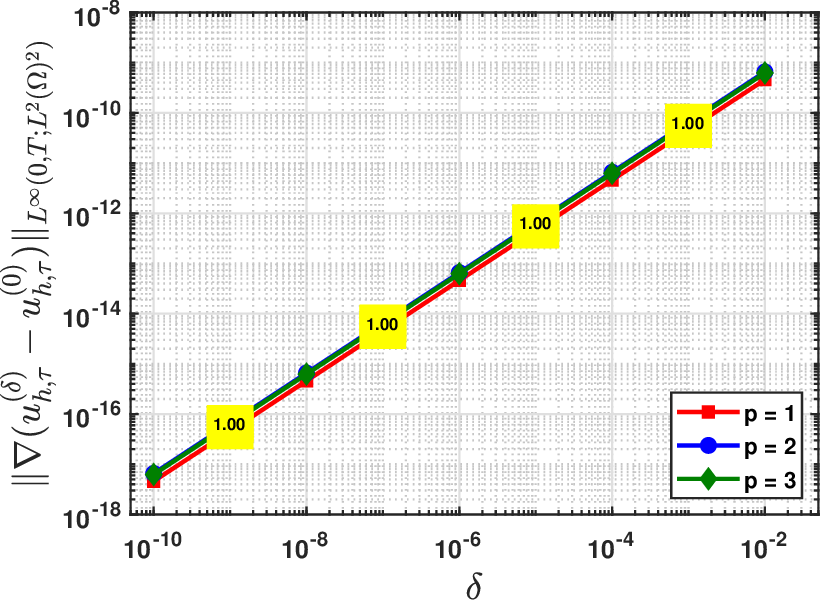}
    \caption{$\delta$-convergence (in~\emph{log-log} scale) for the nonlinear problem in Section~\ref{sect:non-zero-f} with zero initial data, and source term~$f$ in~\eqref{eq:source-term-delta}.}
    \label{fig:delta-convergence-f-not-0}
\end{figure}
\section{Conclusions\label{sec:conclusions}}
In this work, we have extended the theoretical framework of the DG--CG finite element methodology {introduced in~\cite{Walkington:2014}} to the classical model of nonlinear acoustics -- the quasilinear Westervelt equation. 
The principal challenges in the analysis stemmed from the fact that standard Galerkin testing techniques for wave problems fail to yield a bound on the discrete energy at all times, paired with the need to handle the particular type of nonlinearity present in the model that involves the second-order time derivative. 
To overcome this, we employed a testing strategy for a linearized problem based on using an auxiliary weight function, which was then combined with a fixed-point argument. An added facet of the work is that the \emph{a priori} error analysis guarantees that the estimates are robust with respect to the vanishing dissipation parameter~{$\delta$}. Furthermore, the asymptotic behavior of the exact problem is preserved in the singular limit. 

This work can serve as a starting point for developing a space--time finite element methodology for different problems arising in the context of nonlinear acoustics. Extending the framework developed here to refined models of nonlinear acoustics, such as the Kuznetsov or the Blackstock equation, is of interest. Additionally, multiphysics interactions of nonlinear acoustic waves lie at the core of many developing applications of ultrasound, and their simulation could strongly benefit from the findings made here.

\section*{Acknowledgments}
The first author acknowledges support from the
Italian Ministry of University and Research through the project PRIN2020 ``Advanced polyhedral discretizations of heterogeneous PDEs for multiphysics problems", and from the INdAM-GNCS through the
project CUP E53C23001670001.


\appendix 

\section{{Tools used in the stability analysis}\label{appendix:Stab}}
\subsection{Orthogonal projection and trace inequalities}
\begin{lemma}[Stability of~$\Pi_r^t$, {see~\cite[Thm.~18.16]{Ern_Guermond-book-I}}] 
\label{lemma:stability-L2-proj-Lp}
Let~$r \in \IN$ and~$s \in [1, \infty]$. There exists a positive constant~$\CS$ independent of~$\tau$ such that, for any~$\In \in \Tt$, it holds
\begin{equation*}
\Norm{\Pi_r^t \psi}{L^s(\In)} \le \CS \Norm{\psi}{L^s(\In)} \qquad \forall \psi \in L^s(\In).
\end{equation*}
\end{lemma}

\begin{lemma}[Estimates for~$\Pi_r^t$, {see~\cite[Thm.~18.16]{Ern_Guermond-book-I}}]
\label{lemma:polynomial-approx-time}
Let~$r \in \IN$. There exists a positive constant~$\Cp$ independent of~$\tau$ such that, for any~$\In \in \Tt$ and any sufficiently smooth function~$\psi$ defined on~$\In$, it holds
\begin{equation*}
\Norm{(\Id - \Pi_r^t) \psi}{W^m_p(\In)} \le \Cp \tau_n^{s - m} \SemiNorm{\psi}{W^s_p(\In)} \qquad s, m \in \IN, \ m \le s \le r + 1, \ p \in [1, \infty].
\end{equation*}
\end{lemma}

\begin{lemma}[Trace inequalities, {see~\cite[Lemma 1.46]{DiPietro_Ern-book}}]
\label{lemma:trace-inequality}
Let~$\In \in \Tt$. There exists a positive constant~$\Ctr$ independent of~$\tau$ such that
\begin{equation*}
|\psi(\tnmo)|^2 + |\psi(\tn)|^2 \le \Ctr \big(\tau_n^{-1}\Norm{\psi}{L^2(\In)}^2 + \tau_n \Norm{\psi'}{L^2(\In)}^2 \big) \qquad \forall \psi \in H^1(\In).
\end{equation*}

Moreover, there exists a positive constant~$\Ctr^{\star}$ independent of~$\tau$ such that
\begin{equation}
\label{eq:polynomial-trace-inequality}
|p_q(\tnmo)|^2 + |p_q(\tn)|^2 \le \Ctr^{\star} \tau_n^{-1}\Norm{p_q}{L^2(\In)}^2 \qquad \forall p_q \in \Pp{q}{\In}.
\end{equation}
\end{lemma}

\subsection{Polynomial inverse estimates}
\begin{lemma}[Inverse estimates, {see~\cite[Lemma 4.5.3]{Brenner-Scott:book}}]
\label{lemma:inverse-estimate}
\begin{subequations}
There exists a positive constant~$\Cinv$ independent of~$h$ and~$\tau$ such that 
\begin{alignat}{3}
\label{eq:inverse-estimate-space}
\Norm{\nabla \wh}{L^2(\QT)^d} & \le \Cinv \hmin^{-1} \Norm{\wh}{L^2(\QT)} & & \qquad \forall \wh \in L^2(0, T; \Vhp), \\
\label{eq:inverse-estimate-space-L-infty}
\Norm{\wh}{L^{\infty}(\QT)^d} & \le \Cinv \hmin^{-\frac{d}{2}} \Norm{\wh}{L^2(\QT)} & & { \qquad \forall \wh \in L^2(0, T; \Vhp), }
\end{alignat}
and, for~$n = 1, \ldots, N$, 
\begin{alignat}{3}
\label{eq:inverse-estimate-time}
\Norm{\dpt \wt}{L^r(\Qn)} & \le \Cinv \tau_n^{-1} \Norm{\wt}{L^r(\Qn)} & & \qquad \forall \wt \in \Pp{q}{\In; L^r(\Omega)}, \, r \in [1, \infty],\\
{\Norm{\dpt \wt}{L^2(\Qn)}} & \le \Cinv \tau_n^{-\frac12} \Norm{\wt}{L^{\infty}(\In; L^2(\Omega))} & & \qquad \forall \wt \in \Pp{q}{\In; L^{\infty}(\Omega)}. \label{eq:inverse-estimate-time second}
\end{alignat}
\end{subequations}
\end{lemma}
\begin{lemma} 
\label{lemma:Linfty-L2}
Let~$\In \in \Tt$. The following bound holds: 
\begin{equation*}
\Norm{\wt}{L^{\infty}(\In; L^2(\Omega))} \le (1 + \Cinv) \tau_n^{-\frac12} \Norm{w_{\tau}}{L^2(\Qn)} \qquad \forall w_{\tau} \in \Pp{q}{\In; L^2(\Omega)},
\end{equation*}
where~$\Cinv$ is the constant in the polynomial inverse estimate~\eqref{eq:inverse-estimate-time}.
\end{lemma}
\begin{proof}
The result follows by combining the inequality~$\Norm{w}{L^{\infty}(a, b)} \le |b - a|^{-1/2} \Norm{w}{L^2(a, b)} + |b - a|^{1/2} \Norm{w'}{L^2(a, b)}$ (see~\cite[Eq.~(1.9)]{Ern_Guermond-book-I}) with the polynomial inverse estimate~\eqref{eq:inverse-estimate-time}.
\end{proof}

\section{{Tools used in the convergence analysis}}

\subsection{Properties of the temporal projection}

Next lemmas concern the stability and approximation properties of the projection~$\Pt$ in Definition~\ref{DEF:Pt}; see also~\cite[Lemma 2.4 and Prop. 2.5]{Dong_Mascotto_Wang:2024}, where~$q$-explicit stability and approximation properties of~$\Pt$ are derived.

\begin{lemma}[Stability of~$\Pt$, {see~\cite[Lemma 5.2]{Walkington:2014}}]
\label{lemma:stab-Pt}
Let~$q \in \IN$ with~$q \geq 2$. 
For any~$s \in [1, \infty)$, there exists a positive constant depending only on~$q$ such that, for all~$v \in C^1(0, T)$, the following bounds hold:
\begin{alignat*}{3}
\Norm{\dpt \Pt v}{L^s(0, T)} & \le \CstabPt T^{1/s}\Norm{\dpt v}{C^0([0, T])}, \\
\Norm{\Pt v}{L^s(0, T)} & \le \CstabPt T^{1/s} \big( \Norm{v}{C^0([0, T])} + \tau \Norm{\dpt u}{C^0([0, T])} \big).
\end{alignat*}
For~$s = \infty$, the above bounds hold with~$T^{1/s}$ replaced by~$1$.
\end{lemma}

\begin{lemma}[Estimates for~$\Pt$, {see~\cite[Lemma 5.2]{Walkington:2014}}]
\label{lemma:estimates-Pt}
Let~$q \in \IN$ with~$q \geq 2$.
For any~$s \in [1, \infty]$, the following estimates hold for all~$v \in W_s^{m + 1}(0, T)$ with~$1 \le m \le q$: {there exists a positive constant~$\CapproxPt$ independent of~$\tau$ such that}
\begin{alignat*}{3}
\Norm{(\Id - \Pt) v}{L^s(0, T)} & \le \CapproxPt \tau^{m + 1} \SemiNorm{v}{W_s^{m + 1}(0, T)}, \\
\Norm{\dpt (\Id - \Pt) v}{L^s(0, T)} & \le \CapproxPt \tau^m \SemiNorm{v}{W_s^{m + 1}(0, T)}.
\end{alignat*}
\end{lemma}

For any Banach space~$(Z, \Norm{\cdot}{Z})$ with~$Z \subseteq L^1(\Omega)$, the definition of~$\Pt$ can be extended to functions in~$C^1(0, T; Z)$ by requiring that~\eqref{eq:proj-Pt-def} holds almost everywhere in~$\Omega$.

\subsection{Properties of the Ritz projection and the Lagrange interpolant}

We next recall certain approximation properties of the Ritz projection~$\Rh$.
\begin{lemma}[Estimates for~$\Rh$, {see~\cite[Thms 5.4.4 and 5.4.8]{Brenner-Scott:book}}]
\label{lemma:estimates-Rh}
Let~$p \in \IN$ with~$p \geq 1$, and let~$\Omega$ satisfy Assumption~\ref{asm:elliptic-regularity}. 
Then, there exists a positive constant~$\CapproxRh$ independent of~$h$ such that
\begin{subequations}
\begin{alignat}{3}
\Norm{(\Id - \Rh) z}{L^2(\Omega)} & \le \CapproxRh h^{\ell + 1} \SemiNorm{z}{H^{\ell + 1}(\Omega)} & & \qquad \forall z \in H^{\ell + 1}(\Omega) \cap H_0^1(\Omega), \ 0 \le \ell \le p,  \label{est: Rh L2} \\
\Norm{\nabla(\Id - \Rh)z}{L^2(\Omega)^d} & \le \CapproxRh h^{\ell} \SemiNorm{z}{H^{\ell + 1}(\Omega)} & & \qquad \forall z \in H^{\ell + 1}(\Omega) \cap H_0^1(\Omega), \ 0 \le \ell \le p.
\end{alignat}
\end{subequations}
\end{lemma}

We denote the Lagrange interpolant operator
by~$\Ih{} : C^0(\overline{\Omega}) \rightarrow \Vhp$.
The following stability and approximation properties follow from~\cite[Thm. 4.4.20 in Ch.~4]{Brenner-Scott:book}.

\begin{lemma}[Stability of~$\Ih{}$]
\label{lemma:stab-Ih}
Let~$p \in \IN$ with~$p \geq 1$. Then, there exists a positive constant~$\CstabI$ independent of~$h$ such that
\begin{equation*}
\Norm{\Ih v}{\Linf} \le \CstabI \Norm{v}{\Linf}   \quad \forall v \in C^0(\overline{\Omega}).
\end{equation*}
\end{lemma}

\begin{lemma}[Estimates for~$\Ih{}$]
\label{lemma:approx-Ih}
Let~$p \in \IN$ with~$p \geq 1$. Then, there exists a positive constant~$\CapproxI$ independent of~$h$ such that
\begin{alignat*}{3}
\Norm{\Ih v - v}{L^2(\Omega)} \le  \CapproxI h^{\ell + 1} \SemiNorm{v}{H^{\ell + 1}(\Omega)} \quad \forall v \in H^{\ell + 1}(\Omega), \ {(d-2)/2 < \ell} \le p.
\end{alignat*}
\end{lemma}
We finish this section with some useful stability bounds for the Ritz projection.
\begin{lemma}[Stability of~$\Rh$]
\label{lemma:stab-Rh}
Let~$p \in \IN$ 
with~$p \geq 1$. There exists a positive constant~$\CstabRh$ independent of~$h$ such that
\begin{alignat}{3}
\label{eq:Poincare}
\Norm{\Rh z}{L^2(\Omega)} & \le \CstabRh\Norm{\nabla z}{L^2(\Omega)^d} & & \qquad \forall z \in H_0^1(\Omega), \\
\label{eq:bound-Rh-Linf}
\Norm{\Rh z}{\Linf} & \le \CstabRh \Norm{z}{H^2(\Omega)} & & \qquad \forall z \in H^2(\Omega) \cap H_0^1(\Omega).
\end{alignat}
\end{lemma}
\begin{proof}
Bound~\eqref{eq:Poincare} follows from the standard Poincar\'e inequality (see, e.g., \cite[Prop. 5.3.2]{Brenner-Scott:book}).

As for bound~\eqref{eq:bound-Rh-Linf}, we use the triangle inequality, the inverse estimate~\eqref{eq:inverse-estimate-space-L-infty}, and Lemmas~\ref{lemma:stab-Ih} and~\ref{lemma:approx-Ih} to obtain
\begin{alignat*}{3}
\Norm{\Rh z}{\Linf} & \le \Norm{\Rh z - \Ih z}{\Linf} + \Norm{\Ih z}{\Linf} \\
& \lesssim h^{-d/2}\Norm{\Rh z - \Ih z}{\Ltwo} + \Norm{z}{\Linf} \\
& \lesssim h^{2 - d/2} \SemiNorm{z}{H^2(\Omega)} + \Norm{z}{\Linf},
\end{alignat*}
which, combined with the continuous Sobolev embedding~$H^2(\Omega) \hookrightarrow \Linf$, completes the proof.
\end{proof}

\end{document}